\pdfoutput=1
\documentclass[oneside,english]{amsart}
\usepackage{ae,aecompl}
\usepackage[T1]{fontenc}
\usepackage[utf8]{inputenc}
\usepackage{geometry}
\usepackage{xcolor}
\usepackage{babel}
\usepackage{amstext}
\usepackage{amsthm}
\usepackage{amssymb}
\usepackage{undertilde}
\usepackage{setspace}
\usepackage[authoryear]{natbib}
\usepackage[pdfusetitle,
 bookmarks=true,bookmarksnumbered=false,bookmarksopen=false,
 breaklinks=false,pdfborder={0 0 1},backref=false,colorlinks=false]
 {hyperref}

\makeatletter

\providecommand{\tabularnewline}{\\}

\numberwithin{equation}{section}
\numberwithin{figure}{section}
\theoremstyle{plain}
\newtheorem{thm}{\protect\theoremname}
\theoremstyle{definition}
\newtheorem{defn}[thm]{\protect\definitionname}
\theoremstyle{plain}
\newtheorem{fact}[thm]{\protect\factname}
\theoremstyle{plain}
\newtheorem{lem}[thm]{\protect\lemmaname}
\theoremstyle{plain}
\newtheorem{cor}[thm]{\protect\corollaryname}
\theoremstyle{plain}
\newtheorem{prop}[thm]{\protect\propositionname}
\theoremstyle{remark}
\newtheorem*{rem*}{\protect\remarkname}

\usepackage[noload]{qtree}
\usepackage{amsfonts}
\usepackage{bussproofs}
\usepackage{pdflscape}

\AtBeginDocument{\hypersetup{hidelinks}}
\makeatother

\providecommand{\corollaryname}{Corollary}
\providecommand{\definitionname}{Definition}
\providecommand{\factname}{Fact}
\providecommand{\lemmaname}{Lemma}
\providecommand{\propositionname}{Proposition}
\providecommand{\remarkname}{Remark}
\providecommand{\theoremname}{Theorem}

\begin{document}
\title{Loose Inductive Set Theories}
\author{Elliot Glazer and Toby Meadows}
\begin{abstract}
We prove that $KP$, with Full Set Induction, is neither solid nor
tight. More generally, every consistent theory obtained from $KPC$
by adding a quantifier bounded set of sentences has a model with a
parameter-free definable class surjection from its natural numbers
onto its universe. Using such enumerations to define generic extensions
by relative interpretations, we show that theories with the stable
definable ground property are not tight. Ground definability over
$MOST$ and continuum coding over $ZFC^{\Pi_{3}}$ yield applications
to weak inductive set theories with large cardinals and forcing axioms.
Finally, solidity for quantifier bounded extensions of $MOST$ with
Full Set Induction implies a Ground Axiom formulated using a fixed
finite level of genericity; under the same solidity hypothesis, the
corresponding extension of full $ZFC$ proves the ordinary Ground
Axiom.
\end{abstract}

\thanks{We thank Ali Enayat and Farmer Schlutzenberg for helpful comments
and correspondence, especially concerning the results on theories
of bounded-complexity consequences.}

\maketitle

\newpage{}

\tableofcontents{}

\newpage{}

It's well known that no theory with an infinite model that is articulated
in first order logic can be categorical: compactness says there will
always be a bigger model. However, in recent years logicians have
studied a number of properties that a theory can possess which appear
to, in some sense, approximate categoricity \citep{EnayatLelykFOCat}.
As a rule of thumb, the kinds of theories that have categoricity proofs
in second order logic tend to satisfy these properties when they are
converted into first order theories using schematic techniques. For
example, the second-order versions of $PA$ and $ZFC$ are respectively
categorical and quasi-categorical,\footnote{See \citep{Shapiro} for proofs of these facts.}
and both $PA$ and $ZFC$ possess these new properties \citep{Visser2006,EnayatVisTheme}.

Our broad goal in this paper is to investigate set theories that only
use bounded versions of the Separation, Collection and Replacement
schemas while remaining fully inductive in the sense that they satisfy
the Full Set Induction schema; i.e., any definable nonempty class
has an $\in$-minimal element. The main upshot of this paper is that
these approximate categoricity properties fail for broad classes of
such theories. This places our results in the program of identifying
natural set and class theories that fail to be tight or solid. \citet{freire2020biinterpretation}
showed that Zermelo set theory and $ZFC$ without Powerset are not
tight. \citet{FreireWilliams2025} established corresponding failures
for fragments of Kelley-Morse set theory and second-order arithmetic.

This paper extends results of \citet{EnayatLelykFOCat}. In particular,
they show that for any $n\in\omega$, the $\Pi_{n}$-consequences
of $ZF$ are neither \emph{solid} nor \emph{tight}.\footnote{Definitions of these properties will be provided shortly.}
The results of this paper improve this in two directions. First, we
include Full Set Induction in our theories. Second, we do not limit
ourselves to sub-theories of $ZF$, but rather, show that an open
ended collection of theories extending a weak base theory are neither
solid nor tight. In particular, we will show that for a natural class
of these theories, solidity implies that the theory proves a version
of the Ground Axiom.

Aside from the main results, this paper makes use of a couple of strange
techniques that are worth discussing. First, we show that some models
of these bounded theories are capable of defining generics over their
posets and the generic extensions based on those generics.\footnote{The use of definable Cohen generics to obtain bi-interpretations in
\citet{FreireWilliams2025} is closely related to our approach here.} This runs counter to the orthodox pedagogy of set theory where one
learns that generic sets must exist outside their models and thus,
generic extensions are not within the purview of relative interpretation.
The key to this counterintuitive approach is the ability to find models
of these bounded theories that are able to see that they are countable,
in the sense that they can define a surjection from the natural numbers
onto their entire domain. For sub-theories of $ZF$, we can exploit
well known fine-structural results to obtain such enumerations in
models of $V=L$ using the natural well-ordering of $L$.\footnote{For example, the techniques used in the uniformization theorems for
$L$ are easily adapted to this context \citep{DevlinC}.} But for theories incompatible with $L$ this technique is not available.
As such, we prove a very general lemma showing that our bounded theories
are always consistent with the existence of a definable countable
enumeration of the universe. This is the most novel result of the
paper and, as such, it takes up the most space. 

The paper is divided up as follows. In the next section we give proper
definitions of our target properties and explain why they are interesting.
In Section \ref{sec:KP}, we show that $KP$ satisfies neither of
these properties. To achieve this, we develop the theory of definable
forcing over $L_{\omega_{1}^{CK}}$ and then generalize this further.
Aside from the fact that $KP$ is a theory of obvious foundational
interest, this setting provides a convenient location to explain the
underlying ideas of the paper in a setting that is likely to be familiar
to its readers. Section \ref{sec:Quantifier-bounded-theories} then
provides the framework for a substantial generalization of these results
to theories that are bounded extensions of a weak sub-theory of $ZFC$
that includes Full Set Induction. This gives us a theorem template
(Theorem \ref{thm:MainTemplate}) that allows us to show that when
such theories possess a crucial property (Definition \ref{def:StableDefGprop})
they can never be tight. The theorem template is then demonstrated
on the basis of a lemma that we prove in the following section. The
Definable Countable Enumeration Lemma (Lemma \ref{count-enum-thm})
is then proved in Section \ref{sec:DCET}. This is the most original
and unusual contribution of the paper. In order to set up for applications
of this template, we need tools from set-theoretic geology. Section
\ref{sec:GeoSmall} starts from Schlutzenberg’s ground-model definability
theorem over the weak theory $MOST$ \citep{Schlutzenberg2022LocalMantles}.\footnote{The theory $MOST$ is described in Section \ref{sec:GeoSmall}. It
originates with \citet{MathiasStMacST}.} We adapt it to use a smaller parameter and then code that parameter
into the continuum function, so that the ground can be recovered without
parameters. The small-parameter argument works over $MOST$; the coding
construction is carried out over $ZFC^{\Pi_{3}}$. Finally, in Section
\ref{sec:Applications}, we bring these results to bear on a variety
of natural examples.

\section{Background definitions}

\label{sec:Background-definitions}

Let us start by introducing the property of tightness.
\begin{defn}
\citep{EnayatVisTheme} Say that a theory $T$ is \emph{tight }if
any pair of bi-interpretable extensions of $T$ are the same theory
(in the sense that they have the same deductive closure).
\end{defn}

Informally, a theory is tight if it cannot be extended in distinct,
but equivalent ways. To get more of a feel for this, let's contrast
a non-tight theory with a tight one. For a non-tight theory, let $Empty$
be the empty theory of binary relations articulated in $\mathcal{L}_{Empty}=\{R\}$.
Now consider the extensions $S_{0}$ and $S_{1}$ of $Empty$ that
say that $R$ is, respectively, reflexive and irreflexive. $S_{0}$
and $S_{1}$ are clearly bi-interpretable, indeed, they are definitionally
equivalent. From a graph theoretic perspective, we add loops in one
direction and remove the loops in the other direction. But they are
not the same theory. We might say that $Empty$ is, unsurprisingly,
insufficiently \emph{specific} to be tight.

For a contrast, we observe that $PA$ is tight \citep{Visser2006}.
So $PA$ appears to be more specific than $Empty$. To make more sense
of this, let's try to pull on some intuitions. We know that $PA$
has an intended model, $\mathbb{N}=\langle\omega,0,s,+,\times\rangle$,
and we also know that $PA$ is incomplete, so $PA\subsetneq Th(\mathbb{N})$.
Now imagine, for a moment, that $PA$ was not tight and that this
was witnessed by another theory $S\supsetneq PA$ that was bi-interpretable
with $Th(\mathbb{N})$. This would be a strange situation. Since $S$
would be distinct from $Th(\mathbb{N})$, we'd have another way of
understanding $0,s,+,\times$ that was distinct from the standard
interpretation and yet still satisfied the axioms of $PA$. There
would, thus, be an arithmetic sentence true in one interpretation
and false in the other; and little, if any, good reason to choose
between them. Such a situation would attest to a different kind of
incompleteness in $PA$. Fortunately, we know this cannot occur. Indeed,
we don't need to go as far as the true theory of arithmetic. Any pair
of bi-interpretable extensions of $PA$, true or not, must determine
the same theory.

Solidity is a more technical notion that was uncovered by Enayat in
an effort to identify the model theoretic core of tightness. 
\begin{defn}
\citep{EnayatVisTheme} A theory $T$ is \emph{solid} if for all models
$\mathcal{M}_{0},\mathcal{N}$ and $\mathcal{M}_{1}$ of $T$, whenever
we have 
\[
\mathcal{M}_{0}\triangleright_{par}\mathcal{N}\triangleright_{par}\mathcal{M}_{1}
\]
and an isomorphism $f:\mathcal{M}_{0}\cong\mathcal{M}_{1}$ that is
(parametrically) definable over $\mathcal{M}_{0}$, then there is
another isomorphism $g:\mathcal{M}_{0}\cong\mathcal{N}$ that is (parametrically)
definable over $\mathcal{M}_{0}$.
\end{defn}

It is not difficult to show that a solid theory must be tight.\footnote{See Remark 8(b) in \citep{EnayatLelykFOCat}.}
Indeed, proving solidity is the main pathway to establishing the tightness
of a theory. However, the definition can also be given a more familiar
mathematical gloss. The definition above is nicely understood as taking
place in the class, $mod(T)$, of models of $T$. The antecedent clauses
of the definition require that $\mathcal{M}_{0}$ and $\mathcal{M}_{1}$
are ``visibly'' isomorphic to each other. Given this, we might as
well treat them as ``identical.'' The arrangement of $\mathcal{M}_{0},\mathcal{N}$
and $\mathcal{M}_{1}$ then tells us that there is a way of interpreting
a model $\mathcal{N}$ in $\mathcal{M}_{0}$ from which we can then
interpret $\mathcal{M}_{0}$ (up to isomorphism) again. We thus, have
a kind of structure preserving map, known as a retraction, from $\mathcal{M}_{0}$
to $\mathcal{N}$. It preserves structure in the sense that $\mathcal{M}_{0}$
can still be recovered from $\mathcal{N}$ (up to isomorphism). Solidity
then addresses whether such maps can be \emph{non-trivial}; i.e.,
is the output, $\mathcal{N}$, distinct from the input, $\mathcal{M}_{0}$
(up to isomorphism). If we can find an arrangement of models of $T$
where the output is distinct (up to isomorphism), $T$ is not solid.
Thus, intuitively speaking, solidity tells us that $mod(T)$ has no
non-trivial structure preserving self-interpretations.

\section{$KP$ is neither solid nor tight}

\label{sec:KP}

Let's move to our first batch of results. Let $KP$ be the Kripke-Platek
set theory using Full Set Induction.\footnote{We follow \citep{Barwise,MansfieldRDST} here. More recently, sources
like \citep{MathiasStMacST} have employed a version of $KP$ that
removes Full Set Induction and replaces it with $\Sigma_{1}$-Set
Induction (aka $\Pi_{1}$-Foundation). That approach has the pleasing
feature of making $KP$ finitely axiomatizable. } Our target theorem is the following.
\begin{thm}
$KP$ is neither solid nor tight.\label{thm:KPnotTight}
\end{thm}

Since solidity implies tightness, it would suffice for us to merely
show that $KP$ is not tight. However, given that the techniques used
below are a little unorthodox, we have opted to give a detailed proof
that $KP$ is not solid first. Then we adapt that proof to the more
general claim that $KP$ is not tight. We think this will be helpful
as the proof that $KP$ is not solid takes place in the familiar world
of the hyperarithmetic sets. By taking our time here and building
some intuitions, we believe this will make things much easier for
the rest of the paper. There are also techniques developed here that
will be useful for the investigation of solidity (independent of tightness).

\subsection{Solidity failure}

\label{subsec:Solidity-failure}

We start by sketching the idea of the proof. Recall that $L_{\omega_{1}^{CK}}$
is the smallest model of $KP$ in the sense that it is the intersection
of all $\omega$-models of that theory. To save a little clutter,
let's write $N$ instead of $L_{\omega_{1}^{CK}}$. We plan to show
that $N$ can define a sufficiently generic Cohen real $c$ and, using
that, a version of the corresponding generic extension, $N[c]$. Thus,
it can be seen that we have
\[
N\triangleright N[c]\triangleright L^{N[c]}.
\]
Moreover, we shall ensure that $N$ is able to define an isomorphism
$\eta:N\cong L^{N[c]}$, while, of course, $N\not\cong N[c]$. Thus,
we shall see that $KP$ is not solid. 

This ability to define a generic is a little unusual and warrants
some discussion about what that means and how it will be done. First,
observe that $c\subseteq\omega$ must be defined as a proper class
over $N$ rather than a set. Otherwise, we'd fall into the standard
contradiction where the complement of the generic filter associated
with $c$ is a dense set missed by the filter. Second, we note that
in order to ensure this, the complexity of the definition of $c$
must be greater than that of the underlying Separation Axiom. To see
this, note that in the presence of the full axiom of Separation, if
$c$ is definable then it can be separated as a sub\emph{set} of $\omega$.
Fortunately, we only have $\Delta_{1}$-Separation available in $N$.
Thirdly, in order to obtain a generic we wish to follow the standard
Baire-Category strategy. As such, we want the dense sets of $Add(\omega,1)$
to be definably countable so that we can run a version of the standard
inductive argument. This rules out theories that satisfy full Replacement.
Fortunately, we only have $\Sigma_{1}$-Collection in $N$. Fourth
and finally, note that in order for $N[c]$ to satisfy $KP$, and
in particular, $\Sigma_{1}$-Collection, an ordinary generic filter
meeting all dense subsets of $Add(\omega,1)$ that are elements of
$N$ will not suffice.\footnote{See Section 6 and the discussion around Lemma 6.5 and Proposition
6.6 in \citep{GREENBERG_MONIN_2017}; and, more specifically, Proposition
13 of \citep{KRUSCHEWSKI_SCHLUTZENBERG_2025}.} We shall thus introduce a stronger notion of genericity.

With the sketch complete, let us fix our convention for measuring
the complexity of formulae. We use prenex formulae: all quantifiers
are at the front, followed by a quantifier-free formula. We count
alternating blocks of quantifiers, bounded or unbounded, as far as
the last unbounded quantifier; any bounded quantifiers after it do
not contribute to the count. Thus $\forall x\exists y\in z\forall w\,\psi$,
with $\psi$ bounded, is $\Pi_{3}$. A formula with only bounded quantifiers
is $\Delta_{0}$. When obtaining a prenex formula, we first expand
bounded quantifiers with their usual restrictions as necessary and
then use ordinary logical equivalences. This respects empty bounds
and requires no Collection. The classes $\Sigma_{k}$ and $\Pi_{k}$
below use this convention, with harmless padding by unused quantifiers
allowed.

Let us say that a formula is $B\Sigma_{k}$ if it is a finite Boolean
combination of $\Sigma_{k}$ formulae. A subset of $N$ is $B\utilde{\Sigma}_{k}$
if it is definable by such a formula, possibly using parameters from
$N$. Since this paper is focused on interpretability and thus on
definability without parameters, whenever we say that an object is
definable, we mean definable without parameters; we shall always say
when parameters are required.
\begin{defn}
Let $\mathbb{P}$ be a poset. We say that $G\subseteq\mathbb{P}$
is $B\utilde{\Sigma}_{k}$\emph{-generic} over $N$ if $G$ is a filter
on $\mathbb{P}$ and for all $B\utilde{\Sigma}_{k}$ dense subsets
$D$ of $\mathbb{P}$, $G\cap D\neq\emptyset$. \label{def:BSigmaKGen}
\end{defn}

We adopt the common assumption that forcing posets have a greatest
condition, which we denote as $1$. We assume that $\Vdash$ is the
syntactic forcing relation used in \citep{KunenST2}. Note that we
cannot uniformly define the forcing relation for arbitrary formulae,
since this would contradict Tarski's theorem. For bounded formulae,
forcing is uniformly $\Delta_{1}$-definable over $KP$. For higher
formulae, we shall use a uniform finite bound. For $k\in\omega$,
let 
\[
\mathsf{b}(k)=2k+3.
\]
For each fixed $k$, forcing for $B\Sigma_{k}$ formulae is uniformly
$\Delta_{\mathsf{b}(k)}$-definable over $KP$. The definition is
uniform in the formula code, the tuple of names, the condition and
the forcing poset. We prove this bound below, after recalling the
density lemma. The bound is chosen for convenience; we do not claim
that it is optimal.

Using a suitably generic $G\subseteq\mathbb{P}$, we can then define
a generic extension of the ground model. However, given that our eventual
goal is to use a definable generic and interpret the generic extension,
we cannot use the familiar $Val$ function since that construction
would try (and fail) to collapse the generic extension in the ground
model. As such, we shall use something like an ultrapower construction
that interprets our generic extension as a pre-structure. Thus, we
have a structure where the identity relation symbol, $=$, is interpreted
as the congruence relation $\approx^{*}$. 
\begin{defn}
For $G\subseteq\mathbb{P}$, let $N^{\mathbb{P}}/G=\langle N^{*},\in^{*},\approx^{*}\rangle$
where: 
\begin{itemize}
\item $N^{*}=N^{\mathbb{P}}$;
\item $\dot{x}\in^{*}\dot{y}$ iff $\exists p\in G$ $p\Vdash\dot{x}\in\dot{y}$,
for all $\dot{x},\dot{y}\in N^{\mathbb{P}}$; and
\item $\dot{x}\approx^{*}\dot{y}$ iff $\exists p\in G$ $p\Vdash\dot{x}=\dot{y}$,
for all $\dot{x},\dot{y}\in N^{\mathbb{P}}$.
\end{itemize}
Note that $\approx^{*}$ is an equivalence relation that is a congruence
for $\in^{*}$.
\end{defn}

Given that $N$ satisfies $V=L$, we could also define this as a genuine
structure by using $L$-least witnesses. However, for our purposes
this is more trouble than it's worth. Observe that if we have such
a $G$ definable over $N$ without parameters, then $N^{*}$, $\in^{*}$
and $\approx^{*}$ are also all definable without parameters over
$N$. We may then prove a version of the truth lemma for $B\Sigma_{k}$
formulae using a $B\utilde{\Sigma}_{\mathsf{b}(k)}$-generic filter.
First we recall a standard, but very useful, fact about syntactic
forcing.
\begin{fact}
Let $\varphi$ be a fixed formula for which the syntactic forcing
relation has been defined. Then $KP$ proves, for every forcing poset
$\mathbb{P}$ and every $p\in\mathbb{P}$, that the following are
equivalent:\label{fact:Density}\footnote{This is the standard density lemma for the syntactic forcing relation;
compare \citep{KunenST2}. The usual atomic name-induction can be
formulated with a $\Delta_{1}$ induction predicate, so $\Sigma_{1}$-Set
Induction (equivalently, $\Pi_{1}$-Foundation) suffices here; Full
Set Induction is not needed.}
\begin{enumerate}
\item $p\Vdash\varphi$; and
\item $\{q\in\mathbb{P}\mid q\Vdash\varphi\}$ is dense below $p$.
\end{enumerate}
\end{fact}

\begin{lem}
\label{lem:O2ForcingBound}(Uniform forcing bound.) For each fixed
$k\in\omega$, forcing for $B\Sigma_{k}$ formulae is uniformly $\Delta_{\mathsf{b}(k)}$-definable
over $KP$. For bounded formulae alone, the uniform bound is $\Delta_{1}$.
\end{lem}

\begin{proof}
We first consider prenex formulae with at most $r$ alternating blocks
before the final bounded formula. We show, by induction on $r$, that
their forcing relation has both $\Sigma_{2r+1}$ and $\Pi_{2r+1}$
definitions, uniformly in the formula code and the other arguments.
When $r=0$, this is the usual uniform $\Delta_{1}$ definition of
bounded forcing.

Suppose the assertion holds for $r$, and let $h=2r+1$. For a universal
block we have
\[
p\Vdash\forall\bar{x}\,\psi(\bar{x},\bar{\tau})\quad\Longleftrightarrow\quad\forall\bar{\sigma}\ p\Vdash\psi(\bar{\sigma},\bar{\tau}),
\]
where $\bar{\sigma}$ ranges over tuples of names of the required
length. Using the $\Pi_{h}$ definition of forcing for $\psi$, the
right-hand side is $\Pi_{h}$.

For an existential block we use density once for the entire block:
\[
\begin{split}p\Vdash\exists\bar{x}\,\psi(\bar{x},\bar{\tau})\quad\Longleftrightarrow\quad\forall q\leq p\ \exists u\leq q\ \exists\bar{\sigma}\,u\Vdash\psi(\bar{\sigma},\bar{\tau}).\end{split}
\]
This is equivalent to the successive existential forcing clauses.
Given a condition below $p$, refine it successively to obtain the
finitely many witnessing names. Conversely, a condition forcing $\psi$
with those names forces the existential sentence, so the density lemma
applies. A tuple of names is coded by one set. Full Set Induction
proves the existence of the required finite sequences, including when
their length is internally finite in a nonstandard model.

Using the $\Sigma_{h}$ definition of forcing for $\psi$, everything
after $\forall q\leq p$ has a $\Sigma_{h}$ definition. Under our
counting convention the whole expression is therefore $\Pi_{h+1}$.
By adding an unused quantifier, it has both $\Sigma_{h+2}$ and $\Pi_{h+2}$
definitions. These bounds also cover the universal case. The tests
for formula codes and tuples of names are $\Delta_{1}$ and can be
written in the appropriate form without increasing the bound. This
proves the induction.

Finally, a finite Boolean combination of $\Sigma_{k}$ formulae can
be put into prenex form with at most $k+1$ alternating blocks before
a bounded formula. Move negations inward, rename bound variables,
and pad the formulae to a common alternating pattern before joining
their quantifier blocks. This is an effective logical manipulation,
not an application of Collection. Applying the preceding induction
with $r=k+1$ gives
\[
2(k+1)+1=\mathsf{b}(k).
\]
The ordinary logical laws for syntactic forcing identify forcing for
the original formula with forcing for this prenex presentation. The
resulting definition is uniform in its code. Correctness is asserted
for each standard formula; accepting internal formula codes as arguments
does not assert an external truth relation for nonstandard formulae.
\end{proof}
\begin{lem}
\label{lem:O2HYPTruth}($B\Sigma_{k}$-truth lemma.) Suppose $k\geq1$
and $G$ is $B\utilde{\Sigma}_{\mathsf{b}(k)}$-generic over $N$.
Then, for every $B\Sigma_{k}$ formula $\varphi$ and every appropriate
tuple $\bar{\tau}$ of names from $N$,
\[
N^{\mathbb{P}}/G\models\varphi(\bar{\tau})\quad\Longleftrightarrow\quad\exists p\in G\ (p\Vdash\varphi(\bar{\tau}))^{N}.
\]
For bounded formulae, $B\utilde{\Sigma}_{1}$-genericity suffices.
\end{lem}

\begin{proof}
Proceed by simultaneous induction on formulae and their negations.
For $\Delta_{0}$ formulae, the atomic cases are immediate from the
definitions of $\in^{*}$ and $\approx^{*}$. Conjunction follows
from directedness of $G$. For negation, use the dense deciding class
\[
D_{\psi}=\{q\in\mathbb{P}:q\Vdash\psi\ \text{or}\ q\Vdash\neg\psi\}.
\]
Since forcing for bounded formulae is $\Delta_{1}$-definable, $D_{\psi}$
is met by $G$, and the usual compatibility argument gives the negation
clause. 

The bounded-universal case uses the analogous deciding-class argument.
We let
\[
\begin{aligned}D_{0} & =\{q:q\Vdash\forall y\in\dot{z}\,\psi(y)\}\ \text{and }D_{1}=\{q:\exists\dot{y}\ q\Vdash\neg(\dot{y}\in\dot{z}\to\psi(\dot{y}))\}.\end{aligned}
\]
The union is dense and $\Sigma_{1}$-definable. If the bounded universal
is true in the quotient, the inductive hypothesis excludes the second
alternative, so some condition in $G$ forces the bounded universal.
Conversely, if a condition in $G$ forces it, the forcing clauses
and the inductive hypothesis show that it holds of every member of
$\dot{z}$ in the quotient. This establishes the bounded case.

For the higher levels, use the prenex presentations from the uniform
forcing-bound lemma. All the forcing relations needed in the induction
have both $\Sigma_{\mathsf{b}(k)}$ and $\Pi_{\mathsf{b}(k)}$ definitions.
In the negation case the same deciding class $D_{\psi}$ is therefore
$\Sigma_{\mathsf{b}(k)}$-definable, so it is met by $G$. For a universal
formula, use
\[
\begin{aligned}D_{0} & =\{q:q\Vdash\forall y\,\psi(y)\}\ \text{and}\ D_{1}=\{q:\exists\dot{y}\ q\Vdash\neg\psi(\dot{y})\}.\end{aligned}
\]
Both classes have $\Sigma_{\mathsf{b}(k)}$ definitions: for $D_{1}$,
use the $\Sigma_{\mathsf{b}(k)}$ definition of forcing before adjoining
the existential name quantifier. Their union is dense by the forcing
clauses and the density lemma.

If the universal formula is true in the quotient, the inductive hypothesis
rules out the second alternative. Genericity then supplies a condition
in $G$ forcing the universal formula. Conversely, suppose $p\in G$
forces $\forall y\,\psi(y)$. Then $p\Vdash\psi(\dot{y})$ for every
name $\dot{y}\in N$, so the inductive hypothesis shows that $\psi$
holds of every element of the quotient. Existential quantification
is handled by expressing $\exists y\,\psi$ as $\neg\forall y\,\neg\psi$.
Repeated quantifiers of the same kind are handled by induction on
their finite length. This proves the equivalence for the prenex formulae
under consideration; the propositional cases give finite Boolean combinations,
and logical equivalence with the original formula completes the proof.
\end{proof}
Our next goal is to show that $N^{\mathbb{P}}/G$ satisfies $KP$.
We start by observing that standard arguments establish that $KP$
can prove that all of its axioms are forced. 
\begin{lem}
\label{lem:KP-forcing-axioms}For each fixed axiom instance $\varphi$
of $KP$, $KP$ proves that, for every forcing poset $\mathbb{P}$,
\[
1\Vdash_{\mathbb{P}}\varphi.
\]
\end{lem}

We leave the proof of this to the reader although they might like
to consult \citep{BellBVM} or \citep{KunenST2}. As our next step
toward showing $N^{\mathbb{P}}/G$ satisfies $KP$, we first show
that it satisfies Full Set Induction.
\begin{lem}
For all $\mathbb{P}\in N$ and $G$ that is $B\utilde{\Sigma}_{1}$-generic
over $N$, $N^{\mathbb{P}}/G$ satisfies Full Set Induction.\label{lem:FullSetInduction}
\end{lem}

\begin{proof}
It will suffice to show that $N^{\mathbb{P}}/G$ is externally well-founded.
We first record a consequence of the forcing clause for membership.
Suppose that $\dot{y}\in^{*}\dot{x}$ and choose $p\in G$ such that
$p\Vdash\dot{y}\in\dot{x}$. Let
\[
D=\{q\leq p\mid\exists\langle\dot{z},r\rangle\in\dot{x}\ (q\leq r\wedge q\Vdash\dot{y}=\dot{z})\}.
\]
By the syntactic forcing clause for membership, $D$ is dense below
$p$. Moreover, since forcing for atomic equality is $\Delta_{1}$-definable,
$D$ is $\Delta_{1}$-definable over $N$ with parameters. Since $p\in G$
and $D$ is dense below $p$, we may fix some $q\in G\cap D$ and
$\dot{z}\in\mathsf{dom}(\dot{x})$ such that $q\Vdash\dot{y}=\dot{z}$.
Thus
\[
\dot{y}\in^{*}\dot{x}\quad\Longrightarrow\quad\exists\dot{z}\in\mathsf{dom}(\dot{x})\ \dot{z}\approx^{*}\dot{y}.
\]
Now suppose toward a contradiction that there is an infinite descending
sequence
\[
\dot{x}_{0}\ni^{*}\dot{x}_{1}\ni^{*}\dot{x}_{2}\ni^{*}\cdots.
\]
Recursively choose names $\dot{y}_{n}$ such that $\dot{y}_{n}\approx^{*}\dot{x}_{n}$
and
\[
\dot{y}_{n+1}\in\mathsf{dom}(\dot{y}_{n})
\]
for every $n\in\omega$. Start with $\dot{y}_{0}=\dot{x}_{0}$. Given
$\dot{y}_{n}$, the congruence of $\approx^{*}$ gives $\dot{x}_{n+1}\in^{*}\dot{y}_{n}$,
so the preceding observation supplies the required $\dot{y}_{n+1}$.
Consequently,
\[
\mathsf{rank}^{N}(\dot{y}_{n+1})<\mathsf{rank}^{N}(\dot{y}_{n})
\]
for every $n\in\omega$. Since $N=L_{\omega_{1}^{CK}}$ is well-founded,
these are genuine ordinals, giving an impossible infinite descending
sequence of ordinals. Therefore $N^{\mathbb{P}}/G$ is well-founded
and hence satisfies Full Set Induction.
\end{proof}
It is then easy to verify that $N^{\mathbb{P}}/G$ satisfies $KP$
when $G$ is sufficiently generic.
\begin{lem}
For all $\mathbb{P}\in N$ and all $G\subseteq\mathbb{P}$ that are
$B\utilde{\Sigma}_{\mathsf{b}(3)}$-generic over $N$, the quotient
$N^{\mathbb{P}}/G$ satisfies $KP$.\label{lem:genSatKP}
\end{lem}

\begin{proof}
By Lemma \ref{lem:KP-forcing-axioms}, every fixed axiom instance
of $KP$ is forced by $\mathbb{P}$. The basic axioms and the instances
of $\Delta_{0}$-Separation have complexity below $\Pi_{3}$. The
relevant point is $\Delta_{0}$-Collection: its instances can be expressed
as $\Pi_{3}$ sentences. Thus every axiom instance other than Full
Set Induction is a Boolean combination of $\Sigma_{3}$ sentences.
Since $\mathsf{b}(3)=9$, the revised truth lemma transfers all these
forced axioms to $N^{\mathbb{P}}/G$. Finally, $B\utilde{\Sigma}_{\mathsf{b}(3)}$-genericity
implies $B\utilde{\Sigma}_{1}$-genericity, so Lemma \ref{lem:FullSetInduction}
applies. Therefore $N^{\mathbb{P}}/G\models KP$.

\end{proof}
Next we obtain a countable enumeration of the universe of $N$ that
is definable over $N$. 
\begin{lem}
There is a parameter-free $B\Sigma_{1}$-definable, and hence $\Delta_{2}$-definable,
class function $f:\omega\twoheadrightarrow N$.\label{lem:HYPEnumeration}
\end{lem}

\begin{proof}
Working in $N$, fix the standard partial $\Sigma_{1}$ parametrization
$n\mapsto E_{n}$ of hyperarithmetic binary relations on $\omega$,
each with a distinguished root, using numerical indices for witnessed
hyperarithmetic computations.\footnote{For partial parametrization and the Spector-{}-Gandy theorem, see
Y. N. Moschovakis, Descriptive Set Theory, 2nd ed. (2009), Theorems
4D.2 and 4F.3, pp. 166-{}-167 and 185-{}-187.} Call $n$ \emph{good} if $E_{n}$ is defined and below its root,
it is extensional and well-founded in the sense that there is a strictly
order preserving map in $N$ between the segment below its root and
the ordinals of $N$. Let $\mathsf{Val}(n,x)$ assert $n$ is good
and that the transitive collapse of $E_{n}$ sends the root to $x$.
By standard arguments, we see that $\mathsf{Val}(n,x)$ is $\Sigma_{1}$
and functional. Let
\[
f(n)=\begin{cases}
\text{the unique }x\text{ such that }\mathsf{Val}(n,x), & \text{if }n\text{ is good},\\
\varnothing, & \text{otherwise}.
\end{cases}
\]
It should be clear then that $f:\omega\to N$. The graph of $f$ is
defined by
\[
F(n,x)\ \leftrightarrow\ n\in\omega\ \wedge\ \bigl(\mathsf{Val}(n,x)\ \vee\ (\neg\mathsf{Good}(n)\wedge x=\varnothing)\bigr).
\]
This is a Boolean combination of $\Sigma_{1}$ formulae and hence
$\Delta_{2}$. Finally, $f$ is surjective by the well-known fact
that every set in $N$ admits a rooted hyperarithmetic code of the
kind described above.
\end{proof}
We then use this enumeration to define suitably generic sets. 
\begin{lem}
For all $k\in\omega$ and $\mathbb{P}\in N$, there is some $G\subseteq\mathbb{P}$
that is definable over $N$ without parameters and is $B\utilde{\Sigma}_{k}$-generic
over $N$. For $k\geq1$, $G$ may be chosen $\Delta_{k+2}$-definable;
for $k=0$ it may be chosen $\Delta_{2}$-definable.\label{DefiningGenOverHYP}
\end{lem}

\begin{proof}
Fix $k\in\omega$ and $\mathbb{P}\in N$. Since $N\models V=L$, fix
the canonical parameter-free definable well-order $<_{L}$ of $N$.
Let $f:\omega\twoheadrightarrow N$ be as in Lemma \ref{lem:HYPEnumeration}.
Choose $m\in\omega$ with $f(m)=\mathbb{P}$. Since $m$ is a fixed
natural number, it can be represented by its numeral in the defining
formula. Thus the construction below does not use $\mathbb{P}$ as
a parameter. Fix an effective enumeration
\[
\langle\Phi_{i}(x,y)\mid i\in\omega\rangle
\]
of the $B\Sigma_{k}$ formulae with one object variable and one parameter
variable. Using a recursive pairing function, define
\[
e(\langle i,j\rangle)=\langle i,f(j)\rangle.
\]
Thus every pair consisting of a $B\Sigma_{k}$ formula and a parameter
from $N$ occurs in the range of $e$. Let $\mathsf{Sat}_{B\Sigma_{k}}$
be the uniform satisfaction predicate for these formulae. Thus
\[
N\models\Phi_{i}(q,b)\ \Leftrightarrow\ N\models\mathsf{Sat}_{B\Sigma_{k}}(i,q,b).
\]
We now give our modification of the standard Baire category argument.
Thus, we recursively define a descending sequence
\[
\langle p_{n}\mid n\in\omega\rangle
\]
in $\mathbb{P}$ that hits every dense $B\utilde{\Sigma}_{k}$ subset
of $\mathbb{P}$. Let $p_{0}$ be the top element of $\mathbb{P}$.
Then suppose $p_{n}$ has been defined and $e(n)=\langle i,b\rangle$.
If there is $q\leq p_{n}$ such that
\[
N\models\mathsf{Sat}_{B\Sigma_{k}}(i,q,b),
\]
let $p_{n+1}$ be the $<_{L}$-least such $q$; otherwise let $p_{n+1}=p_{n}$.
This is enough to ensure that every dense $B\utilde{\Sigma}_{k}$
subset of $\mathbb{P}$ is hit by some $p_{n}$. Finally let
\[
G=\{q\in\mathbb{P}\mid\exists n\in\omega\ (p_{n}\leq q)\}
\]
which clearly satisfies the requirements of the lemma.

For the complexity calculation, first suppose $k\geq1$. Use the uniform
partial truth predicates for the prenex $\Sigma_{k}$ and $\Pi_{k}$
formulae under our counting convention. Evaluating a finite Boolean
combination gives $\Sigma_{k+1}$ definitions of both its truth and
its falsity: record the finite truth-value assignment and verify its
positive and negative constituents. Hence $\mathsf{Sat}_{B\Sigma_{k}}$
is $\Delta_{k+1}$. The relation saying that $q$ is the selected
successor of $p$ is $\Pi_{k+1}$: use the $\Pi_{k+1}$ presentations
to verify that $q$ qualifies, that no earlier condition qualifies,
or that no condition qualifies and $q=p$. The enumeration values
can be specified universally using the total $B\Sigma_{1}$ graph
of $f$. Consequently, checking that a finite sequence begins with
$p_{0}$ and follows the recursion is a $\Pi_{k+1}$ condition, so
asserting the existence of such a sequence of length $n+1$ ending
in $p$ gives a $\Sigma_{k+2}$ definition of $p=p_{n}$. For $k=0$,
bounded truth and $<_{L}$ are $\Delta_{1}$; recording the finitely
many enumeration values together with such a finite sequence instead
gives a $\Sigma_{2}$ definition. The displayed definition of $G$
therefore gives a $\Sigma_{k+2}$ definition for $k\geq1$, and a
$\Sigma_{2}$ definition for $k=0$. To obtain the complementary definition,
for each $q\in\mathbb{P}$ use the dense $\Delta_{0}$ class $D_{q}=\{p\in\mathbb{P}:p\leq q\text{ or }p\perp q\}$.
Since the constructed filter meets $D_{q}$, we have $q\notin G$
iff some $p_{n}$ is incompatible with $q$. This has the same complexity
bound. Thus $G$ is $\Delta_{k+2}$ for $k\geq1$ and $\Delta_{2}$
for $k=0$. Finally, eliminate the parameter $\mathbb{P}$ using its
fixed numeral index under $f$; this does not increase these bounds.
\end{proof}
Since $G$ is definable over $N$, $N$ can interpret $N^{\mathbb{P}}/G$.
We then observe that we may define the required isomorphism between
$N$ and the constructible universe of $N^{\mathbb{P}}/G$. To make
this a little more concrete, we use $Add(\omega,1)$.
\begin{lem}
\label{DefiningLiso}Let $G\subseteq Add(\omega,1)$ be a parameter-free
definable $B\utilde{\Sigma}_{\mathsf{b}(3)}$-generic filter over
$N$, supplied by the preceding generic-construction lemma. Then there
is an isomorphism
\[
\eta:N\cong L^{N^{Add(\omega,1)}/G}
\]
that is definable over $N$.
\end{lem}

\begin{proof}
Let
\[
\mathcal{M}=N^{Add(\omega,1)}/G.
\]
We first identify the ordinals of $\mathcal{M}$ directly. If $\dot{y}\in^{*}\dot{x}$,
then the syntactic forcing clause for membership and $B\utilde{\Sigma}_{1}$-genericity
give some $\dot{z}\in\mathsf{dom}(\dot{x})$ such that $\dot{z}\approx^{*}\dot{y}$.
Hence any external descending $\in^{*}$-chain can be represented
by names of strictly decreasing $N$-rank.\footnote{A similar argument is given in more detail in the proof of Lemma \ref{lem:FullSetInduction}.}
Since $N=L_{\omega_{1}^{CK}}$ is well-founded, no such chain exists.
Induction on the rank of names also shows that the well-founded rank
of an element represented by $\dot{x}$ is bounded by $\mathsf{rank}^{N}(\dot{x})$.
It follows that every ordinal of $\mathcal{M}$ has order type below
$\omega_{1}^{CK}$. Conversely, every $\alpha<\omega_{1}^{CK}$ is
represented by $[\check{\alpha}]_{G}$. Thus $\mathcal{M}$ and $N$
have the same ordinals. Define a relation $E$ over $N$ by
\[
xE\dot{y}\ \Leftrightarrow\ \exists p\in G\ (p\Vdash\check{x}=\dot{y}).
\]
Because $G$ and atomic forcing are definable over $N$, so is $E$.
On equivalence classes it represents the action of a function such
that
\[
\eta(x)=[\check{x}]_{G}.
\]
The atomic forcing facts for canonical names show that $\eta$ is
an $\in$-embedding of $N$ into $\mathcal{M}$. For each common ordinal
$\alpha$, let
\[
C_{\alpha}=\{[\check{x}]_{G}\mid x\in L_{\alpha}^{N}\}.
\]
We show by induction on $\alpha$ that $C_{\alpha}=L_{\alpha}^{\mathcal{M}}.$
The limit step is immediate. For the successor step, suppose $C_{\alpha}=L_{\alpha}^{\mathcal{M}}$.
The restriction of $\eta$ is an isomorphism between $L_{\alpha}^{N}$
and $L_{\alpha}^{\mathcal{M}}$. Therefore the subsets of $L_{\alpha}^{N}$
definable over that structure with parameters correspond exactly to
the subsets of $L_{\alpha}^{\mathcal{M}}$ definable there with the
corresponding parameters. Hence, $C_{\alpha+1}=L_{\alpha+1}^{\mathcal{M}}.$
Thus the induction holds for every common ordinal. Since $N=L_{\omega_{1}^{CK}}$
and the two structures have the same ordinals, $\mathsf{ran}(\eta)=L^{\mathcal{M}}.$
Consequently $\eta:N\cong L^{\mathcal{M}}$, and the relation $E$
shows that the isomorphism is definable over $N$.
\end{proof}
Finally, we can put all this together to get our first result.
\begin{thm}
$KP$ is not solid. 
\end{thm}

\begin{proof}
Let $\mathbb{P}=Add(\omega,1)$, and let $G\subseteq\mathbb{P}$ be
the parameter-free definable $B\utilde{\Sigma}_{\mathsf{b}(3)}$-generic
filter supplied by Lemma \ref{DefiningGenOverHYP}. Let
\[
\mathcal{M}=N^{\mathbb{P}}/G.
\]
Because $G$ is definable over $N$, we see that 
\[
N\triangleright\mathcal{M}\triangleright L^{\mathcal{M}}.
\]
where the final $\triangleright$ follows since Lemma \ref{lem:genSatKP}
tells us that $\mathcal{M}\models KP$ and so it can define its version
of $L$. Lemma \ref{DefiningLiso} then gives an isomorphism
\[
\eta:N\cong L^{\mathcal{M}}
\]
that is definable over $N$. Finally, we see that $N\not\equiv\mathcal{M}$
since $N$ satisfies $V=L$ while $\mathcal{M}$ does not. Thus, $N\not\cong\mathcal{M}$
and so we see that $KP$ is not solid.
\end{proof}

\subsection{Tightness failure}

Our next goal is to show that $KP$ is not tight. The proof will build
on the techniques of the previous subsection, but generalize them
to the context of arbitrary models of a well-chosen theory, rather
than a particular model. For this purpose, we let
\[
S_{0}=KP+V=L_{\omega_{1}^{CK}}.
\]
The additional axiom may equivalently be expressed as $V=L$ along
with the assertion that there are no admissible ordinals. We then
define a second extension $S_{1}$ by saying that the universe is
the canonically definable Cohen extension of its constructible universe.
The next two lemmas isolate the uniformity required for this definition.
First, we observe that arbitrary models of $S_{0}$ can produce a
suitable version of a definable countable enumeration.
\begin{lem}
There is a parameter-free formula $F(n,x)$ such that every $\mathcal{M}\models S_{0}$
satisfies that $F^{\mathcal{M}}$ is the graph of a surjective class
function\label{lem:tightUniformEnumeration}
\[
f^{\mathcal{M}}:\omega^{\mathcal{M}}\twoheadrightarrow M.
\]
\end{lem}

\begin{proof}
With the internal reading of $V=L_{\omega_{1}^{CK}}$ fixed above,
the proof of Lemma \ref{lem:HYPEnumeration} formalizes in $S_{0}$.
The internal witnessed coding construction gives $S_{0}\vdash\forall x\,\exists n\in\omega\,\mathsf{Val}(n,x)$.
Functionality and the default-value construction are provable in $KP$.
Hence the same parameter-free formula $F$ defines the required surjection
in every $\mathcal{M}\models S_{0}$.
\end{proof}
We then observe that the required generics can also be defined over
models of $S_{0}$.
\begin{lem}
\label{lem:tightUniformGeneric}There is a parameter-free formula
$\gamma(p)$ such that, for every $\mathcal{M}\models S_{0}$,
\[
G^{\mathcal{M}}=\{p\in Add(\omega,1)^{\mathcal{M}}:\mathcal{M}\models\gamma(p)\}
\]
is $B\utilde{\Sigma}_{\mathsf{b}(3)}$-generic over $\mathcal{M}$.
\end{lem}

\begin{proof}
Use the generic construction from Section \ref{subsec:Solidity-failure},
now with the uniform enumeration supplied by the preceding Lemma \ref{lem:tightUniformEnumeration}.
Thus we enumerate the pairs of $B\Sigma_{\mathsf{b}(3)}$ formulae
and parameters and recursively form a descending sequence of conditions.
At the stage corresponding to a formula-parameter pair, if the class
it defines contains an extension of the current condition, choose
the $<_{L}$-least such extension; otherwise keep the current condition.
Full Set Induction proves the existence of every internally finite
initial sequence. The filter generated by this sequence meets every
$B\utilde{\Sigma}_{\mathsf{b}(3)}$-definable dense class. All parts
of this construction are given by fixed formulae and are provable
in $S_{0}$, so the resulting definition $\gamma$ is uniform over
models of $S_{0}$.
\end{proof}
Fix the formula $\gamma$ from Lemma \ref{lem:tightUniformGeneric}.
Let $\Psi_{\gamma}$ be the sentence asserting that there is a set
$g$ such that:
\begin{itemize}
\item $p\in g\text{ iff }p\in Add(\omega,1)^{L}\ \wedge\ L\models\gamma(p);$
\item $g$ is $Add(\omega,1)^{L}$-generic over $L$;
\item $L$ has no admissible ordinals; and
\item $V=L[g]$.
\end{itemize}
Since $\gamma$ is fixed, this is a single first-order sentence. Let
\[
S_{1}=KP+\Psi_{\gamma}.
\]
The set $g$ occurring in $\Psi_{\gamma}$ is unique.

We now verify that the canonical generic extensions of models of $S_{0}$
satisfy $KP$. First we record the form of the truth lemma needed
for arbitrary models.
\begin{lem}
\label{lem:O2GeneralTruth}(Truth lemma for arbitrary models.) Suppose
$\mathcal{M}\models KP$, $\mathbb{P}\in M$ is a forcing poset, and
$k\geq1$ is standard. If $G$ is $B\utilde{\Sigma}_{\mathsf{b}(k)}$-generic
over $\mathcal{M}$, then, for every standard $B\Sigma_{k}$ formula
$\varphi$ and every appropriate tuple $\bar{\tau}$ of $\mathcal{M}$-names,
\[
\mathcal{M}^{\mathbb{P}}/G\models\varphi(\bar{\tau})\quad\Longleftrightarrow\quad\exists p\in G\ \mathcal{M}\models p\Vdash\varphi(\bar{\tau}).
\]
For bounded formulae, $B\utilde{\Sigma}_{1}$-genericity suffices.
\end{lem}

The proof is the same as in Section \ref{subsec:Solidity-failure}:
the induction is external on standard formulae and uses only $KP$,
the forcing clauses, the density lemma and the indicated genericity.
It does not require the ground model to be externally well-founded.

As before, we treat Full Set Induction separately. Here definability
of the generic is used to apply Full Set Induction in the ground model.
\begin{lem}
Suppose $\mathcal{M}\models KP$, $\mathbb{P}\in M$, and $G\subseteq\mathbb{P}$
is definable over $\mathcal{M}$ and $B\utilde{\Sigma}_{1}$-generic
over $\mathcal{M}$. Then $\mathcal{M}^{\mathbb{P}}/G$ satisfies
Full Set Induction.\label{lem:tightFullSetInduction}
\end{lem}

\begin{proof}
Let $\mathcal{N}=\mathcal{M}^{\mathbb{P}}/G$ and suppose some instance
of Full Set Induction fails in $\mathcal{N}$. For a formula witnessing
this failure, let $C$ be the class of names satisfying that formula
in $\mathcal{N}$. Since $G$ is definable over $\mathcal{M}$, the
class $C$ is definable over $\mathcal{M}$. By Full Set Induction
in $\mathcal{M}$, choose $\dot{x}\in C$ of minimal $\mathcal{M}$-rank.
Failure of Set Induction gives $\dot{y}\in C$ with $\dot{y}\in^{*}\dot{x}$.
Choose $p\in G$ with $p\Vdash\dot{y}\in\dot{x}$. By the forcing
clause for membership, the class of $q\leq p$ for which there is
some $\dot{z}\in\mathsf{dom}(\dot{x})$ with $q\Vdash\dot{y}=\dot{z}$
is dense below $p$ and $\Delta_{1}$-definable. Hence $G$ meets
it. For the resulting $\dot{z}$, we have $\dot{z}\approx^{*}\dot{y}$
and therefore $\dot{z}\in C$, while
\[
\mathcal{M}\models\mathsf{rank}(\dot{z})<\mathsf{rank}(\dot{x}).
\]
This contradicts the choice of $\dot{x}$.
\end{proof}
The verification of the rest of the axioms of $KP$ then follows easily.
\begin{cor}
If $\mathcal{M}\models S_{0}$ and $G^{\mathcal{M}}$ is the canonical
generic from Lemma \ref{lem:tightUniformGeneric}, then\label{cor:tightKPPreservation}
\[
\mathcal{M}^{Add(\omega,1)}/G^{\mathcal{M}}\models KP.
\]
\end{cor}

\begin{proof}
The proof of $KP$-preservation in Section \ref{subsec:Solidity-failure}
applies to all axioms other than Full Set Induction. Those axioms
are forced and are $B\Sigma_{3}$ sentences. The canonical generic
is $B\utilde{\Sigma}_{\mathsf{b}(3)}$-generic, so the truth lemma
for arbitrary models transfers them to the quotient. Full Set Induction
is supplied by the preceding lemma.
\end{proof}
Next we show that one of the isomorphisms required to witness bi-interpretability
is suitably definable.
\begin{lem}
Let $\mathcal{M}\models S_{0}$ and let
\[
\mathcal{Q}=\mathcal{M}^{Add(\omega,1)}/G^{\mathcal{M}}.
\]
There is an isomorphism
\[
\eta^{\mathcal{M}}:\mathcal{M}\cong L^{\mathcal{Q}}
\]
 which is uniformly definable over $\mathcal{M}$.\label{lem:tightLRecovery}
\end{lem}

\begin{proof}
We first show that every ordinal of $\mathcal{Q}$ is represented
by a canonical ordinal name. Suppose $\mathcal{Q}\models\dot{\alpha}\in Ord.$
By the $\Delta_{0}$ truth lemma, choose $p\in G^{\mathcal{M}}$ such
that $(p\Vdash\dot{\alpha}\in Ord)^{\mathcal{M}}$. Working in $\mathcal{M}$,
the usual rank-bounding argument for ordinal names\footnote{This rank bound is provable in $KP$. One argues by induction on the
rank of the name: if $p\Vdash\dot{x}\in Ord$, then $KP$ constructs
an ordinal $\delta$ bounding all possible ordinal values of $\dot{x}$
and proves $p\Vdash\dot{x}\leq\check{\delta}$. No genericity is used
in this rank-bound itself; the genericity of $G$ is used immediately
afterward to meet the resulting definable dense class.} gives some $\delta\in Ord^{\mathcal{M}}$ with $p\Vdash\dot{\alpha}\leq\check{\delta}.$
Hence $p\Vdash\dot{\alpha}\in\check{\delta}+1$. By the syntactic
forcing clause for membership,
\[
D=\{q\leq p:\exists\beta<\delta+1\ q\Vdash\dot{\alpha}=\check{\beta}\}
\]
is dense below $p$. This class is $\Delta_{1}$-definable over $\mathcal{M}$,
so $G^{\mathcal{M}}$ meets it. Thus every ordinal of $\mathcal{Q}$
is represented by $\check{\beta}$ for some $\beta\in Ord^{\mathcal{M}}$.
The converse is immediate, so the canonical-name map identifies the
ordinals of $\mathcal{M}$ and $\mathcal{Q}$. Define $\eta^{\mathcal{M}}$
by the relation
\[
\eta^{\mathcal{M}}(x)=[\check{x}]_{G^{\mathcal{M}}}.
\]
Since $G^{\mathcal{M}}$ is defined by the fixed formula $\gamma$,
this relation is uniformly definable over $\mathcal{M}$. The atomic
forcing facts show that it is an $\in$-embedding. Using the preceding
identification of ordinals, Full Set Induction in $\mathcal{M}$ gives,
for every $\alpha\in Ord^{\mathcal{M}}$,
\[
\eta^{\mathcal{M}}[L_{\alpha}^{\mathcal{M}}]=L_{\alpha}^{\mathcal{Q}}.
\]
The successor step follows because definable subsets are preserved
under isomorphism, and the limit step is immediate. Since $\mathcal{M}\models V=L$
and the two structures have the same ordinals, the range of $\eta^{\mathcal{M}}$
is exactly $L^{\mathcal{Q}}$.
\end{proof}
Then, analogously to the previous subsection, we see that models of
$S_{0}$ can interpret models of $S_{1}$.
\begin{cor}
For every $\mathcal{M}\models S_{0}$, the quotient
\[
\mathcal{M}^{Add(\omega,1)}/G^{\mathcal{M}}
\]
 is a model of $S_{1}$. Hence $S_{0}$ interprets $S_{1}$.\label{cor:tightS0InterpretsS1}
\end{cor}

\begin{proof}
Let $\mathcal{Q}$ be the quotient. By Corollary \ref{cor:tightKPPreservation},
$\mathcal{Q}\models KP$. By Lemma \ref{lem:tightLRecovery}, $L^{\mathcal{Q}}$
is isomorphic to $\mathcal{M}$ and therefore has no admissible ordinals.
Let $g$ be the interpretation in $\mathcal{Q}$ of the standard name
for the generic filter. Under the isomorphism $\eta^{\mathcal{M}}$,
this is exactly the subset of $Add(\omega,1)^{L^{\mathcal{Q}}}$ defined
over $L^{\mathcal{Q}}$ by $\gamma$. It is generic over $L^{\mathcal{Q}}$,
and the standard valuation recursion gives
\[
\mathcal{Q}=L^{\mathcal{Q}}[g].
\]
Thus $\mathcal{Q}\models\Psi_{\gamma}$.
\end{proof}
Finally, we show that the other isomorphism, required for the bi-interpretation,
is suitably definable.
\begin{lem}
If $\mathcal{N}\models S_{1}$, then $L^{\mathcal{N}}\models S_{0}$.
Moreover, if $g$ is the unique set specified by $\Psi_{\gamma}$,
there is a uniformly definable isomorphism\label{CohenIso}
\[
\nu^{\mathcal{N}}:(L^{\mathcal{N}})^{Add(\omega,1)}/g\cong\mathcal{N}.
\]
\end{lem}

\begin{proof}
The constructible universe of a model of $KP$ again satisfies $KP$.
By $\Psi_{\gamma}$, $L^{\mathcal{N}}$ has no admissible ordinals,
and of course it satisfies $V=L$. Hence $L^{\mathcal{N}}\models S_{0}$.
The same axiom says that $g$ is the canonical $Add(\omega,1)^{L^{\mathcal{N}}}$-generic
filter and that $\mathcal{N}=L^{\mathcal{N}}[g]$. Let $\nu^{\mathcal{N}}$
be induced by the familiar valuation function:\footnote{See Chapter VII of \citep{KunenST}.}
\[
\nu^{\mathcal{N}}([\dot{x}]_{g})=\mathsf{val}_{g}(\dot{x}).
\]
The forcing theorem gives that this is well-defined and preserves
$\in$, and every element of $L^{\mathcal{N}}[g]$ is the value of
an $L^{\mathcal{N}}$-name. Thus $\nu^{\mathcal{N}}$ is an isomorphism.
Since $g$ is defined without parameters from $L^{\mathcal{N}}$ by
the fixed formula $\gamma$, the valuation relation, and hence $\nu^{\mathcal{N}}$,
is uniformly definable over $\mathcal{N}$.
\end{proof}
And now we have our second, more powerful, result.
\begin{thm}
$KP$ is not tight.\label{thm:KPisNotTight}
\end{thm}

\begin{proof}
Let $t$ be the interpretation of $S_{1}$ in $S_{0}$ supplied by
Corollary \ref{cor:tightS0InterpretsS1}, and let $s$ be the interpretation
of $S_{0}$ in $S_{1}$ given by the constructible universe. Lemma
\ref{lem:tightLRecovery} gives a uniformly definable isomorphism
from each $\mathcal{M}\models S_{0}$ to $s\circ t(\mathcal{M})$,
while Lemma \ref{CohenIso} gives a uniformly definable isomorphism
from $t\circ s(\mathcal{N})$ to each $\mathcal{N}\models S_{1}$.
Thus $S_{0}$ and $S_{1}$ are bi-interpretable. They are not the
same theory. Indeed, $S_{0}\vdash V=L$, whereas $S_{1}\vdash V\neq L$,
since the set $g$ asserted by $\Psi_{\gamma}$ is Cohen generic over
$L$. Hence $KP$ is not tight.
\end{proof}

\subsection{Some generalizations}

Our goal now is to generalize these results to suitable theories that
are consistent with $V=L$. The reader will have observed there are
two crucial ingredients to the proofs above: 
\begin{enumerate}
\item We need a definable countable enumeration of the universe in order
to define generic filters and generic extensions. 
\item We need a ground model to generically extend from and which can be
defined in the generic extension. 
\end{enumerate}
If we force out of a model of $V=L$, then (2) is easy to achieve.
For theories consistent with $V=L$, (1) can also be obtained straightforwardly
using elementary fine-structural techniques as long as we don't have
Full Separation and Collection. With that in mind, let us say that
a set $S$ of sentences from $\mathcal{L}_{\in}$ is \emph{quantifier}
\emph{bounded} if there is some $k\in\omega$ such that every sentence
$\varphi\in S$ is $\Sigma_{k}$.

We then observe that whenever we extend $KP$ by a quantifier bounded
set of sentences, we have a theory that can be extended to say that
the universe has a definable countable enumeration.
\begin{lem}
Let $T=U+KP$ be consistent with $V=L$, where $U$ is quantifier
bounded. Then there is a parameter-free formula $\theta(n,x)$ such
that
\[
T+V=L+``\theta\text{ defines a class surjection from }\omega\text{ onto }V''
\]
 is consistent.\footnote{This lemma is closely related to \citet[Theorem 3(a,b)]{enayat2026corrigendumaddendumcategoricitylike}.
Their construction gives a partially elementary hull $K_{n}(\mathcal{M})$
of elements definable by standard $\Sigma_{n}$-formulae and establishes
its idempotence. In a nonstandard model this hull need not be definable
in $\mathcal{M}$, so Full Set Induction does not follow merely by
relativization. Here the hull is definable in the ambient model and
has the inherited membership relation. It therefore satisfies Full
Set Induction, without any assumption of $\omega$-standardness.}\label{lem:def-countable-enumeration-L}
\end{lem}

\begin{proof}
Our plan is to work in $T+V=L$ and define a sufficiently strong Skolem
function on the universe that will give us a hull upon which that
function delivers our required enumeration. Working now in $T+V=L$,
let $m+1\geq3$ be sufficiently large that the axioms of $U$ and
$KP+V=L$, except Full Set Induction, can be articulated as $\Pi_{m+1}$
sentences. Let $\mathsf{Sat}_{\Pi_{m}}$ be a satisfaction predicate
for $\Pi_{m}$ formulae with three free variables, so we have
\[
\mathsf{Sat}_{\Pi_{m}}(\ulcorner\psi\urcorner,\langle x,y,a\rangle)\;\Leftrightarrow\ \psi(x,y,a).
\]
Note that $\mathsf{Sat_{\Pi_{m}}}$ can be articulated as a $\Pi_{m}$
formula. We assume that our Gödel coding is in bijection with $\omega$.

We now define a $\Sigma_{m+1}$-Skolem function in three steps. First,
we let $h_{0}:\omega\times V\rightharpoondown V\times V$ be such
that $h_{0}(\ulcorner\psi\urcorner,a)=\langle x,y\rangle$ iff 
\[
\underbrace{\mathsf{Sat}_{\Pi_{m}}(\ulcorner\psi\urcorner,\langle x,y,a\rangle)}_{\Pi_{m}}\wedge\underbrace{\forall u,v,w(\underbrace{w=\mathsf{pr}_{L}(\langle x,y\rangle)}_{\Delta_{1}}\wedge\langle u,v\rangle\in w\to\underbrace{\neg\mathsf{Sat}_{\Pi_{m}}(\ulcorner\psi\urcorner,\langle u,v,a\rangle)}_{\Sigma_{m}})}_{\Pi_{m+1}}
\]
where $\mathsf{pr}_{L}(z)$ is the $\Delta_{1}$ definable set of
$L$-predecessors of $z$.\footnote{See Lemma II.3.5 in \citep{DevlinC}.}
$h_{0}$ takes the Gödel code of a $\Pi_{m}$-formula along with a
parameter and returns the $L$-least pair satisfying it. Note that
$h_{0}$ is $\Pi_{m+1}$ definable. Next we let $h_{1}:\omega^{<\omega}\rightharpoondown V^{<\omega}$
be such that $h_{1}(s)=x$ iff
\[
\underbrace{x_{0}=h_{0}(s_{0},\emptyset)}_{\Pi_{m+1}}\wedge\underbrace{\forall i+1<lh(s)\ \underbrace{x_{i+1}=h_{0}\bigl(s_{i+1},\langle\pi_{0}(x_{0}),\ldots,\pi_{0}(x_{i})\rangle\bigr)}_{\Pi_{m+1}}}_{\Pi_{m+1}}
\]
where $\pi_{0}$ gives us the first element of an ordered pair. This
function takes a sequence of formulae and returns a sequence of sets
that are inductively defined from using those formulae. This is intended
to replicate the inductive steps in a standard argument for the downward
Löwenheim-Skolem theorem. Finally, we let $h(s)=y$ iff $y$ is the
first member of the final pair in $h_{1}(s)$.

A standard Tarski-Vaught argument then establishes that $h[\omega^{<\omega}]$
is a $\Sigma_{m+1}$-Skolem hull of the universe. Moreover, since
$h_{0}$ is $\Pi_{m+1}$ definable over, it can then be seen that
the formula defining $h$ in $V$ also defines $h$ in $h[\omega^{<\omega}]$.
Clearly then, $h:\omega\to h[\omega^{<\omega}]$ is a partial surjection
that is definable over $h[\omega^{<\omega}]$. It is then a simple
matter to define a total surjection $h^{*}:\omega\twoheadrightarrow h[\omega^{<\omega}]$
in $h[\omega^{<\omega}]$. To close things out, we then observe that
$h[\omega^{<\omega}]$ satisfies $T+V=L$ since we chose $m$ to be
sufficiently large for all axioms with the exception of Full Set Induction,
and the latter holds since $h[\omega^{<\omega}]$ is a definable submodel
of a universe satisfying Full Set Induction that shares its membership
relation.
\end{proof}
We then note that if we follow the template suggested by the proof
of Theorem \ref{thm:KPisNotTight}, we want to ensure that our theory
is preserved when we move to a generic extension.
\begin{defn}
Let $T$ be a theory extending $KP$ and let $\mathbb{P}$ be a forcing
notion uniformly definable without parameters over models of $T$.
We say that $T$ is \emph{generically invariant} with respect to $\mathbb{P}$
if there is some $k\in\omega$ such that whenever $\mathcal{M}\models T$
and $G\subseteq\mathbb{P}^{\mathcal{M}}$ is definable over $\mathcal{M}$
and $B\utilde{\Sigma}_{k}$-generic over $\mathcal{M}$, then $\mathcal{M}^{\mathbb{P}}/G\models T$.
\end{defn}

Now we may state a more general theorem.
\begin{thm}
Let $T=U+KP$ be consistent with $V=L$ where $U$ is quantifier bounded.
Moreover suppose that $T$ is generically invariant with respect to
some nontrivial set forcing uniformly definable without parameters
over $T$. Then $T$ is not tight.\footnote{See Proposition \ref{prop:MG-single} and the subsequent remark for
further non-tightness results, including applications to theories
incompatible with $V=L$.} \label{thm:generic-invariance-not-tight}
\end{thm}

\begin{proof}
Let $k$ witness generic invariance. Enlarge $k$, if necessary, so
that $k\geq\mathsf{b}(j)$ for a bound $j$ on the fixed additional
formulae used in the forcing and truth arguments below. By Lemma \ref{lem:def-countable-enumeration-L},
fix a parameter-free formula $\theta(n,x)$ such that the theory
\[
S_{0}=T+V=L+\text{``}\theta\text{ defines a class surjection from }\omega\text{ onto }V\text{''}
\]
is consistent. Let $\gamma$ be the fixed parameter-free formula obtained
by applying the least-extension construction from the earlier tightness
proof to $\mathbb{P}$ using $\theta$. Thus, in every model of $S_{0}$,
$\gamma$ defines a $B\utilde{\Sigma}_{k}$-generic filter $G$ on
$\mathbb{P}$. By generic invariance, the corresponding generic extension
is again a model of $T$.

Let $S_{1}$ extend $T$ by asserting that its constructible universe
satisfies $S_{0}$ and that there is a set $g$ such that
\[
p\in g\ \Leftrightarrow\ L\models\gamma(p),
\]
$g$ is $\mathbb{P}^{L}$-generic over $L$, and $V=L[g]$. If $\mathcal{M}\models S_{0}$
and $G$ is its canonical generic, then generic invariance gives
\[
\mathcal{M}^{\mathbb{P}}/G\models T.
\]
 The same ordinal-recovery argument used in the earlier $KP$ tightness
proof identifies the constructible universe of this extension with
$\mathcal{M}$, while the standard valuation recursion gives
\[
\mathcal{M}^{\mathbb{P}}/G=(L[G])^{\mathcal{M}^{\mathbb{P}}/G}.
\]
 Hence $S_{0}$ interprets $S_{1}$.

Conversely, if $\mathcal{N}\models S_{1}$, then $L^{\mathcal{N}}\models S_{0}$.
Interpreting $S_{0}$ by $L$ and then taking its canonical $\mathbb{P}$-generic
extension recovers $\mathcal{N}$ up to the usual definable valuation
isomorphism. In the other direction, the constructible universe of
the canonical generic extension of a model of $S_{0}$ is definably
isomorphic to the original model, exactly as in the earlier proof.
Thus $S_{0}$ and $S_{1}$ are bi-interpretable.

Finally, $S_{0}\vdash V=L$, whereas $S_{1}\vdash V\neq L$. Indeed,
the canonical $g$ cannot belong to $L$: otherwise $\mathbb{P}^{L}\setminus g$
would be a dense set in $L$ disjoint from $g$, using the nontriviality
of $\mathbb{P}^{L}$. Hence the two theories have different deductive
closures. They are distinct bi-interpretable extensions of $T$, so
$T$ is not tight.
\end{proof}
Thus, to obtain theories that aren't tight, we simply need to find
quantifier bounded theories that are generically invariant and consistent
with $V=L$. Here are some obvious examples. 
\begin{cor}
Let $\Phi$ be any of the statements below. If $KP+\text{Powerset}+\Phi+V=L$
is consistent, then $KP+\text{Powerset}+\Phi$ is not tight.
\begin{enumerate}
\item There is an inaccessible cardinal; 
\item There is a Mahlo cardinal;
\item There is a weakly compact cardinal, where weak compactness is formulated
as inaccessibility together with the tree property.
\end{enumerate}
\end{cor}

\begin{proof}
Fix one of the displayed statements $\Phi$ and let $T=KP+\text{Powerset}+\Phi$.
By hypothesis, $T+V=L$ is consistent. Let $\mathbb{P}=Add(\omega,1)$.
This is a nontrivial forcing uniformly definable without parameters.
If $\kappa$ witnesses $\Phi$, then $|\mathbb{P}|<\kappa$. The standard
small-forcing argument preserves inaccessibility and Mahloness; in
the third case, the usual name-tree argument also preserves the tree
property. The usual name-bound argument for Powerset, together with
the preceding $KP$-preservation results and these small-forcing arguments,
shows that the non-induction axioms of $T$ are forced. Choose a fixed
$j$ bounding their complexity. The truth lemma for arbitrary models
transfers them whenever the generic is $B\utilde{\Sigma}_{\mathsf{b}(j)}$-generic.
For definable generics, the separate Full Set Induction lemma supplies
the remaining axioms. Thus $T$ is generically invariant with respect
to $\mathbb{P}$, and the result now follows from Theorem \ref{thm:generic-invariance-not-tight}.
\end{proof}

\section{Quantifier bounded theories generally aren't tight}

\label{sec:Quantifier-bounded-theories}

We now have a rich collection of theories compatible with $V=L$
that fail to be tight. But what if we wanted to consider theories
that countenance, say, the existence of a measurable cardinal or a
proper class of extendible cardinals? These extensions are not consistent
with $V=L$ and thus, the techniques described above are of little
use. Nonetheless, the underlying strategy can be adapted. In general,
we are looking for theories of the appropriate form that: 
\begin{enumerate}
\item are consistent with a countable definable enumeration; and
\item admit generic extensions where the ground model can be recovered by
definition.
\end{enumerate}
With regard to (2), we now identify a property of theories that makes
them suitable for the anti-tightness arguments we are about to give.
For this purpose, we make use of a simple extension of $KP$. In particular,
we let $KPC$ be $KP$ with the Axiom of Choice in the form: $\forall x\exists f\exists\alpha\ f:x\rightarrowtail\alpha$. 
\begin{defn}
Say that a theory $T$ has the \emph{stable definable ground property}
if there is an inner model $\mathsf{N}$ of $KPC$ definable over
$T$ and a term $\mathbb{P}$ defined over $T$ by a formula relativized
to $\mathsf{N}$ such that $T$ proves \label{def:StableDefGprop}
\begin{itemize}
\item $\mathbb{P}\in\mathsf{N}\text{ is a non-trivial forcing poset}$; 
\item $``(\Vdash_{\mathbb{P}}\psi)^{\mathsf{N}}"$ for all $\psi\in T$; 
\item $\exists g\in gen(\mathbb{P},\mathsf{N})\ V=\mathsf{N}[g]$; and
\item $(\Vdash_{\mathbb{P}}``\mathsf{N}^{V[\dot{G}]}=\check{V}")^{\mathsf{N}}$.
\end{itemize}
The last clause deserves some explanation. We are saying that any
$\mathbb{P}$-generic extension $\mathsf{N}[h]$ of the defined $\mathbb{P}$-ground
$\mathsf{N}$ is such that $\mathsf{N}[h]$'s version of $\mathsf{N}$
is still $\mathsf{N}$. This is why we call it stably definable.
\end{defn}

In essence, theories with the stable definable ground property are
theories whose models are generic extensions of inner models that
are robustly definable from their generic extensions. For example,
$KP$ together with the assertion that the universe is a Cohen extension
of $L$ has the stable definable ground property, witnessed by $\mathsf{N}=L$
and $\mathbb{P}=Add(\omega,1)^{L}$.

Moving onto (1), we prove a very general lemma about theories obtained
from extending a weak set theory with a quantifier bounded collection
of sentences. 
\begin{lem}
If $T=KPC+S$ is consistent and $S$ is quantifier bounded, then there
is a parameter free formula $\theta(n,x)$ such that 
\[
T+``\theta\text{ defines a class surjection from }\omega\text{ onto }V''
\]
is consistent.\label{count-enum-thm} The same conclusion holds for
a fixed finite expansion of the language by constant symbols.
\end{lem}

Let's call this the \emph{Definable Countable Enumeration Lemma}.\emph{
}Its proof takes a little while and uses some unusual techniques.
As such, we defer its proof to the next section and offer a simple
demonstration of our main theorem template on the basis of this promise. 
\begin{thm}
Suppose $T=KPC+U$ is consistent where $U$ is quantifier bounded
and that $T$ has the stable definable ground property. Then $T$
is not tight.\label{thm:MainTemplate}
\end{thm}

\begin{proof}
Let $R$ be $KPC$ together with the assertion that $\mathbb{P}$
is non-trivial, the assertions $1_{\mathbb{P}}\Vdash\varphi$ for
$\varphi\in U$, and the assertion that forcing with $\mathbb{P}$
preserves the fixed definition of $\mathsf{N}$ in the sense of the
last clause of Definition \ref{def:StableDefGprop}. A model of $T$
together with its distinguished ground witnesses the consistency of
$R$. Choose a fixed $j$ such that the sentences in $U$, the non-induction
axioms of $KPC$, and the fixed sentence expressing stability of the
ground definition are all $B\Sigma_{j}$. By the uniform forcing-bound
lemma, their forcing translations have a common bound $\Delta_{\mathsf{b}(j)}$.
The remaining assertion in $R$, that $\mathbb{P}$ is nontrivial,
also has fixed finite complexity. Thus $R$ is of the form $KPC+S$
with $S$ quantifier bounded.

By the Definable Countable Enumeration Lemma (Lemma \ref{count-enum-thm}),
we may then fix a parameter-free formula $\theta$ and a model $W\models R$
in which $\theta$ defines a class surjection from $\omega$ onto
$W$. Choose $k\geq\mathsf{b}(j)$, enlarged as necessary for the
usual product-forcing argument, and apply the least-extension construction
to $\mathbb{P}\times\mathbb{P}$. At each stage, choose an eligible
extension with the least index in the enumeration defined by $\theta$.
This gives a canonical $B\utilde{\Sigma}_{k}$-generic pair $\langle G_{0},G_{1}\rangle$
over $W$, with each projection sufficiently generic for the formulae
just listed. Let $M_{i}=W[G_{i}]$ for $i\in2$. Set forcing forces
the non-induction axioms of $KPC$, and the forcing clauses in $R$
give the members of $U$. The truth lemma for arbitrary models transfers
these sentences to $M_{i}$, while definability of $G_{i}$ gives
Full Set Induction by the separate preservation lemma (Lemma \ref{lem:tightFullSetInduction}).
Hence $M_{i}\models T$. Stability gives $\mathsf{N}^{M_{i}}=W$ for
each $i\in2$. For $i\in2$, let $S_{i}$ extend $T$ by saying that
$\theta^{\mathsf{N}}$ enumerates $\mathsf{N}$ and that the universe
is the $i^{th}$ coordinate of the canonical $\mathbb{P}\times\mathbb{P}$-generic
pair thereby defined over $\mathsf{N}$. The models $M_{i}$ witness
consistency. In any model of $S_{i}$, first recover $\mathsf{N}$,
then $\theta$, and hence the other coordinate extension. When we
interpret back, the stability clause recovers the same $\mathsf{N}$
and therefore the same canonical pair, so the usual valuation map
gives the required definable isomorphism. Thus $S_{0}$ and $S_{1}$
are bi-interpretable. Finally, we observe that $S_{0}\neq S_{1}$
since $S_{0}$ says it's formed from the $0^{th}$ projection while
$S_{1}$ says it's formed from the $1^{st}$ projection. This then
suffices since the associated filters are mutually generic. Hence
$T$ is not tight.
\end{proof}
Thus, if we want to show that an extension of $KPC$, of the kind
discussed above, fails to be tight, it will suffice for us to show
that theory has a consistent extension with the stable definable ground
property. In Section \ref{sec:Applications}, we will show this for
a wide variety of theories.

\section{The Definable Countable Enumeration Lemma}

\label{sec:DCET}

Our goal now is to prove Lemma \ref{count-enum-thm}. This is the
most original part of the paper and it takes the most time and effort
to prove. The strategy is simple to sketch, but fussy to execute.
The argument has three main steps. We start by working in an arbitrary
model of a suitable theory $T$ which we treat as the universe. Working
in this context, we then describe a way of building a model of $T$
by approximation using something like a tableau proof system, model
existence game, inverted sequent, or analytic consistency property
\citep{SmullFO,Vaananen_2011,Pohlers,Barwise}. Such models are countable
as witnessed by the fact that they are defined from a countable set
of constant symbols. The model we obtain can then be interpreted in
our adopted universe and we call it $\mathcal{M}_{0}$. We shall also
make use of an ordinal side-condition that delivers a definable order-preserving
map from $\mathcal{M}_{0}$ into the ordinals of this background universe.
This is the sense in which the model is well-founded, and it ensures
Full Set Induction. For the second step, we shall be very careful
in setting up our model building process so that when we use the same
process inside $\mathcal{M}_{0}$, we define another model $\mathcal{M}_{1}$
with an isomorphism $j:\mathcal{M}_{0}\cong\mathcal{M}_{1}$ definable
in the background universe. Thus, there is a sense in which our model
building process is idempotent. This is the tricky part. For the final
step, we use $\mathcal{M}_{0}$'s countable presentation of $\mathcal{M}_{1}$
to define, without parameters and inside $\mathcal{M}_{0}$, a countable
enumeration of itself. Obtaining this internal definition requires
a further argument, supplied in Lemma \ref{lem:defCountEnum}. Thus,
$\mathcal{M}_{0}$ witnesses the claim.

So much for the sketch. The rest of this section is divided up as
follows. First, we recall some of the basics of tableau proof systems.
Then we describe some modifications of this system that will fulfill
the requirements of the sketch above. After that, we show that it
does what we need it to do.

\subsection{Tableau proof systems: building models by approximation}

Tableau proof systems originated with Beth, Hintikka and Kanger \citeyearpar{Beth1955,Hintikka1955,Kanger1957}
as a means of providing a simple proof theory for first order logic
whereby we take a sentence $\varphi$ and search for a counterexample
to it. This is done by taking $\neg\varphi$ and systematically unpacking
it to find out if it is satisfiable. If $\vdash\varphi$, then $\neg\varphi$
is inconsistent, and the unpacking process applied to $\neg\varphi$
eventually crashes revealing that $\neg\varphi$ is unsatisfiable;
i.e., $\models\varphi$. On the other hand, if $\neg\varphi$ is consistent,
the process never crashes and we end up defining a model in which
$\neg\varphi$ is satisfied; thus, $\not\models\varphi$. More generally,
the unpacking process can be used to take a simple, countable set
of sentences and ascertain whether it is satisfiable.\footnote{Of course, in general, it will take forever to ascertain that it is
satisfiable.} 

\subsubsection{Rules for $FOL$ without identity}

Let us now recall the process in a little more detail.\footnote{For a more thorough discussion, see \citep{SmullFO,TolTab,RestLog,PriestNCL}}
Suppose we are given a sentence $\varphi$ from some language $\mathcal{L}$
without identity and we wish to figure out whether there is a model
$\mathcal{M}$ satisfying $\varphi$.\footnote{We'll generalize to sets in a moment, but it is easy to start this
way. Let us also assume that $\mathcal{L}$ has no function or constant
symbols.} The tableau for $\varphi$ is a tree whose nodes are sentences of
the language $\mathcal{L}(C)$, which is $\mathcal{L}$ expanded with
a countable set $C=\{c_{n}\}_{n\in\omega}$ of new constant symbols.
The tree grows downward and at the top (the root node) we find $\varphi$.
The tableau is then produced by unpacking formulae that occur on the
tree by following a set of rules that we give below.

\begin{center}

\begin{tabular}{cccccccccccc}
 &  &  &  &  &  &  &  &  &  &  & \tabularnewline
($\vee$) & 
\Tree
[.{$\varphi \vee \psi$}
	[.{$\varphi\:\:\:$}
	]
	[.{$\:\:\:\psi$}
	]
] &  &  &  & ($\neg\vee$) & 
\Tree
[.{$\neg(\varphi \vee \psi)$}
	[.{$\neg \varphi$\\
		$\neg \psi$}
	]
] &  &  &  & ($\neg\neg$) & 
\Tree
[.{$\neg \neg \varphi$}
	[.{$\varphi$}
	]
]\tabularnewline
 &  &  &  &  &  &  &  &  &  &  & \tabularnewline
($\exists$) & \Tree
[.{$\exists x \varphi (x)$}
	[.{$\varphi (c) $}
	]
] &  &  &  & ($\neg\exists$) & \Tree
[.{$\neg\exists x \varphi (x)$}
	[.{$\neg \varphi (a)$}
	]
] &  &  &  &  & \tabularnewline
 &  &  &  &  &  &  &  &  &  &  & \tabularnewline
\end{tabular}

\end{center}

For the $\exists$ rule the constant symbol $c$ must be new to the
tree. For the $\neg\exists$ rule the constant symbol $a$ is arbitrary.
To save some space, we shall assume that other connectives and quantifiers
are defined in the usual way, and thus, can be given the obvious derived
rules that we'll often use below.

Next we describe what happens when the unpacking process crashes.
If we ever end up in a situation where a branch in the tableau tree
has both $\psi$ and $\neg\psi$ occurring upon it, we deem that branch
to be \emph{closed} and place a $\times$ below the final sentence
node to indicate that we are finished with that branch. 

Suppose we have a tableau tree $\mathcal{T}$ for $\varphi$ before
us. We say that a branch $b$ in $\mathcal{T}$ is \emph{complete}
if $b$ is closed or every rule that can be applied to every formula
occurring on $b$ has been applied.\footnote{For the $\neg\exists$ rule we must ensure that every constant symbol
occurring on the branch is eventually used.} We say that a branch is \emph{open} if it is complete but not closed.
We then say that the entire tableau tree is \emph{closed }if all its
branches are closed; and that it is \emph{open}, if it has at least
one open branch.

\subsubsection{Soundness and completeness}

We then obtain the following version of soundness and completeness. 
\begin{thm}
$\varphi$ is satisfiable iff every completed tableau $\mathcal{T}$
for $\varphi$ is open.\label{thm:FOLs=000026c}
\end{thm}

We give a quick sketch of this as it provides a simple illustration
of what we will be doing later.
\begin{proof}
($\to$, soundness) Suppose $\varphi$ is satisfied in $\mathcal{M}$
and without loss of generality suppose that $\mathcal{M}$ is countable.
Our goal is to show that the tableau $\mathcal{T}$ for $\varphi$
has an open branch. To do this we use $\mathcal{M}$ to guide us through
$\mathcal{T}$. This is done by induction on the length of the branch.
Thus, we suppose we have an incomplete branch $b$ in $\mathcal{T}$
of length $n$ and that every sentence occurring on $b$ is satisfiable
in an expansion of $\mathcal{M}$ that interprets the new constant
symbols added via the ($\exists$) rule. We then show that the rule
whose application gives a branch $b^{*}$ in $\mathcal{T}$ of length
$n+1$ extending $b$ is such that a possible further expansion of
$\mathcal{M}$ satisfies the new formula. The interesting cases are
($\vee$) (where we may need to use $\mathcal{M}$ to make a choice)
and ($\exists$) where we add new constant symbols. At the limit,
this gives us a complete branch upon which every formula is satisfied
in an expansion of $\mathcal{M}$. 

($\leftarrow$, completeness) Suppose $b$ is an open branch in the
tableau $\mathcal{T}$ for $\varphi$. We use this branch to define
a model $\mathcal{M}=\langle M,R^{\mathcal{M}}\rangle$ where, for
simplicity, we suppose that $\mathcal{L}$ just has one $2$-place
relation symbol. We let: 
\begin{itemize}
\item $M$ be the set of constant symbols occurring in sentences on $b$;
and
\item $R^{\mathcal{M}}$ be the set of pairs $\langle c,d\rangle$ of constant
symbols such that $``Rcd"$ occurs on $b$. 
\end{itemize}
We then argue by induction on the (tableau) complexity of sentences
that every sentence occurring on $b$ is true in $\mathcal{M}$. Thus,
for example, if $\forall x\psi(x)$ occurs on $b$, we know that since
$b$ is complete that for all constant symbols $c$ occurring on $b$,
$\psi(c)$ must occur on $b$. The induction hypothesis then tells
us that $\mathcal{M}\models\psi(c)$ for all $c$ occurring on $b$;
and so $\mathcal{M}\models\forall x\psi(x)$ as required. 
\end{proof}
This gives us the completeness theorem for sentences of first order
logic without identity. 

\subsubsection{Adding identity}

To deal with identity relations, we add two new rules and a new closure
condition. The new rules are set out below. The second new rule is
essentially the indiscernibility of identicals: 

\begin{center}

\begin{tabular}{ccccc}
 &  &  &  & \tabularnewline
 & \Tree
[.{$$}
	[.{$d=d$}
	]
] &  & \Tree
[.{$\varphi(c)$\\
	$d=c$}
	[.{$\varphi (d)$}
	]
] & \tabularnewline
 &  &  &  & \tabularnewline
\end{tabular}

\end{center}

The first of these rules uses an arbitrary constant symbol. It ensures
that the identity relation is reflexive. The second rule is essentially
the indiscernibility of identicals. For the purposes of this paper,
we shall only allow use of atomic and negated atomic sentences in
this rule; thus, $\varphi(c)$ must be an atomic or negated atomic
sentence.\footnote{This will be important later on as it will help us ensure that the
complexity of sentences occurring in our modified tableau is uniformly
bounded.}

The new closure condition says that whenever $c\neq c$ occurs on
a branch, then that branch is closed. The completeness proof is then
modified by letting the domain be a set of equivalence classes of
constant symbols determined by the identity statements occurring on
the branch. The identity rules then ensure that this interpretation
of identity is congruent with the interpretation of the rest of the
non-logical vocabulary.

\subsection{Modifying the tableaux}

Our goal is to modify the tableau process described above to fit our
target application. We have two main goals:
\begin{enumerate}
\item we should be able to build models of any quantifier bounded extension
of $KPC$; and
\item the model building process should be idempotent. 
\end{enumerate}
To address (1), we shall adjust our rules so that for any $k\in\omega$,
we can build a well-founded model that has the same $\Sigma_{k}$
theory as the background universe. This will allow us to absorb any
quantifier bounded extension of $KPC$. To achieve this, we make use
of the fact that $\Sigma_{k}$ satisfaction is definable in $KP$
and use this, like the model $\mathcal{M}$ in the proof of Theorem
\ref{thm:FOLs=000026c}, to guide the construction of an open branch.
In particular, we shall make sure that there is a bound on the complexity
of sentences occurring in the tree.

Regarding (2), we need to control the theory of the model we construct
so that when we repeat the construction inside this model, we get
essentially the same branch that determined this model. For this purpose
we introduce an unusual notion of satisfiability, that we will use
to guide the definition of a model in a similar fashion to the soundness
proof for Theorem \ref{thm:FOLs=000026c}.
\begin{defn}
Let $\Phi(\bar{c})$ be a finite set of sentences from $\mathcal{L}_{\in}(C)$
where $\bar{c}=c_{0},...,c_{n}$ and let $s:\{\bar{c}\}\to Ord$ be
an assignment of ordinals to the constant symbols occurring on $\Phi(\bar{c})$.
\begin{enumerate}
\item We say that $\Phi(\bar{c})$ is \emph{rank-satisfiable with} $s$
if there exists $\bar{a}=a_{0},...,a_{n}$ from the universe, $V$,
with $rank(a_{i})=s(c_{i})$ for all $i\leq n$ such that $\bigwedge\Phi(\bar{a})$
is true. We say that $\bar{a}$ is a \emph{realizer }of the rank-satisfiability
of $\Phi$ by $s$.
\item We say that $\Phi(\bar{c})$ is \emph{minimal-rank-satisfiable with
$s$ }if $\Phi(\bar{c})$ is rank-satisfiable with $s$ and there
is no $t:\{\bar{c}\}\to Ord$ lexicographically preceding $s$ such
that $\Phi(\bar{c})$ is rank-satisfiable with $t$. 
\end{enumerate}
Whenever assignments $s,t:\{c_{i_{0}},\ldots,c_{i_{m}}\}\to\mathrm{Ord}$
are compared lexicographically, we assume $i_{0}<\cdots<i_{m}$ and
compare the tuples $\langle s(c_{i_{0}}),\ldots,s(c_{i_{m}})\rangle$
and $\langle t(c_{i_{0}}),\ldots,t(c_{i_{m}})\rangle$.
\end{defn}

Less formally, $\Phi(\bar{c})$ is rank-satisfiable with some $s:\{\bar{c}\}\to Ord$
if there is a way of assigning objects to the constant symbols occurring
in sentences on $\Phi(\bar{c})$ whose ranks correspond to $s$ and
which satisfies each sentence in $\Phi(\bar{c})$. $\Phi(\bar{c})$
is then minimal-rank-satisfiable if it cannot be rank-satisfied using
an ordinal assignment that is lexicographically prior to $s$.

As before, we'll construct our tableau using sentences from the language
$\mathcal{L}_{\in}(C)$. Let us then note a few notational conventions.
We adopt the convention of writing $\varphi(\bar{x},\bar{c})$ to
mean that: $\varphi$ has at most the variables $\bar{x}=x_{0},...,x_{n}$
free; and that $\varphi$ contains at most the constant symbols $\bar{c}=c_{0},...,c_{m}$.
To save space and clutter on the page, we'll avoid being too specific
about how matters like substitution work leaving most of this to the
reader's imagination. However, at places where confusion seems likely
to arise we'll try to flag this and discuss how this might be resolved.

By $\bar{y}<_{\text{rank-lex}}\bar{d}$ we mean that $\langle rank(y_{0}),...,rank(y_{n})\rangle$
precedes $\langle rank(d_{0}),...,rank(d_{n})\rangle$ in the lexicographic
order. The natural adaptation of this definition also gives us $\leq_{\text{rank-lex}}$
and $=_{\text{rank-lex}}$. Note that each of these is easily seen
to be $\Delta_{1}$ in $KP$. While we are introducing this as a definition
in our background theory, we shall also freely use this notation as
a metalinguistic abbreviation in the sentences of $\mathcal{L}_{\in}(C)$
occurring in the tableau. This will be helpful in ensuring idempotence.
To save a little space on the page, we shall frequently write $\forall\bar{y}<_{\text{rank-lex}}\bar{d}\ \varphi(\bar{y},\bar{c})$
as a metalinguistic abbreviation for $\forall\bar{y}(\bar{y}<_{\text{rank-lex}}\bar{d}\to\varphi(\bar{y},\bar{c}))$.

We use the prenex and complexity conventions fixed in Section \ref{subsec:Solidity-failure}.
Write $\varphi^{\text{prenex}}$ for the result of applying a fixed
computable procedure that expands bounded quantifiers with their restrictions
as necessary and moves the quantifiers to the front by ordinary logical
equivalences. Thus $\varphi^{\text{prenex}}$ is logically equivalent
to $\varphi$. Bounded quantifiers before the last unbounded quantifier
contribute to the alternating-block count; no closure under arbitrary
bounded quantification is assumed. We shall use this procedure on
the sentences occurring in the tableau when calculating their complexity.

For our first modification, our new tableau will incorporate \emph{side
conditions} on the sentences occurring in the tree.\emph{ }Nodes on
the tree will, thus, be of the form $\langle\varphi(\bar{c}),s\rangle$
where $\varphi$ is a sentence of $\mathcal{L}_{\in}(C)$ using the
constant symbols $\bar{c}=c_{i_{0}},....,c_{i_{m}}$ and $s:\{c_{i_{0}},...,c_{i_{m}}\}\to Ord$.
We call $s$ an \emph{ordinal assignment}. The motivating idea here
is that $s$ will provide an assignment of ranks to the constant symbols
from which we shall build our model. Among other things, this will
ensure that the model is well-founded. When writing out the tableau,
we shall, however, save some space and only notate ordinal assignments
when we add new constant symbols to a branch, as in, say, an application
of the $\exists$ rule. We shall then speak of the \emph{accumulated
ordinal assignment }which is essentially the union of the ordinal
assignments that have occurred up to the appearance of some sentence.

In order to ensure the idempotence property holds, we also need to
modify the rules of our tableau. In particular, we shall introduce
a new cut rule and change the $\vee$ and $\exists$ rules. These
are rules that allow us to choose different ways of extending a branch.
The modifications are intended to ensure that the theory of the model
built from the tree is strong enough to ensure that we will get the
same open branch if we repeat the process.

\subsubsection{Modified rules}

Next we describe how each of the modified rules works and make some
observations that will assist us in applying them. First, we leave
the $\neg\vee,\neg\neg$ and identity rules unchanged. We modify the
$\neg\exists$ rule by demanding that if a negated block of $\exists$
quantifiers occurs at the front of some sentence, we must apply the
old $\neg\exists$ rule repeatedly until there are no more negated
$\exists$ quantifiers at the front of the sentence.\footnote{For example, to apply the modified $\neg\exists$ rule to $\neg\exists x_{1}...\exists x_{n}\psi(x_{1},...,x_{n})$
where $\psi(x_{1},...,x_{n})$ does not begin with an $\exists$ quantifier,
we would need to use a sentence of the form $\neg\psi(d_{1},...,d_{n})$
where $d_{1},...,d_{n}$ are arbitrary.} We call the rules we've discussed so far, \emph{ordinary rules}.
With regard to $\vee$ and $\exists$ a more complex modification
is required, which we describe below. In addition to this, we shall
add a modified cut rule to ensure that we can preserve truths from
the background universe. We call the rules that aren't ordinary, \emph{special
rules}.

In order to keep the complexity of sentences occurring on the tableau
under control we introduce a technical definition that partitions
sentences into \emph{major} and \emph{minor} sentences. Any output
of one of the ordinary rules will give us a major sentence. However,
the special rules below produce both major and minor sentences. In
the construction that is described below, we shall see that minor
sentences contain a kind of redundant information that somewhat obviously
holds of the model in which we are building the tableau, but which
we also wish to ensure holds in the model determined by the tableau.

\subsubsection*{Special $\Sigma_{k}$-cut.}

Let us fix some $k>1$. In a tableau system, the standard cut rule
tells us that for any sentence $\varphi$ we can extend a branch by
splitting into two paths with $\varphi$ on one side and $\neg\varphi$
on the other.\footnote{See \citep{TolTab} or \citep{MeadowsCarn} for examples of cut rules
in tableau systems.} However, in our situation, we don't want to consider arbitrary sentences
but those below a specific complexity bound. Thus, we shall use what
we might call a $\Sigma_{k}$-cut rule that considers arbitrary $\Sigma_{k}$
sentences. This will allow us to use a definable bounded satisfaction
relation to navigate the tableau.

Suppose we have a $\Sigma_{k}$ sentence from $\mathcal{L}_{\in}$
(so it contains no constant symbols from $C$) of the form
\[
\exists\bar{x}\underbrace{\varphi(\bar{x})}_{\Pi_{k-1}}.
\]
Then we may apply the following cut rule:

\Tree
[.{$$}
	[.{$\varphi(\bar d),s$\\
	   $\bigl[\forall \bar y <_{\text{rank-lex}} \bar d \:\:\neg \varphi (\bar y)\bigr]^{\text{prenex}}$}
	]
	[.{$\forall \bar x \neg \varphi(\bar x)$}
	]
]

where: 
\begin{itemize}
\item on the left, $\bar{d}=d_{0},...,d_{n}$ must be the next constant
symbols not already used on the branch; and 
\item $s:\{d_{0},...,d_{n}\}\to Ord$.
\end{itemize}
We say that $\varphi(\bar{d})$ is a major sentence and that $\bigl[\forall\bar{y}<_{\text{rank-lex}}\bar{d}\ \neg\varphi(\bar{y})\bigr]^{\text{prenex}}$
and $\forall\bar{x}\neg\varphi(\bar{x})$ are minor sentences.

Observe that we are eliminating all of the existential quantifiers
at the front of the sentence, not just the first one. Note that the
diagram for the rule is also a little misleading in that it looks
like there are just two choices. In fact, the left move is an abbreviation
for a proper class of different options according to our choice of
$s:\{\bar{d}\}\to Ord$. Also observe that: $\varphi(\bar{d})$ is
$\Pi_{k-1}$; $\forall\bar{x}\neg\varphi(\bar{x})$ is $\Pi_{k}$
and $\bigl[\forall\bar{y}<_{\text{rank-lex}}\bar{d}\ \neg\varphi(\bar{y})\bigr]^{\text{prenex}}$
is $\Pi_{k}$. For the last of these, we rely on the fact that we
have converted this entire sentence into prenex normal form.

The motivation behind the rule may be described as follows. In applying
the rule, we are considering whether or not to add $\exists\bar{x}\varphi(\bar{x})$.
If we take the right move, then we are choosing not to make this sentence
true. If we take one of the left moves, we are choosing to make it
true and so we witness with a sequence $\bar{d}$ of constant symbols
and an ordinal assignment $s$ of ranks to those constants. But more
than that, we add a second sentence which says that there is no way
of satisfying $\varphi(\bar{x})$ with a sequence of objects whose
ranks preceded those of $\bar{d}$. This will help us to ensure that
we get the same branch when we repeat the model building process.\footnote{Note that we are yet to describe this process. This will come with
Lemma \ref{lem:DefBranch}.}

\subsubsection*{Special $\vee$}

Next we consider a modification of the $\vee$ rule.

\Tree
[.{$\varphi(\bar c) \vee \psi(\bar c)$}
	[.{$\varphi(\bar c)$\\
		$\bigl[\forall \bar y <_{\text{rank-lex}} \bar d \:\:\neg (\psi(\bar y) \wedge \bigwedge \Psi(\bar y))\bigr]^{\text{prenex}}$}
	]
	[.{$\psi(\bar c)$\\
		$\bigl[\forall \bar y \leq_{\text{rank-lex}} \bar d \:\:\neg (\varphi(\bar y) \wedge \bigwedge \Psi(\bar y))\bigr]^{\text{prenex}}$}
	]
]

where:
\begin{itemize}
\item $\bar{d}=d_{0},...,d_{n}$ are the constant symbols that have already
occurred on the branch and which, thus, include $\bar{c}$; and
\item $\Psi(\bar{d})$ is the set of major sentences already occurring on
the branch.
\end{itemize}
We then say that $\varphi(\bar{c})$ and $\psi(\bar{c})$ are major
sentences while $\ensuremath{\bigl[\forall\bar{y}<_{\text{rank-lex}}\bar{d}\:\:\neg(\psi(\bar{y})\wedge\bigwedge\Psi(\bar{y}))\bigr]^{\text{prenex}}}$
and $\bigl[\forall\bar{y}\leq_{\text{rank-lex}}\bar{d}\:\:\neg(\varphi(\bar{y})\wedge\bigwedge\Psi(\bar{y}))\bigr]^{\text{prenex}}$
are minor sentences.

This is the same as the standard $\vee$ rule except for the extra
sentences below $\varphi(\bar{c})$ and $\psi(\bar{c})$. The reader
will also see that this rule is incredibly inefficient in that when
we apply the rule, we are forcing ourselves to essentially repeat
a lot of work that has already been done. Fortunately for our purposes,
this kind of efficiency is not important.

The motivating idea is to ensure that:
\begin{itemize}
\item when we turn left, there is no way to satisfy $\psi(\bar{y})$ with
an assignment sequence whose ranks lexicographically precede those
of the assignment for $\varphi(\bar{c})$ and the sentences already
occurring on the branch; and
\item when we turn right, there is no way to satisfy $\varphi(\bar{y})$
with an assignment sequence whose ranks lexicographically precede
or equal those of the assignment for $\psi(\bar{c})$ and the sentences
already occurring on the branch.
\end{itemize}
Note that if $\varphi(\bar{x})$ and all the sentences occurring in
$\Psi(\bar{y})$ (i.e., the major sentences already occurring on the
branch) are $\Sigma_{k}$ (or less), then $\bigl[\forall\bar{y}\leq_{\text{rank-lex}}\bar{d}\ \neg(\varphi(\bar{y})\wedge\bigwedge\Psi(\bar{y}))\bigr]^{\text{prenex}}$
is $\Pi_{k}$.

\subsubsection*{Special $\exists$}

Finally, we consider a modification of the $\exists$ rule. 

\Tree
[.{$\exists\bar{x}\ \varphi(\bar{x},\bar{c}_0)$}
	[.{$\varphi(\bar d, \bar c), s$\\
	   $\bigl[\forall \bar y,\bar z <_{\text{rank-lex}} \bar c, \bar d \:\:\neg (\varphi (\bar z, \bar y)\wedge \bigwedge \Psi(\bar y))\bigr]^{\text{prenex}}$}
	]
]

where: 
\begin{itemize}
\item $\bar{c}$ are the constant symbols that have already occurred on
the branch;\footnote{Note that this does not mean that every constant symbol from $\bar{c}$
is used in $\varphi(\bar{x},\bar{c})$.}
\item $\Psi(\bar{c})$ is the set of major sentences already occurring on
the branch;
\item $\bar{d}$ must be the next constant symbols not already used on the
branch; and
\item $s:\{\bar{d}\}\to Ord$.
\end{itemize}
We say that $\varphi(\bar{d},\bar{c})$ is a major sentence and that
$\bigl[\forall\bar{y},\bar{z}<_{\text{rank-lex}}\bar{c},\bar{d}\:\:\neg(\varphi(\bar{z},\bar{y})\wedge\bigwedge\Psi(\bar{y}))\bigr]^{\text{prenex}}$
is a minor sentence. Again note that the diagram for the rule is a
little misleading: there is a proper class of different options according
to the different ordinal assignments $s$ that go along with $\varphi(\bar{d},\bar{c})$.
Observe that this is the same as the ordinary $\exists$ rule except
that we add an extra sentence which is intended to ensure that there
can be no sequence $\bar{y},\bar{z}$ whose ranks precede those of
$\bar{c},\bar{d}$ that satisfies $\varphi(\bar{z},\bar{y})$ along
with every major sentence already occurring on the branch. Note that
if $\varphi(\bar{y},\bar{c})$ and every sentence in $\Psi(\bar{c})$
are $\Sigma_{k}$ (or less) then 
\[
\bigl[\forall\bar{y},\bar{z}<_{\text{rank-lex}}\bar{c},\bar{d}\:\:\neg(\varphi(\bar{z},\bar{y})\wedge\bigwedge\Psi(\bar{y}))\bigr]^{\text{prenex}}
\]
is $\Pi_{k}$.

\subsubsection{An important lemma}

The following lemma establishes that we can put a quantifier bound
on the complexity of sentences occurring in our tableau. This means
that we can use a definable satisfaction predicate to guide the definition
of an open branch.
\begin{lem}
Suppose $\mathcal{T}$ is a (possibly incomplete) tableau produced
using the system described above. Then every minor sentence on $\mathcal{T}$
is $\Pi_{k}$ and every major sentence is $\Sigma_{k}$.
\end{lem}

\begin{proof}
We proceed by induction on the length of branches to show that: 
\begin{enumerate}
\item the complexity of any major sentence is $\Sigma_{k}$; and
\item the complexity of every minor sentence is $\Pi_{k}$.
\end{enumerate}
For the $\neg\vee,\neg\exists$ and $\neg\neg$ rules this is obvious:
applying those rules always adds sentences of the same or lower complexity
than the sentence to which those rules are applied. This suffices
since it can be seen that every sentence on $\mathcal{T}$ is in prenex
normal form and so we can only apply the $\neg\vee$ and $\neg\neg$
rules to $\Delta_{0}$-sentences. The $\neg\exists$ rule can apply
to sentences more complex than $\Delta_{0}$; for instance a $\neg\exists$
formula may result from taking the right move on a Special $\Sigma_{k}$-cut
rule and thus, be $\Pi_{k}$. The sentence resulting will always be
$\Sigma_{k-1}$ which suffices. For the identity rules, note that
we are only allowed to use them on atomic and negated atomic sentences.
So we just need to consider the special rules described above.

(Special $\Sigma_{k}$-cut) Supposing we cut on $\exists\bar{x}\varphi(\bar{x})$
where $\varphi(\bar{x})$ is $\Pi_{k-1}$, we see that: the major
sentence $\varphi(\bar{d})$ is $\Pi_{k-1}$ and thus $\Sigma_{k}$;
the minor sentences $\forall\bar{x}\neg\varphi(\bar{x})$ and $\bigl[\forall\bar{y}<_{\text{rank-lex}}\bar{d}\ \neg\varphi(\bar{y})\bigr]^{\text{prenex}}$
are both $\Pi_{k}$.

(Special $\vee$) It is not difficult to see that our rules ensure
that every sentence occurring in $\mathcal{T}$ is in prenex normal
form. This means that the only way for a disjunction, $\varphi(\bar{c})\vee\psi(\bar{c})$,
to occur in $\mathcal{T}$ is if both $\varphi(\bar{c})$ and $\psi(\bar{c})$
are $\Delta_{0}$. Thus, the major sentences are both $\Sigma_{k}$.
Moreover, by the induction hypothesis, we see that the set $\Psi(\bar{d})$
of major sentences occurring on the branch so far is a conjunction
of $\Sigma_{k}$ sentences. Thus, the minor sentences can both be
seen to be $\Pi_{k}$.

(Special $\exists$) Similarly, since every sentence occurring in
$\mathcal{T}$ is in prenex normal form, we see that the only way
for a sentence of the form $\exists\bar{x}\ \varphi(\bar{x},\bar{c})$
to occur on $\mathcal{T}$ is if it is $\Sigma_{k}$. Using the induction
hypothesis, we see that every sentence in $\Psi(\bar{c})$ is $\Sigma_{k}$,
and so the prenex normal form of
\[
\forall\bar{y},\bar{z}<_{\text{rank-lex}}\bar{c},\bar{d}\:\:\neg(\varphi(\bar{z},\bar{y})\wedge\bigwedge\Psi(\bar{y}))
\]
will be $\Pi_{k}$ as required.
\end{proof}

\subsubsection{A standard process}

\label{subsec:StanProc}

Before we move on to describing the model-building process and establishing
that it is idempotent, let us first suppose that we have fixed upon
a particular order in the applications of our rules. The purpose of
this is to give us a particular canonical tableau $\mathcal{T}$ that
we can refer to below. Similar processes can be found in many sources,
so we content ourselves with a quick sketch \citep{SmullFO}.

For one way of doing this, we essentially just work our way down a
branch applying the next rule that needs to be applied. This works
for rules that only require one application, like $\neg\vee$, Special
$\vee$ and Special $\exists$. However, for the Special $\exists$
rule let us suppose that we always add the next sequence $\bar{d}=d_{0},...,d_{m}$
of elements from $C$ to witness that $\exists\bar{x}\ \varphi(\bar{x},\bar{c})$.
Thus, if we've already used $\{c_{0},...,c_{n}\}$ on a particular
branch, then we let $\bar{d}=c_{n+1},...,c_{n+m+1}$.

For the identity and $\neg\exists$ rules, we need to make infinitely
many applications. So after every application of the other rules,
we should go back to the beginning of the branch and do the next application
of each identity and $\neg\exists$ rule before moving on. In applications
of the $\neg\exists$ rule and the reflexive identity rule, we only
allow the substitution of constant symbols that have already occurred
on the branch; and we demand that every such constant symbol is eventually
used.\footnote{It is not difficult to see that an open branch, using every instance
of $\Sigma_{k}$-cut will end up using every constant symbol in $C$.
This is because there are infinitely many true $\Sigma_{k}$-sentences
beginning with existential quantifiers.} After doing that we still need to deal with Special $\Sigma_{k}$-cut.
For this rule, we suppose we have a simple enumeration of the $\Sigma_{k}$
sentences and so, after doing the next round of applications to each
identity and $\neg\exists$ rule, we then do the next Special $\Sigma_{k}$-cut.
Finally, fix once and for all an initial branch (say the empty branch)
and a primitive-recursive fair scheduling procedure which, from the
code of a finite branch, determines the next pending ordinary-rule
application, the next identity or $\neg\exists$ instance, and the
next $\Sigma_{k}$ cut. The tableau generated using this fixed schedule
will be called the \emph{canonical} \emph{tableau} that we shall denote
as $\mathcal{T}$.

\subsection{Building models }

Let $T$ be a theory as described in the statement of Lemma \ref{count-enum-thm}.
Thus, we may fix some $k>3$ such that for all sentences $\varphi\in T\backslash KPC$,
$\varphi$ is $\Sigma_{k}$. For the remainder of this section, we
work in the theory $T$ unless stated otherwise. This means, for example,
if we say that some relation is well-founded, it is well-founded according
to $T$.\footnote{Moreover, since we could be using a theory as weak as $KP$, when
we say that a relation is well-founded, we mean that there is an order-preserving
relation from that relation into the ordinals.}

Let $\mathcal{T}$ be the canonical tableau using Special $\Sigma_{k}$-cut
described above. Our goal is to use minimal-rank-satisfiability as
a guide to defining an open branch through $\mathcal{T}$ and then
use that branch to interpret a well-founded model that is $\Sigma_{k}$
elementary equivalent to our background universe.

\subsubsection{Soundness}

We are ready to describe a process for finding an open branch through
$\mathcal{T}$. We start by making our lives a little easier with
a sequence of three little claims showing that it will actually suffice
to merely preserve rank-satisfiability (rather than minimal-rank-satisfiability).
First, we observe that the truth of minor sentences only depends on
the ranks of their constant symbols.
\begin{prop}
Suppose $\chi(\bar{c})$ is of the form of one of the minor sentences.
If $\bar{a}$ and $\bar{e}$ are such that $rank(a_{i})=rank(e_{i})$
for all $i<lh(\bar{a})$ then\label{claim:rank-invar} 
\[
\chi(\bar{a})\Leftrightarrow\chi(\bar{e}).
\]
\end{prop}

\begin{proof}
Observe that minor sentences are only concerned with the ranks of
the constants they mention. For example,
\[
\forall\bar{y},\bar{z}<_{\mathrm{rank\text{-}lex}}\bar{c},\bar{d}\ \neg\left(\varphi(\bar{z},\bar{y})\wedge\bigwedge\Psi(\bar{y})\right)
\]

is clearly satisfied by any interpretations $\bar{a},\bar{e}$ for
$\bar{c},\bar{d}$ with the same rank.
\end{proof}
Now we shift our attention to branches in our tableau. We use the
proposition above to show that if all the sentences on some branch
are rank-satisfied with their accumulated ordinal assignment, then
any realizer of its major sentences will also realize its minor sentences.
\begin{prop}
Let $b$ be a branch in $\mathcal{T}$, $\Phi(\bar{c})$ be the sentences
occurring on $b$; and $\Psi(\bar{c})$ be the major sentences occurring
on $b$. Suppose that $\Phi(\bar{c})$ is rank-satisfiable with $s:\{\bar{c}\}\to Ord$.
Then every realizer $\bar{a}$ for $\Psi(\bar{c})$ being rank satisfiable
with $s$ is also a realizer for the rank-satisfiability of $\Phi(\bar{c})$
with $s$.\label{realizerShifting}
\end{prop}

\begin{proof}
Let $\bar{e}$ be a realizer for the rank-satisfiability of $\Phi(\bar{c})$
with $s$. Then suppose that $\bar{a}$ is a realizer for the rank-satisfiability
of $\Psi(\bar{c})$ at $s$. We claim that $\bar{a}$ is also a realizer
for the satisfiability of $\Phi(\bar{c})$ with $s$. To see this,
let $\chi(\bar{c})$ be a minor sentence occurring on $b$. By our
assumption, $\chi(\bar{e})$ is true and so Proposition \ref{claim:rank-invar}
tells us that $\chi(\bar{a})$ is also true; and thus, we see that
$\bar{a}$ is a realizer for the rank-satisfiability of $\Phi(\bar{c})$
with $s$.
\end{proof}
We then put this together to show that for branches in a tableau,
rank-satisfiability implies minimal-rank-satisfiability.
\begin{lem}
Suppose that $b$ is a branch in $\mathcal{T}$. Let $\Phi(\bar{c})$
be the set of sentences occurring on $b$; let $\Psi(\bar{c})$ be
the set of major sentences occurring on $b$; and let $s$ be the
accumulated ordinal assignment. Suppose that $\Phi(\bar{c})$ is rank-satisfiable
with $s$. Then\label{lem:AllrankToAllmin}
\begin{enumerate}
\item $\Psi(\bar{c})$ is minimal-rank-satisfiable with $s$; and
\item $\Phi(\bar{c})$ is minimal-rank-satisfiable with $s$.
\end{enumerate}
\end{lem}

\begin{proof}
We will prove (1) by induction on the length of branches. To see
that (2) follows, fix using (1), a realizer $\bar{a}$ for $\Psi(\bar{c})$
being minimal-rank-satisfiable with $s$. Then by Proposition \ref{realizerShifting},
we see that $\bar{a}$ is also a realizer for $\Phi(\bar{c})$ being
rank-satisfiable. Then since $\bar{a}$ is minimal for $\Psi(\bar{c})$
with $s$ and $\Psi(\bar{c})\subseteq\Phi(\bar{c})$, it must also
be minimal for $\Phi(\bar{c})$ as required.

We proceed by induction on the length of branches. Suppose that we
have established that (1) holds for $b$ and wish to show that (1)
still holds if we extend the branch to a new position $b^{*}$ in
$\mathcal{T}$ by following the rules. 

Suppose first that the rule used to get to $b^{*}$ from $b$ was
one that adds a formula that introduces no new constant symbols (i.e.,
an ordinary rule, Special $\vee$, or the right move in Special $\Sigma_{k}$-cut).
Let $\Phi^{*}(\bar{c})$ and $\Psi^{*}(\bar{c})$ be respectively
the sets of all and merely major sentences occurring on $b^{*}$;
and let $\Phi(\bar{c})$ and $\Psi(\bar{c})$ be the corresponding
sets for $b$. Let $s$ be the accumulated ordinal assignment for
$b$ and, thus, $b^{*}$ too. 

Then suppose that $\Phi^{*}(\bar{c})$ is rank-satisfiable with $s$.
Then $\Phi(\bar{c})$ is also rank-satisfiable with $\bar{s}$ and
so the induction hypothesis tells us that $\Psi(\bar{c})$ is minimal-rank-satisfiable
with $s$. But then since $\Psi^{*}(\bar{c})\supseteq\Psi(\bar{c})$,
we see that $\Psi^{*}(\bar{a})$ cannot be rank-satisfied by any $s_{0}$
preceding $s$.

Suppose next that the rule used to get from $b$ to $b^{*}$ does
add new constant symbols. Thus, it must be the left move of Special
$\Sigma_{k}$-cut or Special $\exists$. Let $\Phi^{*}(\bar{c},\bar{d})$
and $\Psi^{*}(\bar{c},\bar{d})$ be respectively the sets of all and
merely major sentences occurring on $b^{*}$ where $\bar{d}$ are
the new constant symbols; and let $\Phi(\bar{c})$ and $\Psi(\bar{c})$
be the corresponding sets for $b$. Let $s_{0}$ be the accumulated
ordinal assignment for $b$ and $s_{0}\cup s$ be the accumulated
ordinal assignment for $b^{*}$.

(Special $\Sigma_{k}$-cut, left move) Suppose that $b^{*}$ is extended
from $b$ using
\begin{align*}
 & \varphi(\bar{d}),s\\
 & \forall\bar{y}<_{\text{rank-lex}}\bar{d}\ \neg\varphi(\bar{y}).
\end{align*}
Now suppose that $\Phi^{*}(\bar{c},\bar{d})$ is rank-satisfiable
with $s_{0}\cup s$ as realized by $\bar{a},\bar{e}$. Suppose toward
a contradiction that $\Psi^{*}(\bar{c},\bar{d})$ is not minimal-rank-satisfiable
with $s_{0}\cup s$. Then we may fix $\bar{o},\bar{u}<_{\text{rank-lex}}\bar{a},\bar{e}$
such that $\Psi^{*}(\bar{o},\bar{u})$ is satisfiable. We break into
two cases. First, if $\bar{o}<_{\text{rank-lex}}\bar{a}$ then, $\Psi(\bar{o})$
is true, contradicting the induction hypothesis that told us $\Psi(\bar{c})$
was minimal-rank-satisfiable with $s_{0}$. Second, if $\bar{o}=_{\text{rank-lex}}\bar{a}$
and $\bar{u}<_{\mathrm{rank\text{-}lex}}\bar{e}$, then $\forall\bar{y}<_{\text{rank-lex}}\bar{d}\ \neg\varphi(\bar{y})$
is false, which contradicts our assumption that $\Phi(\bar{c},\bar{d})$
is rank-satisfiable with $s_{0}\cup s$.

(Special $\exists$) Suppose $b$ is extended to $b^{*}$ by adding
\begin{align*}
 & \varphi(\bar{d},\bar{c}),s\\
 & \bigl[\forall\bar{y},\bar{z}<_{\text{rank-lex}}\bar{c},\bar{d}\:\:\neg(\varphi(\bar{z},\bar{y})\wedge\bigwedge\Psi(\bar{y}))\bigr]^{\text{prenex}}.
\end{align*}
Now suppose that $\Phi^{*}(\bar{c},\bar{d})$ is rank-satisfiable
with $s_{0}\cup s$ as realized by $\bar{a},\bar{e}$. Then suppose
toward a contradiction that $\Psi^{*}(\bar{c},\bar{d})$ is not minimal-rank-satisfiable
with $s_{0}\cup s$. Then we may fix $\bar{o},\bar{u}<_{\text{rank-lex}}\bar{a},\bar{e}$
such that $\Psi^{*}(\bar{o},\bar{u})$ is satisfiable. We break into
two cases. First, if $\bar{o}<_{\text{rank-lex}}\bar{a}$ then, $\Psi(\bar{o})$
is true, contradicting the induction hypothesis that told us $\Psi(\bar{c})$
was minimal-rank-satisfiable with $s_{0}$. Second, if $\bar{o}=_{\text{rank-lex}}\bar{a}$
and $\bar{u}<_{\text{rank-lex}}\bar{e}$, then $\forall\bar{y},\bar{z}<_{\text{rank-lex}}\bar{c},\bar{d}\:\:\neg(\varphi(\bar{z},\bar{y})\wedge\bigwedge\Psi(\bar{y}))$
is false, contradicting our assumption that $\Phi^{*}(\bar{c},\bar{d})$
is rank satisfiable with $s_{0}\cup s$. 
\end{proof}
Finally, we are ready to build our open branch. Since the process
is inductive, we focus first on the process of extending a particular
branch. First a quick technical definition: let us say that a finite
branch $b$ in $\mathcal{T}$ is (\emph{minimal}-)\emph{rank-satisfiable},
if the sentences on $b$ are (minimal-)rank-satisfiable with their
accumulated rank-assignment.
\begin{lem}
(Branch Extension) Suppose $b$ is a finite branch in $\mathcal{T}$
that is minimal rank-satisfiable. Then there is a way of selecting
an extension $b^{*}$ in $\mathcal{T}$ of $b$ that is also minimal-rank-satisfiable.
Moreover, $b^{*}$ is uniformly $\Delta_{k+1}$ definable from $b$.\label{lem:BranchExtension}
\end{lem}

\begin{proof}
We must consider what happens when each of the rules is applied. If
$b$ is extended with an application of an ordinary rule minimal-rank-satisfaction
is obviously preserved. So we shall just consider: Special $\Sigma_{k}$-cut,
Special $\vee$; and Special $\exists$. We shall make use of a uniformly
$\Delta_{k+1}$-definable satisfaction relation which is able to assess
the truth of Boolean combinations of $\Sigma_{k}$-formulae. We also
note that Lemma \ref{lem:AllrankToAllmin} makes our life easier by
establishing that we only need to show that the extension of $b$
is rank-satisfiable.

(Special $\Sigma_{k}$-cut) Suppose $b$ is extended to $b^{*}$ by
an application of the Special $\Sigma_{k}$-cut and that $\exists\bar{x}\varphi(\bar{x})$
is the cut-formula where $\varphi(\bar{x})$ is $\Pi_{k-1}$. If $\exists\bar{x}\varphi(\bar{x})$
is true, we extend $b$ to $b^{*}$ by taking the left move which
adds 
\begin{align*}
 & \varphi(\bar{d}),s\\
 & \forall\bar{y}<_{\text{rank-lex}}\bar{d}\ \neg\varphi(\bar{y})
\end{align*}
to $b$, where $s:\{\bar{d}\}\to Ord$ is the lexicographically least
assignment such that there is some $\bar{a}=a_{0},...,a_{n}$ from
$V$ with $rank(a_{i})=s(d_{i})$ for $i\leq n$ such that $\varphi(\bar{a})$
is true. The formula identifying such $s$ is $\Delta_{k+1}$. Indeed,
existence of a realizer with a prescribed rank profile is $\Sigma_{k}$,
and taking the lexicographically least such profile is therefore $\Delta_{k+1}$.
It is then clear that the extension of $b$ with the major sentence
$\varphi(\bar{d})$ is minimal-rank-satisfiable with $s$. Moreover,
our choice of $s$ for $\bar{d}$ ensures that $\forall\bar{y}<_{\text{rank-lex}}\bar{d}\ \neg\varphi(\bar{y})$
is satisfied, which means that $b^{*}$ is rank-satisfiable. If $\exists\bar{x}\varphi(\bar{x})$
is not true, we simply extend $b$ to $b^{*}$ by taking the right
move.

(Special $\vee$) Suppose $b$ is extended by an application of the
Special $\vee$ rule to $\varphi(\bar{c})\vee\psi(\bar{c})$. Then
by our assumptions, we see that $b$ is minimal-rank-satisfiable,
as witnessed by an ordinal assignment $s:\{\bar{d}\}\to Ord$ where
$\bar{d}$ are the constant symbols occurring in sentences on $b$.
Let $\Psi(\bar{d})$ be the set of major sentences occurring on $b$
and observe that the constant symbols $\bar{c}$ are included in $\bar{d}$.
Then clearly one of the extensions of $b$ by $\varphi(\bar{c})$
or $\psi(\bar{c})$ is rank-satisfiable, as witnessed by $s$. Let
$b^{*}$ be the extension of $b$ that adds the \emph{leftmost} of
those sentences.\footnote{Note that if the addition of either $\varphi(\bar{c})$ or $\psi(\bar{c})$
would result in a branch that is rank satisfied using $s$, we break
that tie by choosing to go left.} Our selection process then ensures that the minor sentence is satisfied
by a sequence of objects whose ranks accord with the ordinal assignment
for $\bar{d}$. If we turn left, the minor sentence follows from the
minimality of major sentences in $\Psi$. Similarly, if we turn right,
except that we have incorporated a tie-breaking condition. 

(Special $\exists$) Suppose $b$ is extended by an application of
the Special $\exists$ rules to $\exists\bar{x}\ \varphi(\bar{x},\bar{c})$.
Suppose that $s_{0}:\{\bar{c}\}\to Ord$ is the accumulated ordinal
assignment for $b$. Then we extend $b$ to $b^{*}$ by adding 
\begin{align*}
 & \varphi(\bar{d},\bar{c}),s_{1}\\
 & \bigl[\forall\bar{y},\bar{z}<_{\text{rank-lex}}\bar{c},\bar{d}\:\:\neg(\varphi(\bar{z},\bar{y})\wedge\bigwedge\Psi(\bar{y}))\bigr]^{\text{prenex}}
\end{align*}
below $b$ where: $\Psi(\bar{c})$ is the conjunction of the major
sentences occurring on $b$; and $s_{1}:\{\bar{d}\}\to Ord$ is the
lexicographically least assignment such that there exist $\bar{a}=a_{0},...,a_{n}$
and $\bar{e}=e_{0},...,e_{m}$ from $V$ with $rank(a_{i})=s_{0}(c_{i})$
for all $i\leq n$ and $rank(e_{i})=s_{1}(d_{i})$ for $i\leq m$
such that
\[
\varphi(\bar{e},\bar{a})\wedge\bigwedge\Psi(\bar{a})
\]
is true. Proposition \ref{realizerShifting} ensures that the minor
sentences on $b$ are also satisfied by this assignment, so it should
be clear that $b\cup\{\varphi(\bar{d},\bar{c})\}$ is rank-satisfiable
with $s\cup s_{1}$. Finally, it is clear from the induction hypothesis
and the way we selected $\bar{e},\bar{a}$ and $s_{1}$ that the minor
sentence is also satisfied by $\bar{a}$ and $\bar{e}$ and thus,
$b^{*}$ is rank-satisfiable\@.
\end{proof}
We may then repeatedly use this lemma to define an open branch.
\begin{lem}
There is a $\Delta_{k+1}$ definable open branch $b_{0}$ through
$\mathcal{T}$.\label{lem:DefBranch}
\end{lem}

\begin{proof}
Define $b_{0}$ inductively using Lemma \ref{lem:BranchExtension}.
At each stage use the next rule prescribed by the canonical schedule
described in Section \ref{subsec:StanProc} and the preferred extension
supplied by Lemma \ref{lem:BranchExtension}. Every finite initial
segment is rank-satisfiable, so the resulting branch is not closed;
fairness of the schedule makes it complete, and hence open. This gives
a $\Sigma_{k+1}$ definition and since it is unique, it can also be
given a $\Pi_{k+1}$ definition. Note that $b_{0}$ is defined as
a class, not a set.
\end{proof}
Let us call $b_{0}$ the \emph{canonical} \emph{branch} of $\mathcal{T}$. 

\subsubsection{Completeness}

We now describe how to use the canonical branch $b_{0}$ through $\mathcal{T}$
to define a model satisfying all the true $\Sigma_{k}$ sentences
of our background universe. The definition is essentially the same
as that used to build models from tableaux for first order logic.
However, we need to take a little more care since, in general, we
are merely defining the model as a class (or interpretation) rather
than a set. In particular, we shall work with pre-structures rather
than genuine structures. Thus, we shall work with a defined congruence
relation $\approx$ rather than genuine identity.\footnote{This will allow us to make the domain of our defined models be the
entirety of our set of constant symbols, which seems convenient below.
It is not, however, required. Since $C$ is a simple set of constant
symbols indexed by $\omega$, they have a natural well-ordering that
could be used to pick distinguished elements.}
\begin{lem}
\label{lem:Comp}If $b_{0}$ is the canonical branch through $\mathcal{T}$,
then we may define a well-founded $\mathcal{L}_{\in}(C)$-model $\mathcal{M}$
from $b_{0}$. Moreover, we see that for all sentences $\varphi$
of $\mathcal{L}_{\in}(C)$\footnote{Strictly, we should write something like $\varphi^{\mathcal{M}}$
instead of $\mathcal{M}\models\varphi$ since, in general, $\mathcal{M}$
will not be a set. However, we allow ourselves this abuse of notation
so that the underlying ideas in an already busy forest of syntax are
easier to see.}
\[
\varphi\text{ occurs on }b_{0}\ \Rightarrow\ \mathcal{M}\models\varphi.
\]
\end{lem}

\begin{proof}
We let $\mathcal{M}=\langle C,\approx^{\mathcal{M}},\in^{\mathcal{M}},c_{i}^{\mathcal{M}}\rangle_{i\in\omega}$
so that: 
\begin{itemize}
\item $\approx^{\mathcal{M}}=\{\langle c_{i},c_{j}\rangle\in C\times C\ |\ ``c_{i}=c_{j}"\text{ occurs on }b_{0}\}$; 
\item $\in^{\mathcal{M}}=\{\langle c_{i},c_{j}\rangle\in C\times C\ |\ ``c_{i}\in c_{j}"\text{ occurs on }b_{0}\}$.
\end{itemize}
Observe that since there will be infinitely many true $\Sigma_{k}$-sentences:
the process we described in Section \ref{subsec:StanProc} and Lemma
\ref{lem:DefBranch}, for building our canonical tableau, ensures
that every constant symbol occurs on $b_{0}$. Also note that $\mathcal{M}$
is defined as a proper class and, as such, we cannot in general define
the associated quotient.\footnote{In particular, $\approx^{\mathcal{M}}$ and $\in^{\mathcal{M}}$ are
proper classes, although $C$ is not.} We leave it to the reader to show that the identity rules ensure
that $\approx^{\mathcal{M}}$ is a congruence relation with respect
to $\in^{\mathcal{M}}$ and $c_{i}^{\mathcal{M}}$ for all $i\in\omega$.\footnote{See \citep{PriestNCL,RestLog,SmullFO} to see how this works.}

We then proceed by induction on tableau complexity. This is a standard
proof that can be found in many places, so we just highlight some
representative cases.

(Base) Suppose $c_{i}=c_{j}$ occurs on $b_{0}$. Then $\langle c_{i},c_{j}\rangle\in\approx^{\mathcal{M}}$
and so $\mathcal{M}\models c_{i}=c_{j}$. The other cases are similar. 

(Step) Suppose $\exists\bar{x}\varphi(\bar{x},\bar{c})$ occurs on
$b_{0}$. Then since $b_{0}$ is complete, the Special $\exists$
rule guarantees that there is some $\bar{d}$ from $C$ such that
$\varphi(\bar{d},\bar{c})$ occurs on $b_{0}$. By the induction hypothesis,
we see that $\mathcal{M}\models\varphi(\bar{d},\bar{c})$ and so $\mathcal{M}\models\exists\bar{x}\varphi(\bar{x},\bar{c})$.
We leave the other cases to the reader. 

Finally, the well-foundedness of the model can be seen since the accumulation
of the ordinal assignments gives an order preserving map from the
interpreted model into the ordinals of our background universe. In
particular, observe that if $s$ is the accumulated ordinal assignment
determined by $b_{0}$, then rank-satisfiability ensures that: $c_{i}\approx^{\mathcal{M}}c_{j}$
implies $s(c_{i})=s(c_{j})$; and $c_{i}\in^{\mathcal{M}}c_{j}$ implies
$s(c_{i})<s(c_{j})$ for any $c_{i},c_{j}\in C$.
\end{proof}
Let $\mathcal{M}_{0}$ be the model defined using the canonical branch
$b_{0}$ through $\mathcal{T}$. It is not then difficult to see that: 
\begin{cor}
$\mathcal{M}_{0}\equiv_{\Sigma_{k}}V$; i.e., $\mathcal{M}_{0}$ and
$V$ satisfy the same $\Sigma_{k}$ sentences of $\mathcal{L}_{\in}$.\label{cor:Sigma_kSubMod}
\end{cor}

\begin{proof}
For any $\Sigma_{k}$ sentence $\exists\bar{x}\varphi(\bar{x})$ of
$\mathcal{L}_{\in}$, we see that:
\begin{itemize}
\item if $\exists\bar{x}\varphi(\bar{x})$ is true then we take the left
move on that cut and add $\varphi(\bar{d})$; and
\item if $\exists\bar{x}\varphi(\bar{x})$ is false, then we take the right
move and add $\forall\bar{x}\neg\varphi(\bar{x})$. 
\end{itemize}
Then Lemma \ref{lem:Comp} tells us that 
\[
\mathcal{M}_{0}\models\exists\bar{x}\varphi(\bar{x})\ \Leftrightarrow\ \exists\bar{x}\varphi(\bar{x})\text{ is true.}
\]

\end{proof}
Thus, whenever $k>3$, we see that: 
\begin{prop}
$\mathcal{M}_{0}\models KPC$. 
\end{prop}

\begin{proof}
By Corollary \ref{cor:Sigma_kSubMod}, we see that $\mathcal{M}_{0}$
satisfies every axiom of $KPC$ with the possible exception of Full
Set Induction, since they can all be articulated as $\Pi_{3}$ sentences.
To see that Full Set Induction holds, we suppose toward a contradiction
that $\psi(x,a)$ with $a\in M_{0}$ defines a nonempty subclass of
$M_{0}$ that has no $\in^{\mathcal{M}_{0}}$-minimal element. We
shall use this to define an infinite decreasing chain of ordinals
in the background universe. 

First we define a $\in^{\mathcal{M}_{0}}$-descending (class) sequence
$\langle d_{n}\rangle_{n\in\omega}$ of elements of $M_{0}$. First,
we let $d_{0}$ be $c_{i}$ for the least $i\in\omega$ such that
$\psi(c_{i},a)$ is true. Then for each $n\in\omega$, let $d_{n+1}$
be $c_{i}$ for the least $i\in\omega$ such that $\psi(c_{i},a)$
is true and $c_{i}\in^{\mathcal{M}_{0}}d_{n}$. Note that this is
clearly definable as a class function. Now from Lemma \ref{lem:Comp},
we see that there is a definable order preserving map $\pi$ from
$\mathcal{M}_{0}$ into the ordinals. So let $\langle r_{n}\rangle_{n\in\omega}$
be such that for all $n\in\omega$, $r_{n}=\pi(d_{n})$. Then we see
that for all $n\in\omega$
\[
r_{n+1}\in r_{n}
\]
which is impossible since we have Full Set Induction in the background
universe.
\end{proof}
We have now constructed $\mathcal{M}_{0}\models KPC$, with the same
$\Sigma_{k}$ theory as the background universe. We next repeat the
construction inside $\mathcal{M}_{0}$; the goal of this repetition
is to obtain an enumeration definable in $\mathcal{M}_{0}$ itself.

\subsubsection{Idempotence}

The following lemma establishes that if we take the canonical branch
through $\mathcal{M}_{0}$'s version of $\mathcal{T}$, we essentially
get $b_{0}$ again. This gives us the idempotence property.
\begin{lem}
Let $\mathcal{T}^{\mathcal{M}_{0}}$ be the $\Sigma_{k}$-tableau
as we defined it above but now in $\mathcal{M}_{0}$. 
\begin{itemize}
\item Given $b_{0}$, let $\langle\varphi_{n}\rangle_{n\in\omega}$ be the
sequence of $\mathcal{L}_{\in}(C)$ sentences occurring along $b_{0}$
and let $\langle\alpha_{n}\rangle_{n\in\omega}$ be the assignment
of ordinals to elements of $C$. 
\item Let $b_{1}$ be obtained using the same sentence sequence $\langle\varphi_{n}\rangle_{n\in\omega}$
but where we use $rank(c_{i})^{\mathcal{M}_{0}}$ for all $i\in\omega$
to determine ordinal assignment moves. 
\end{itemize}
Then $b_{1}$ is the canonical branch through $\mathcal{T}^{\mathcal{M}_{0}}$.\label{lem:SameBranch}
\end{lem}

\begin{rem*}
Note that this proof is done from the perspective of our background
universe, $V$.
\end{rem*}
\begin{proof}
We proceed by induction on the length of $b_{1}$. First we avoid
some notational clutter. Since $\mathcal{M}_{0}$ is well-founded
from the perspective of the background universe, we can identify $\mathcal{M}_{0}$'s
natural numbers with those of the background via their unique isomorphism.
Analogously, we identify $\mathcal{M}_{0}$ and $V$'s rendering of
constant symbols and sentence codes. The syntactic process described
in Section \ref{subsec:StanProc} is primitive recursive and thus,
also behaves identically in $\mathcal{M}_{0}$ and $V$ via this isomorphism. 

Now suppose that we have verified that the canonical branch $b^{*}$
through $\mathcal{T}^{\mathcal{M}_{0}}$ is such that up to length
$n$, $b^{*}$ uses the same sentences as $b_{0}$ and the ordinal
assignments sending $c_{i}$ to $rank(c_{i})^{\mathcal{M}_{0}}$ for
all $c_{i}$ occurring in sentences so far. We wish to show that this
also holds at length $n+1$. 

If the rule used to extend $b_{1}\restriction n$ to $b_{1}\restriction(n+1)$
is $\neg\exists$, $\neg\vee,\neg\neg$ or an identity rule there
is nothing to do: the next sentence is predetermined by those rules
and there is no ordinal assignment move. So we focus on three cases:
Special $\Sigma_{k}$-cut; Special $\exists$ and Special $\vee$.

Suppose that Special $\Sigma_{k}$-cut is used on the $\Sigma_{k}$
sentence $\exists\bar{x}\varphi(\bar{x})$. 
\begin{itemize}
\item Suppose that in $b_{0}$, we take the left move and add: 
\begin{align*}
 & \varphi(\bar{d}),s\\
 & \forall\bar{y}<_{\text{rank-lex}}\bar{d}\ \neg\varphi(\bar{y}).
\end{align*}
Then we see that $\mathcal{M}_{0}\models\varphi(\bar{d})\wedge\forall\bar{y}<_{\text{rank-lex}}\bar{d}\ \neg\varphi(\bar{y})$.
Thus, $\mathcal{M}_{0}\models\exists\bar{x}\varphi(\bar{x})$ and
so $b^{*}$ must also turn left, meaning that $b_{1}$ and $b^{*}$
agree on their $(n+1)^{th}$ sentences as required. Moreover, since
$\mathcal{M}_{0}\models\forall\bar{y}<_{\text{rank-lex}}\bar{d}\ \neg\varphi(\bar{y})$,
we see that $\varphi(\bar{y})$ cannot be satisfied in $\mathcal{M}_{0}$
with any ordinal assignment lexicographically preceding the ranks
of $d_{0}^{\mathcal{M}_{0}},...,d_{m}^{\mathcal{M}_{0}}$. Thus,
we preserve minimal-rank-satisfaction and so if we use the ordinal
assignment corresponding to the ranks of $d_{0}^{\mathcal{M}_{0}},...,d_{m}^{\mathcal{M}_{0}}$,
$b_{1}$ will agree with $b^{*}$ as required.
\item Suppose that in $b_{0}$, we take the right move and add $\forall\bar{x}\neg\varphi(\bar{x})$.
Then we see that $\mathcal{M}_{0}\models\forall\bar{x}\neg\varphi(\bar{x})$.
But this means that $b^{*}$ cannot take the left turn and thus, must
also add $\forall\bar{x}\neg\varphi(\bar{x})$ as required. 
\end{itemize}
Suppose that Special $\vee$ was used on the sentence $\varphi(\bar{c})\vee\psi(\bar{c})$. 
\begin{itemize}
\item Suppose that in $b_{0}$, we turned left and added $\varphi(\bar{c})$.
Then 
\[
\mathcal{M}_{0}\models\varphi(\bar{c})\wedge\forall\bar{y}<_{\text{rank-lex}}\bar{d}\neg(\psi(\bar{y})\wedge\bigwedge\Psi(\bar{y}))
\]
where $\bar{d}$ are the constant symbols already occurring on $b_{0}$
and $\Psi(\bar{d})$ is the set of major sentences occurring so far.
This means that $b^{*}$ will also turn left at this point. Thus,
$b^{*}$ and $b_{1}$ both add $\varphi(\bar{c})$. No ordinal assignment
occurs in the application of the Special $\vee$ rules, so this suffices.
\item Suppose that in $b_{0}$, we turned right and added $\psi(\bar{c})$.
Then we see that 
\[
\mathcal{M}_{0}\models\psi(\bar{c})\wedge\forall\bar{y}\leq_{\text{rank-lex }}\bar{d}\ \neg(\varphi(\bar{y})\wedge\bigwedge\Psi(\bar{y}))
\]
where $\Psi(\bar{d})$ is the set of major sentences occurring on
the branch so far and $\bar{d}$ are the constant symbols occurring
in those sentences. Similarly, this means that $b^{*}$ must turn
right at this point, so $b^{*}$ and $b_{1}$ continue to agree. 
\end{itemize}
Suppose that Special $\exists$ was used on $\exists\bar{x}\ \varphi(\bar{x},\bar{c})$.
Then $b^{*}$ must also extend with: 
\begin{align*}
 & \varphi(\bar{d},\bar{c}),s\\
 & \forall\bar{y},\bar{z}<_{\text{rank-lex}}\bar{c},\bar{d}\:\:\neg(\varphi(\bar{z},\bar{y})\wedge\bigwedge\Psi(\bar{y}))
\end{align*}
where $\Psi(\bar{c})$ is the set of major sentences already occurring
and $\bar{c}$ are the constant symbols occurring in them. Clearly,
$b^{*}$ and $b_{1}$ agree on the sentence added next. Moreover,
we see that 
\[
\mathcal{M}_{0}\models\varphi(\bar{d},\bar{c})\wedge\forall\bar{y},\bar{z}<_{\text{rank-lex}}\bar{c},\bar{d}\ \neg(\varphi(\bar{z},\bar{y})\wedge\bigwedge\Psi(\bar{y}))
\]
where $\Psi(\bar{c})$ is the set of sentences occurring so far.
Now let $\bar{e}=\bar{c},\bar{d}$ and $s:\{\bar{e}\}\to Ord^{\mathcal{M}_{0}}$
be such that $s(e_{i})=rank(e_{i})^{\mathcal{M}_{0}}$ for all $e_{i}$
occurring in $\bar{e}$. We then see that this assignment agrees with
the canonical branch $b^{*}$ since the second clause of the conjunction
above ensures that there is no way to find an ordinal assignment preceding
$\langle rank(e_{0})^{\mathcal{M}_{0}},...,rank(e_{m})^{\mathcal{M}_{0}}\rangle$.
\end{proof}
Let $\mathcal{M}_{1}$ be the model interpreted in $\mathcal{M}_{0}$
using $b_{1}$. We then observe that from the perspective of our background
universe, these models are isomorphic.
\begin{lem}
There is an isomorphism $j:\mathcal{M}_{0}\cong\mathcal{M}_{1}$,
where $j$ is definable in the background universe, $V$.
\end{lem}

\begin{proof}
The sentences occurring on an open branch of any tableau determine
the isomorphism type of the model it defines. Lemma \ref{lem:SameBranch}
establishes that $\mathcal{M}_{0}$ and $\mathcal{M}_{1}$ are determined
by branches with the same sentences. We then define the isomorphism
by letting $j(c_{n}^{\mathcal{M}_{0}})=c_{n}^{\mathcal{M}_{1}},$
or equivalently, we send the $\approx^{\mathcal{M}_{0}}$-class of
$c_{n}$ to the corresponding $\approx^{\mathcal{M}_{1}}$-class.
\end{proof}
The task is now to recover $j^{-1}(c_{n}^{\mathcal{M}_{1}})$ by a
definition inside $\mathcal{M}_{0}$ without parameters. The ordinal
assignment first recovers ordinals; membership then recovers sets
of ordinals; and rooted codes together with Mostowski collapse recover
arbitrary sets. For a little technical notation, we let $trcl$ be
the transitive closure operator and $trcoll$ be the transitive collapse
function.
\begin{lem}
There is a function $f$ that is definable over $\mathcal{M}_{0}$
without parameters such that\label{lem:defCountEnum} 
\[
f:\omega\twoheadrightarrow M_{0}.
\]
\end{lem}

\begin{proof}
Working in $\mathcal{M}_{0}$, let $b_{1}$ be the canonical branch
and let $s:C\to Ord$ be its accumulated ordinal assignment. We proceed
in four stages.

First, we define a partial surjection $f_{0}:C\rightharpoondown Ord^{M_{0}}$
such that $\langle c_{n},\alpha\rangle\in f_{0}$ iff $\mathcal{M}_{1}\models c_{n}\in Ord$
and $s(c_{n})=\alpha$. The key points then are that: 
\[
f_{0}(c_{n})=c_{n}^{\mathcal{M}_{0}}=j^{-1}(c_{n}^{\mathcal{M}_{1}})
\]
 whenever $\mathcal{M}_{1}\models c_{n}\in Ord$; and that since $c_{n}^{\mathcal{M}_{0}}$
is itself an ordinal, 
\[
\alpha=s(c_{n})=rank^{\mathcal{M}_{0}}(c_{n}^{\mathcal{M}_{0}})=c_{n}^{\mathcal{M}_{0}}.
\]
For $\beta\in Ord^{\mathcal{M}_{0}}$, we choose $c_{i}$ with $c_{i}^{\mathcal{M}_{0}}=\beta$.
Then $c_{i}^{\mathcal{M}_{1}}=j(\beta)$ is an ordinal and the definition
gives $f_{0}(c_{i})=\beta$.

Second, we define a partial surjection $f_{1}:C\rightharpoondown[Ord]^{<Ord}$
such that $f_{1}(c_{n})=X$ iff $\mathcal{M}_{1}\models c_{n}\subseteq Ord$
and for all $\alpha$
\[
\alpha\in X\ \Leftrightarrow\ \exists m\in\omega\ (f_{0}(c_{m})=\alpha\wedge\mathcal{M}_{1}\models c_{m}\in c_{n}).
\]
Similarly to the previous case, we then observe that 
\[
f_{1}(c_{n})=c_{n}^{\mathcal{M}_{0}}=j^{-1}(c_{n}^{\mathcal{M}_{1}}).
\]
From here it can be seen that:
\begin{enumerate}
\item $\alpha\in f_{1}(c_{n})\Longleftrightarrow\exists m\,(f_{0}(c_{m})=\alpha\wedge\mathcal{M}_{1}\models c_{m}\in c_{n})$;
\item $\exists m\,(f_{0}(c_{m})=\alpha\wedge\mathcal{M}_{1}\models c_{m}\in c_{n})\Longleftrightarrow j(\alpha)\in c_{n}^{\mathcal{M}_{1}}$;
and
\item $j(\alpha)\in c_{n}^{\mathcal{M}_{1}}\Longleftrightarrow\alpha\in c_{n}^{\mathcal{M}_{0}}.$
\end{enumerate}
If $Y\subseteq Ord^{\mathcal{M}_{0}}$ is a set, we choose $c_{i}$
with $c_{i}^{\mathcal{M}_{0}}=Y$. The displayed equivalences then
establish that $f_{1}(c_{i})=Y$.

Third, we define a relation $F_{2}$ on $C\times M_{0}$. Given
$c_{n}\in C$ and $x\in M_{0}$, let $F_{2}(c_{n},x)$ hold iff there
are an ordinal $\kappa$, a bijection $g:\kappa\to trcl(\{x\})$,
a distinguished point $r=g^{-1}(x)$, a relation $R\subseteq\kappa^{2}$
obtained by pulling membership back along $g$, a set of ordinals
$Y$ coding the rooted structure $\langle\kappa,R,r\rangle$, and
$k\in\omega$ such that $f_{1}(c_{k})=Y$ and, in $\mathcal{M}_{1}$,
the rooted relation decoded from $c_{k}$ has a Mostowski collapse
whose distinguished root is $c_{n}^{\mathcal{M}_{1}}$. All of this
is definable over $\mathcal{M}_{0}$ without parameters.

To see what this relation does, temporarily reason in the background
universe. If $f_{1}(c_{k})=Y$, then $c_{k}^{\mathcal{M}_{1}}=j(Y)$.
Hence the rooted relation decoded from $c_{k}$ is isomorphic to the
$j$-image of $\langle\kappa,R,r\rangle$, whose rooted collapse has
distinguished root $j(x)$. By uniqueness of the Mostowski collapse,
\[
F_{2}(c_{n},x)\ \Leftrightarrow\ c_{n}^{\mathcal{M}_{1}}=j(x).
\]
Since $j:\mathcal{M}_{0}\cong\mathcal{M}_{1}$ is a bijection and
the constants in $C$ exhaust $\mathcal{M}_{1}$, $F_{2}$ is the
graph of a total surjection $f_{2}:C\twoheadrightarrow M_{0}$. The
same calculation also shows that the definition is independent of
the choices of $g$ and of the particular rooted code.

Fourth and finally, we let $f:\omega\twoheadrightarrow M_{0}$ be
such that $f(n)=y$ iff $f_{2}(c_{n})=y$.
\end{proof}

\subsection{Main proof}

Finally, we can complete the proof of Lemma \ref{count-enum-thm}.
\begin{proof}
(of Lemma \ref{count-enum-thm}) Let $\mathcal{N}$ be a model of
$T=KPC+S$ and then work within $\mathcal{N}$. Let $k>3$ be such
that every sentence in $S$ has complexity $\Sigma_{k}$. Then let
$b_{0}$ be the canonical branch through the tableau $\mathcal{T}$
built using Special $\Sigma_{k}$-cut. Let $\mathcal{M}_{0}$ be the
model interpreted using $b_{0}$. Then we see that $\mathcal{M}_{0}$
satisfies $T$. Next we consider the tableau $\mathcal{T}^{\mathcal{M}_{0}}$
built inside $\mathcal{M}_{0}$ and its canonical branch $b_{1}$.
We let $\mathcal{M}_{1}$ be the model interpreted from $b_{1}$.
Then using Lemma \ref{lem:defCountEnum}, we see that there is a function
$f:\omega^{\mathcal{M}_{0}}\twoheadrightarrow\mathcal{M}_{0}$ definable
without parameters over $\mathcal{M}_{0}$, which is what we require.
\end{proof}

\section{Geology with little hammers}

\label{sec:GeoSmall}

With the definable countable enumeration lemma in hand, we have addressed
problem (1) from the sketch given at the beginning of Section \ref{sec:Quantifier-bounded-theories}.
However, before we give any examples, we need to address problem (2).
In particular, we need to obtain theories that have the stable definable
ground property. For this purpose, we make use of some set theoretic
geology. The idea is to make a small piece of the ground recognizable
in its generic extension, and then use that piece to recover the whole
ground.

There are two tasks here. First, we need a ground-model definition
that works over a weak set theory. Schlutzenberg has already provided
such a definition; we shall adapt it to use a smaller parameter. Second,
we need to arrange that a suitable parameter is itself definable without
parameters. We do this by coding it into the continuum function. The
first task can be carried out over a weaker theory than the second,
and we keep that distinction explicit.

As before, $KP$ includes Full Set Induction and $KPC$ adds Choice
in the form that every set injects into an ordinal. We use the following
axiomatization of $MOST$: Extensionality, set Foundation, Pairing,
Union, Infinity, Powerset, $\Sigma_{1}$-Separation, $\Sigma_{1}$-Collection,
and the assertion that every set is bijectable with an ordinal. Schlutzenberg
denotes this theory by $T_{1}$ \citep[Definition 2.2]{Schlutzenberg2022LocalMantles}.
Set Foundation says that every nonempty set has an $\in$-minimal
member. We write 
\[
MOST_{\mathrm{Ind}}=MOST+\text{Full Set Induction}=KPC+\text{Powerset}+\Sigma_{1}\text{-Separation}.
\]

We write $ZFC^{\Pi_{n}}$ for $KPC$ extended by Powerset, $\Pi_{n}$-Separation
and $\Pi_{n}$-Collection. Thus $ZFC^{\Pi_{0}}$ is $KPC$ with Powerset.
The coding construction in Section \ref{subsec:MPisDef} is carried
out over $ZFC^{\Pi_{3}}$. We give the preservation argument for every
$n\geq3$, since the later applications also use stronger fragments.

\label{sec:set-forcing-preservation} Our forcing constructions must
produce extensions satisfying the theory with which we began. Schlutzenberg
proves internally that set forcing forces every axiom of $MOST$ \citep[Lemma 2.5]{Schlutzenberg2022LocalMantles}.
Over $ZFC^{\Pi_{n}}$, we must also verify the stronger Separation
and Collection schemes. We first check these forcing assertions inside
the ground theory, and then use the truth lemma to transfer them to
a generic extension.

The usual name constructions give the required sets once the forcing
predicates have the right complexity. The main point is existential
forcing: local witness names can be collected and mixed into one name.
This is where Collection and Choice enter.
\begin{lem}
\label{lem:sharp-forcing} Let $n\geq3$ and work in $ZFC^{\Pi_{n}}$.
Fix a set forcing $\mathbb{P}$. For every fixed $\Sigma_{i}$ formula,
where $1\leq i\leq n$, its forcing relation has a $\Sigma_{i}$ definition.
The corresponding assertion holds for $\Pi_{i}$ formulas. Bounded
forcing has a $\Delta_{1}$ definition.
\end{lem}

\begin{proof}
Begin with the $\Delta_{1}$ definition of bounded forcing and proceed
through the prenex quantifier blocks. Suppose the matrix $\psi$ is
$\Pi_{i-1}$ and its forcing relation already has the required definition.
If $p\Vdash\exists x\,\psi(x,\vec{\tau})$, then for every $q\leq p$
there are $s\leq q$ and a name $\dot{x}$ such that $s\Vdash\psi(\dot{x},\vec{\tau})$.
Collection gathers a set of such witness pairs, one available below
each $q\leq p$. At the first level $\Sigma_{1}$-Collection suffices,
since bounded forcing is $\Delta_{1}$; at later levels $\Pi_{i-1}$-Collection
suffices.

Using Separation, retain from the collected set precisely those pairs
$\langle s,\dot{x}\rangle$ with $s\leq p$, $\dot{x}$ a $\mathbb{P}$-name,
and $s\Vdash\psi(\dot{x},\vec{\tau})$. Using Choice, take a maximal
antichain from their witnessing conditions and select one witness
name for each antichain condition. The antichain is predense below
$p$. Mix its selected names into a single name $\dot{u}$ agreeing
with each selected witness below its corresponding condition. This
construction uses unions and bounded Separation on sets of names and
conditions. Substitution and the density lemma give $p\Vdash\psi(\dot{u},\vec{\tau})$.
Consequently 
\[
p\Vdash\exists x\,\psi(x,\vec{\tau})\quad\longleftrightarrow\quad\exists\dot{u}\;p\Vdash\psi(\dot{u},\vec{\tau}).
\]
The right side has a $\Sigma_{i}$ definition. Quantification over
dotted variables is over names; the $\Delta_{1}$ name predicate is
absorbed into this bound. Negation gives the dual $\Pi_{i}$ definition,
since a condition forces a negation exactly when no stronger condition
forces the negated formula. Repeated quantifiers of the same kind
are handled in the same block. This proves the induction for each
fixed formula.
\end{proof}
Working in $ZFC^{\Pi_{n}}$, we now verify that every instance of
$\Pi_{n}$-Separation and $\Pi_{n}$-Collection is forced. Fix a
$\Pi_{n}$ formula $\varphi$ and names $\dot{a},\vec{\tau}$. For
Separation, form 
\[
\dot{b}=\{\langle\dot{x},q\rangle\in\mathsf{dom}(\dot{a})\times\mathbb{P}:q\Vdash\dot{x}\in\dot{a}\ \land\ q\Vdash\varphi(\dot{x},\vec{\tau})\}.
\]
By Lemma \ref{lem:sharp-forcing}, this set exists by $\Pi_{n}$-Separation.
The forcing clauses give 
\[
1\Vdash\forall x\bigl(x\in\dot{b}\leftrightarrow(x\in\dot{a}\land\varphi(x,\vec{\tau}))\bigr).
\]

For Collection, suppose $p\Vdash\forall x\in\dot{a}\,\exists y\,\varphi(x,y,\vec{\tau})$.
For every $\langle\dot{x},r\rangle\in\dot{a}$ and $q\leq p,r$, there
are $s\leq q$ and a name $\dot{y}$ with $s\Vdash\varphi(\dot{x},\dot{y},\vec{\tau})$.
The triples $\langle\dot{x},r,q\rangle$ form a set, and the relation
describing a witness pair $\langle s,\dot{y}\rangle$ is $\Pi_{n}$.
Thus $\Pi_{n}$-Collection gathers enough witness pairs. Let $Y$
be the set of their name components, discarding any objects which
are not names, and let $\dot{b}=\{\langle\dot{y},1\rangle:\dot{y}\in Y\}$.
Below each relevant $q$ there is a condition forcing that a suitable
witness belongs to $\dot{b}$. The density lemma therefore gives 
\[
p\Vdash\forall x\in\dot{a}\,\exists y\in\dot{b}\,\varphi(x,y,\vec{\tau}).
\]
The remaining basic axioms use the usual name constructions, as in
the $MOST$ argument. In particular, Powerset is obtained from the
set of subnames of $\mathsf{dom}(\dot{a})\times\mathbb{P}$. Every
non-induction axiom instance of $ZFC^{\Pi_{n}}$ is therefore forced.

To pass to a generic extension we use the truth lemma. A definable
dense class need not be a set in a weak ground model, so meeting the
ground's dense sets does not by itself give the truth lemma at every
complexity. External well-foundedness does not remove this distinction.

Let $\mathcal{M}\models ZFC^{\Pi_{n}}$. For fixed $n$, the non-induction
axiom instances have a common finite complexity bound. Choose a standard
$j\geq1$ such that they are all $B\Sigma_{j}$, including any further
formulas needed in the argument when choosing $j$. Lemma \ref{lem:O2GeneralTruth}
then transfers these forced assertions to $\mathcal{M}^{\mathbb{P}}/G$
whenever $G$ is $B\utilde{\Sigma}_{\mathsf{b}(j)}$-generic over
$\mathcal{M}$. This applies to both externally well-founded and ill-founded
grounds. The bound $\mathsf{b}(j)$ supplies a sufficient genericity
level for transfer; the sharper definitions in Lemma \ref{lem:sharp-forcing}
justify the earlier uses of Separation and Collection. If $G$ is
definable over $\mathcal{M}$, Lemma \ref{lem:tightFullSetInduction}
supplies Full Set Induction separately.

\subsection{Ground model definability in a weak set theory}

Ground model definability was proved independently by Laver and Woodin;
see \citet{LaveRefDef} and the historical account in \citet{FuchsHamkinsReitz2015},
Section 1. Informally speaking, it tells us that the inner models
that generically extend to the universe, its grounds, are uniformly
definable using a parameter. Schlutzenberg’s version works over $MOST$.
We use it to supply the candidate grounds, and then compare candidates
to obtain a smaller parameter needed for coding.
\begin{thm}
\label{thm:SchlutzenbergGMD}(Schlutzenberg.) Over $MOST$ there is
a fixed formula $\mathsf{grd}(t,x)$ such that the classes $W_{t}=\{x:\mathsf{grd}(t,x)\}$
enumerate exactly the set-forcing grounds, including the universe
itself. If $V=W[G]$ and the forcing belongs to $t=(H_{\eta^{+}})^{W}$,
then $W_{t}=W$.
\end{thm}

This is \citet{Schlutzenberg2022LocalMantles}, Definition 2.9 and
Corollary 2.10. A ground here satisfies $MOST$. If the ambient universe
has Full Set Induction, each of these definable transitive grounds
has it too: relativize the formula in an induction instance to the
ground and apply ambient induction. Thus this family supplies grounds
strong enough for the interpretation arguments of the paper.

The hereditary-size parameter in Theorem \ref{thm:SchlutzenbergGMD}
recalls the large initial segment $V_{\alpha}^{W}$ used in Laver's
first proof. \citet{LaveRefDef} published the ground-model definability
theorem using Hamkins's approximation and cover argument, which allows
the parameter $\mathcal{P}(\delta)^{W}$. Hamkins and Johnstone subsequently
observed that $(^{<\delta}2)^{W}$ suffices. See \citet[Theorem 6 and the discussion following Lemma 8]{FuchsHamkinsReitz2015}
for this history.

For the coding applications in Section \ref{subsec:MG-single}, we
need this smaller parameter. We shall make the ground definable without
parameters by coding a set that determines it into the continuum function.
In our construction that set is $r_{\kappa}=({}^{<\kappa^{+}}2)^{W}$,
where $\kappa$ is fixed before the coding begins. The coding forcing
can be much larger than $r_{\kappa}$, so we need a ground definition
whose parameter does not grow with the size of that forcing. Theorem
\ref{thm:PI_3GMDT} provides this over $MOST$.

Our strategy is to use the family supplied by Theorem \ref{thm:SchlutzenbergGMD}
and select its intended member using the smaller parameter. We require
a candidate ground to have the specified $<\delta$-sequences, to
have $\delta$-approximation and $\delta$-cover in the ambient universe,
and to agree with that universe on $\delta^{+}$. The intended ground
satisfies these requirements. Lemma \ref{lem:small-parameter-uniqueness}
will show that every candidate satisfying them is that ground. We
can therefore quantify over the indices in Theorem \ref{thm:SchlutzenbergGMD},
leaving only the smaller parameter in the definition.

Recall that $W\subseteq V$ has $\delta$-cover if every set $a\subseteq W$
with $|a|^{V}<\delta$ is contained in a set $b\in W$ with $|b|^{W}<\delta$.
It has $\delta$-approximation if a set $A\subseteq W$ belongs to
$W$ whenever $A\cap b\in W$ for every $b\in W$ with $|b|^{W}<\delta$.
Here $\delta$ will be an infinite cardinal regular in the models
in question.
\begin{thm}
\label{thm:PI_3GMDT}($MOST$) There is a single first-order formula
$\vartheta(x,r)$ with the following property. Suppose $V=W[G]$ is
a set-forcing extension of a proper inner model $W\models MOST$.
If $\delta$ is an infinite cardinal regular in both $W$ and $V$,
the extension has the $\delta$-cover and $\delta$-approximation
properties, and $(\delta^{+})^{W}=(\delta^{+})^{V}$, then, for $r=(^{<\delta}2)^{W}$,
\[
W=\{x\in V:V\models\vartheta(x,r)\}.
\]
In particular, every proper ground is uniformly definable from such
a parameter: if $G$ is generic for $\mathbb{P}\in W$, we may take
$\delta=(|\mathbb{P}|^{+})^{W}$.
\end{thm}

The definition has a simple form. We search the family $W_{t}$ for
a ground with the given $<\delta$-sequences and the approximation,
cover and cardinal-agreement properties in the theorem. There are
two things to prove: the intended ground meets these requirements,
and every ground meeting them is the intended one. Theorem \ref{thm:SchlutzenbergGMD}
provides an index for the first part. Lemma \ref{lem:small-parameter-uniqueness}
will give the second, once agreement on $<\delta$-sequences has been
upgraded to agreement on $\mathcal{P}(\delta)$.
\begin{lem}
\label{lem:small-parameter-uniqueness}($MOST$.) Suppose $M,N$ belong
to the ground family above, $\delta$ is regular in $M,N,V$, and
both inclusions into $V$ have $\delta$-approximation and $\delta$-cover.
If 
\[
\mathcal{P}(\delta)^{M}=\mathcal{P}(\delta)^{N},\qquad(\delta^{+})^{M}=(\delta^{+})^{N}=(\delta^{+})^{V},
\]
 then $M=N$.
\end{lem}

The idea of Lemma \ref{lem:small-parameter-uniqueness} is to place
each small set inside a cover common to both grounds. Subsets of that
cover can then be coded below $\delta^{+}$, where the grounds agree.
Approximation turns this local agreement into agreement on all sets
of ordinals. We give the set-existence details because the grounds
are defined by a formula whose complexity need not be low.
\begin{proof}
We first explain why the sets used in the comparison exist. For every
ambient set $a$, the intersection $a\cap W_{t}$ is a set. For a
default index $W_{t}=V$, this is immediate. For a true index, membership
in $W_{t}$ is given by the local-witness formula $\varphi_{\mathsf{grd}}(t,x)$
of \citet{Schlutzenberg2022LocalMantles}, Definition 2.9. Choose
an infinite cardinal $\xi$ with $|\mathsf{tc}(\{t,a\})|\leq\xi$,
and let $B=H_{(2^{\xi})^{+}}$.

The witness bound in Schlutzenberg’s Remark 2.11 applies uniformly
to all $x\in a$: when a witness exists, it can be taken to be a local
ground of $H_{\xi^{+}}$, with the whole witnessing structure of hereditary
size at most $2^{\xi}$. The same remark gives absoluteness of the
witness test, so
\[
x\in W_{t}\quad\longleftrightarrow\quad B\models\varphi_{\mathsf{grd}}(t,x)\qquad(x\in a).
\]
Here $t$ is a true index in the ambient universe; we do not ask $B$
to decide that. Since $\varphi_{\mathsf{grd}}$ is fixed, the right-hand
side is just its relativization to the set $B$. Bounded Separation
on $a$ therefore gives $a\cap W_{t}$.

Fix an ordinal $\alpha$. Given $a\subseteq\alpha$ of ambient size
less than $\delta$, alternately cover in $M$ and $N$ to obtain,
for $i<\delta$, 
\[
a\cup\bigcup_{j<i}Y_{j}\subseteq X_{i}\subseteq Y_{i}\subseteq\alpha,
\]
 where $X_{i}$ is small in $M$ and $Y_{i}$ is small in $N$. This
is a recursion on sets. Indeed, all possible covers and their bijections
with ordinals below $\delta$ lie in the intersections with $M$ and
$N$ of 
\[
\mathcal{P}(\alpha)\cup\mathcal{P}(\delta\times\alpha).
\]
 These intersections are sets by the preceding paragraph. Well-order
the two families, and at each stage take the first suitable cover
together with its witnessing bijection. Regularity and cover ensure
that a choice is always possible. The partial histories form a subset
of a fixed function set; their uniqueness and compatibility give the
recursion by taking unions.

Let $C=\bigcup_{i<\delta}X_{i}=\bigcup_{i<\delta}Y_{i}$. The retained
bijections show that $|C|^{V}\leq\delta$. If $b\in M$ is small in
$M$, regularity bounds the stages at which members of $C\cap b$
enter the union. Thus $C\cap b=X_{i}\cap b\in M$ for some $i$, and
approximation gives $C\in M$. The same argument gives $C\in N$.
We have obtained a common cover of $a$, of ambient size at most $\delta$.

Agreement on $\mathcal{P}(\delta)$ gives agreement on the subsets
of every $\beta<\delta^{+}$. To see this, choose in $M$ a bijection
from some $d\leq\delta$ onto $\beta$ and pull back the ordinal order
to $d$. A fixed ordinal pairing codes this relation by a subset of
$\delta$. The code belongs to $N$, and its ordinal collapse there
is the same as in $M$. Consequently the bijection belongs to both
grounds, and it transports their common subsets of $d$ to the same
subsets of $\beta$.

Now let $A\subseteq\alpha$ belong to $M$, and let $b\in N$ be small
in $N$. Choose a common cover $C$ of $b\cap\alpha$ as above. Its
increasing enumeration belongs to both grounds. Its order type $\beta$
is below the common $\delta^{+}$, since $|C|^{V}\leq\delta$. The
code of $A\cap C$ is a subset of $\beta$, and hence belongs to both
grounds. It follows that $A\cap b\in N$. Approximation gives $A\in N$.
By symmetry, $M$ and $N$ have the same sets of ordinals.

Finally, every set in either ground has an ordinal code for its pointed
transitive closure. The same code belongs to the other ground, and
its collapse recovers the same set. The needed collapse is available
in $MOST$: $\Sigma_{1}$-Separation forms the set of nodes admitting
partial collapses, $\Sigma_{1}$-Collection collects compatible predecessor
collapses, and well-foundedness rules out a least failure. Uniqueness
of collapse gives agreement. Thus $M=N$.
\end{proof}
We can now recover a ground from its $<\delta$-sequences.
\begin{proof}
(of Theorem \ref{thm:PI_3GMDT}.) Recover $\delta$ from $r$ by $\delta=\bigcup\{\mathsf{dom}(s):s\in r\}$.
Let $C(t,r)$ say that
\[
r=(^{<\delta}2)^{W_{t}},\qquad(\delta^{+})^{W_{t}}=(\delta^{+})^{V},
\]
that $\delta$ is infinite regular in $W_{t}$ and $V$, and that
$W_{t}\subseteq V$ has $\delta$-approximation and $\delta$-cover.
Define
\[
\vartheta(x,r)\ \longleftrightarrow\ \exists t\bigl(C(t,r)\land\mathsf{grd}(t,x)\bigr).
\]
All occurrences of $W_{t}$ are replaced by its fixed defining formula.
Thus $r$ is the only parameter. The conditions in $C(t,r)$ are additional
requirements on a candidate; Theorem \ref{thm:SchlutzenbergGMD} does
not assert that every candidate satisfies them.

First, the intended ground qualifies. Fix $W,G,\delta,r$ as in the
theorem. Theorem \ref{thm:SchlutzenbergGMD} supplies an index $t_{0}$
with $W_{t_{0}}=W$. The hypotheses on $W\subseteq V$ give $C(t_{0},r)$.
Hence every member of $W$ satisfies $\vartheta(x,r)$.

Next, suppose $C(t,r)$ holds, and let $M=W_{t}$. Equality of the
$<\delta$-sequences gives agreement between $M$ and $W$ on all
bounded subsets of $\delta$. If $A\in W$ and $A\subseteq\delta$,
then $A\cap\alpha\in M$ for every $\alpha<\delta$. For any $b\in M$
with $|b|^{M}<\delta$, regularity bounds $b\cap\delta$ by some $\alpha<\delta$.
Thus
\[
A\cap b=(A\cap\alpha)\cap b\in M.
\]
Approximation for $M\subseteq V$ gives $A\in M$. The symmetric argument
gives $\mathcal{P}(\delta)^{M}=\mathcal{P}(\delta)^{W}$.

We can now apply Lemma \ref{lem:small-parameter-uniqueness}. Its
hypotheses hold for $W$ by assumption and for $M$ because $C(t,r)$
holds. Hence $M=W$. Every qualifying index therefore defines the
same ground, so $\vartheta(x,r)$ defines $W$. For forcing of size
less than $\delta$, the small-forcing argument gives $\delta$-cover
and $\delta$-approximation and preserves $\delta$ and its successor.
We may consequently take $\delta=(|\mathbb{P}|^{+})^{W}$.
\end{proof}
The small parameter selects the ground; it need not be an index in
Schlutzenberg’s family. We have used Theorem \ref{thm:SchlutzenbergGMD}
to quantify over definable candidates, and Lemma \ref{lem:small-parameter-uniqueness}
to show that the requirements leave only the intended ground.

\subsection{A uniquely definable ground theorem}

\label{subsec:MPisDef}Our goal is to construct set forcings whose
generic extensions can define their associated grounds without parameters,
for use in the stable definable ground property. We construct a family
$\mathbb{G}_{\kappa}$ for suitable regular cardinals $\kappa$. At
a fixed index $\kappa$, the universe has at most one ground from
which it arises by the designated forcing, and that ground is definable
from $\kappa$. For the preserved indices $\omega$ and $\omega_{2}$
used in our applications, this gives a definition without parameters.

Theorem \ref{thm:PI_3GMDT} reduces recovery of the ground to making
$r=(^{<\delta}2)^{V}$ recognizable. We first add a small generic
set, and then code $r$ much farther up the continuum function. This
strategy uses Hamkins's approximation and cover argument; see \citet[Theorem 6 and Lemma 8]{FuchsHamkinsReitz2015}.
Together the two steps give approximation and cover, while the closed
tail preserves the relevant low cardinals. The code must begin at
a recognizable place. Once the extension can locate and read it, Theorem
\ref{thm:PI_3GMDT} recovers the ground.

We first give the two-stage form and then construct the forcing used
at the second stage. Lemma \ref{lem:localCoding} will supply $\mathsf{Code}(r,\theta)$
whenever $\theta$ is infinite regular and $|\mathsf{tc}(\{r\})|<\theta$.
This forcing will be $<\theta^{+}$-directed closed,\footnote{A forcing is $<\gamma$-directed closed if every nonempty directed
family of fewer than $\gamma$ conditions has a common stronger condition.
A family is directed if any two of its members have a common stronger
member in the family.} preserve $\theta$ and all smaller cardinals, and make $r$ uniformly
recoverable from $\theta$.

Let $\kappa$ be an infinite regular cardinal with $\kappa^{<\kappa}=\kappa$,
and let $r_{\kappa}=({}^{<\kappa^{+}}2)^{V}$. Thus the approximation-and-cover
cardinal in Theorem \ref{thm:PI_3GMDT} will be $\delta=\kappa^{+}$.
In the intermediate $Add(\kappa,1)$-extension let $\theta=(2^{\kappa^{++}})^{+}$,
and let $\dot{\mathbb{Q}}_{\kappa}$ be the canonical name for $\mathsf{Code}(\check{r}_{\kappa},\dot{\theta})$.
Since $|\mathsf{tc}(\{r_{\kappa}\})|\leq(2^{\kappa})^{V}<\theta$,
this is defined. Let 
\[
\mathbb{G}_{\kappa}=Add(\kappa,1)*\dot{\mathbb{Q}}_{\kappa}.
\]
The name retains the original ground parameter $r_{\kappa}$.

The first factor has size $\kappa$, is $<\kappa$-directed closed,
and preserves cardinals. The tail's $<\theta^{+}$-directed closure
is more than is needed to preserve $\kappa^{<\kappa}=\kappa$ and
all cardinals through $\kappa^{++}$. Thus the whole forcing has these
preservation properties and is nontrivial and $<\kappa$-directed
closed. The tail also preserves the intermediate power set of $\kappa^{++}$
and the cardinals through $\theta$. Consequently the final universe
computes the same $\theta=(2^{\kappa^{++}})^{+}$ from $\kappa$.
It can then read an ordinal code $X\subseteq\theta$, recover $r_{\kappa}$,
and use Theorem \ref{thm:PI_3GMDT} to define the ground. Lemma \ref{M_P-is-def}
will verify this recovery.

For general $\kappa$ the definition uses $\kappa$ as a parameter.
In the applications $\kappa=\omega$ or $\omega_{2}$, subject to
$\kappa^{<\kappa}=\kappa$. Preservation of all cardinals through
$\kappa$ keeps this index equal to the extension's own $\omega$
or $\omega_{2}$, so the definition is parameter-free. We write $\mathbb{G}=\mathbb{G}_{\omega}$
where the assumption $\omega^{<\omega}=\omega$ is automatic.

For $\kappa=\omega$, this is the promised example
\[
V\subseteq V[c]\subseteq U=V[c][H],\qquad r=({}^{<\omega_{1}}2)^{V}.
\]
The tail codes this original $r$, not the $<\omega_{1}$-sequences
in $V[c]$. It is closed enough to add no ordinal sequences of length
at most $\omega_{2}$. We shall use Fact \ref{fact:Geol} for approximation
and cover, and preservation of $\omega_{2}$ for the remaining cardinal
agreement in Theorem \ref{thm:PI_3GMDT}.\footnote{Here we can see why the smaller parameter was needed. For $\kappa=\omega$,
$\mathrm{Add}(\omega,1)\in H_{\omega_{1}}^{V}$, but $\mathbb{G}_{\omega}\notin H_{\omega_{1}}^{V}$.
Theorem \ref{thm:SchlutzenbergGMD} defines $V$ from a sufficiently
large $t=(H_{\eta^{+}})^{V}$, provided the whole forcing belongs
to $t$. The tail, however, makes $r_{\omega}$ recoverable. If we
used this construction to code $t$ instead, we would have to begin
above $|\mathsf{tc}(\{t\})|$, and the resulting forcing would again
be too large to belong to $t$. Theorem \ref{thm:PI_3GMDT} lets us
keep $r_{\omega}$ fixed while placing its code farther up.}

For $\mathsf{Code}(r,\theta)$ we use a local version of the continuum
coding in \citet[Section 3, proof of Theorem 11]{ReitzTGA}. Compare
also the preparation of a coding interval in \citet[proof of Theorem 67]{FuchsHamkinsReitz2015}.
Each bit is recorded by whether $GCH$ holds at its coding cardinal,
after preparing the cardinal arithmetic there. The additional problem
is to construct this forcing and verify its preservation properties
in the weaker theory. Lemma \ref{lem:coding-presentation} supplies
that verification.

Let $r$ be a set and let $\theta$ be an infinite regular cardinal
with $|\mathsf{tc}(\{r\})|<\theta$. Use a fixed ordinal pairing to
code a domain, a binary relation and a distinguished root by a subset
of $\theta$. We consider all such codes $X\subseteq\theta$ of $\langle\mathsf{tc}(\{r\}),\in,r\rangle$.
Note that to obtain a particular code for $r$, we use Choice, but
this does not give us a uniformly definable choice of code. To get
around this, we use a lottery sum over all possible codes.\footnote{The lottery over ordinal codes is an additional device in this adaptation
of Reitz's continuum coding. Its purpose is to make the forcing uniformly
definable from $r,\theta$ without choosing a particular ordinal presentation.}

For each such code $X\subseteq\theta$, let $\mathbb{P}_{X}=\mathbb{R}_{\theta}*\dot{\mathbb{E}}_{X}$.
The first factor prepares $GCH$ throughout a sufficiently long interval
and is independent of $X$. The second codes $X$ at consecutive successor
cardinals in that interval. Define $\mathsf{Code}(r,\theta)$ to be
the lottery sum of these $\mathbb{P}_{X}$, with a greatest condition
adjoined. Every component codes the same fixed set $r$.

To prepare the interval we use the following collapse. For an infinite
regular cardinal $\rho$ and a cardinal $\chi\geq\rho$, let $Coll(\rho,\chi)$
consist of partial functions from $\rho$ to $\chi$ of size below
$\rho$, ordered by reverse inclusion. This forcing is $<\rho$-directed
closed and makes $|\chi|=\rho$.

Let $\mathbb{R}_{\theta}$ be a full-support iteration of length $\theta+1$.
Write $M_{i}$ for the extension by its first $i$ stages. Define
the cardinals $\mu_{i}$, for $i\leq\theta$, during the iteration
by
\[
\mu_{0}=\theta,\qquad\mu_{i+1}=(\mu_{i}^{+})^{M_{i+1}}\quad(i<\theta),\qquad\mu_{i}=\sup_{j<i}\mu_{j}\quad(0<i\leq\theta\text{ a limit}).
\]
At stage $i\leq\theta$, force over $M_{i}$ with $Coll(\mu_{i}^{+},2^{\mu_{i}})$,
with both parameters computed in $M_{i}$. The collapse therefore
uses a regular successor even when $\mu_{i}$ is singular.

In the completed preparation extension $W$, let $\eta=\mu_{\theta}$
and, for $\xi<\theta$, let
\[
\lambda_{\xi}=\theta^{+(\xi+1)},\qquad F_{X}(\lambda_{\xi})=\begin{cases}
\lambda_{\xi}^{+}, & \xi\in X,\\
\lambda_{\xi}^{++}, & \xi\notin X.
\end{cases}
\]
Define $\mathbb{E}_{X}$ in $W$ to be the full-support product
\[
\mathbb{E}_{X}=\prod_{\xi<\theta}Add(\lambda_{\xi},F_{X}(\lambda_{\xi})).
\]
Here $Add(\lambda,\nu)$ consists of partial functions from $\nu\times\lambda$
to $2$ of size below $\lambda$. All the cardinals, factors and their
sets of conditions in this product are computed in $W$. With only
$\theta$ coordinates, all above $\theta$, full support satisfies
the Easton support requirement.

For the large-cardinal applications in Section \ref{subsec:MG-single},
we also need the coding forcing to remain small relative to the cardinal
we wish to preserve. Inaccessibility suffices for this size estimate;
the stronger large-cardinal hypothesis enters the subsequent preservation
argument.
\begin{lem}
\label{lem:localCoding}($ZFC^{\Pi_{3}}$.) The forcing $\mathsf{Code}(r,\theta)$
is uniformly definable from $r,\theta$ and is $<\theta^{+}$-directed
closed. It preserves $\theta$ and all smaller cardinals. In every
generic extension, $r$ is uniformly definable from $\theta$: the
chosen code is 
\[
X=\{\xi<\theta:2^{\theta^{+(\xi+1)}}=\theta^{+(\xi+2)}\},
\]
 and its rooted collapse recovers $r$. Here the index operations
are ordinal operations and all the displayed cardinals are computed
in the final extension. If $\zeta$ is inaccessible and $\theta<\zeta$,
the forcing has a set of conditions of size below $\zeta$.
\end{lem}

\begin{proof}
At preparation stage $i$, the collapse is $<\mu_{i}^{+}$-directed
closed, preserves $\mu_{i}$ and $\mu_{i}^{+}$, and adds no subsets
of $\mu_{i}$. It makes the existing power set have size $\mu_{i}^{+}$,
so it forces $2^{\mu_{i}}=\mu_{i}^{+}$. The next stage, and the full-support
tail after it, are closed enough to preserve this equality and its
two cardinals. At limits the supremum of the earlier preserved cardinals
remains a cardinal. Induction therefore gives, in $W$,
\[
\mu_{i}=\theta^{+i},\qquad2^{\mu_{i}}=\mu_{i}^{+}\quad(i\leq\theta).
\]
Thus $GCH$ holds at every cardinal from $\theta$ through $\eta=\theta^{+\theta}$,
including the singular limits and the endpoint. In particular, $2^{<\theta^{+}}=\theta^{+}$
and $2^{\eta}=\eta^{+}$.

Work now in $W$. The function $F_{X}$ is nondecreasing, and its
values have cofinality greater than their arguments. Each factor of
$\mathbb{E}_{X}$ is $<\theta^{+}$-directed closed. Coordinatewise
unions give the same closure for the product. The preparation and
its two-step iteration with this product have that closure as well.
Consequently they preserve $\theta$ and all smaller cardinals.

Fix a coding cardinal $\lambda$ and factor $\mathbb{E}_{X}=\mathbb{P}\times\mathbb{Q}$
at $\lambda$, with $\mathbb{P}$ containing the coordinates at most
$\lambda$. A condition in $\mathbb{P}$ has combined domain of size
$<\lambda$. Local $GCH$ gives $\lambda^{<\lambda}=\lambda$, so
the generalized $\Delta$-system argument gives the $\lambda^{+}$-chain
condition: among $\lambda^{+}$ conditions, two agree on their common
domain and hence are compatible. In $W$, the tail $\mathbb{Q}$ is
$<\lambda^{+}$-closed. Easton's lemma therefore gives
\[
1_{\mathbb{P}}\Vdash\check{\mathbb{Q}}\text{ is }<\lambda^{+}\text{-distributive}.
\]
In particular, the tail adds no subsets of $\lambda$ over the extension
by $\mathbb{P}$.

The product preserves all cofinalities. To see this, suppose a ground
cofinality decreases to $\tau$, where $\tau$ is regular in the final
extension and hence was already regular in $W$. If $\tau\leq\theta$,
the product's closure prevents the new cofinal sequence. If $\theta<\tau<\eta$,
then $\tau$ is a coding cardinal: every limit cardinal in this interval
has cofinality at most $\theta$. The factorization at $\tau$ puts
the sequence in the extension by the product through $\tau$, whose
$\tau^{+}$-chain condition prevents the decrease. Finally, $\eta$
is singular of cofinality $\theta$, and all product domains can be
coded in a set of size $\eta$. Thus $|\mathbb{E}_{X}|\leq2^{\eta}=\eta^{+}$,
which preserves every ground cofinality at least $\eta^{++}$. This
rules out a decrease to any regular $\tau\geq\eta^{+}$. Preservation
of cofinalities also gives preservation of cardinals.

Return to the factorization at $\lambda$. By monotonicity, the domains
of conditions in $\mathbb{P}$ lie in a set of size $F_{X}(\lambda)$.
Local $GCH$ gives
\[
|\mathbb{P}|\leq F_{X}(\lambda)^{<\lambda}=F_{X}(\lambda),\qquad(\lambda^{+})^{\lambda}=\lambda^{+},\qquad(\lambda^{++})^{\lambda}=\lambda^{++}.
\]
For the last equality, every function from $\lambda$ to $\lambda^{++}$
is bounded, and its range lies in an ordinal of cardinality at most
$\lambda^{+}$. The $\lambda^{+}$-chain condition therefore bounds
nice names for subsets of $\lambda$ by $F_{X}(\lambda)$. The $\lambda$-coordinate
adds that many distinct subsets, and the tail adds no more. Hence
$2^{\lambda_{\xi}}=F_{X}(\lambda_{\xi})$ in the final extension.

Since the product preserves cardinals, its coding cardinals are still
the final universe's own $\theta^{+(\xi+1)}$. The displayed test
recovers $X$, and its rooted collapse recovers $r$. A directed family
of non-greatest conditions in the lottery sum lies in one summand,
so the lottery has the same closure bound. Lemma \ref{lem:coding-presentation}
supplies the set construction, the weak-theory verification of the
preservation argument, and the inaccessible size bound.
\end{proof}

The remaining check concerns set existence and complexity. The cited
coding arguments are carried out in $ZFC$; here we bound the names,
construction histories and preservation arguments needed over $ZFC^{\Pi_{3}}$.
\begin{lem}
\label{lem:coding-presentation}($ZFC^{\Pi_{3}}$.) The construction
of $\mathsf{Code}(r,\theta)$ above is uniformly definable as a set
forcing. If $\theta<\zeta$ and $\zeta$ is inaccessible, that forcing
belongs to $V_{\zeta}$ and has size less than $\zeta$.
\end{lem}

\begin{proof}
Powerset has a $\Pi_{1}$ graph, and cardinalization and cardinal
succession have $\Sigma_{2}$ graphs. Their set-length iterations
are available by $\Sigma_{2}$-Collection and Full Set Induction:
at limits collect the earlier histories and take their compatible
union. Use the ground-model bounds
\[
\chi_{0}=(2^{\theta})^{+},\qquad\chi_{i}=\bigl(\beth_{\omega}(\sup_{j<i}\chi_{j})\bigr)^{+}\quad(0<i\leq\theta+3),
\]
where $\beth_{\omega}(\mu)$ is the supremum of the finite powerset
iterates starting at $\mu$. At preparation stage $i\leq\theta$,
maintain $|\mathbb{R}_{\theta}\mathbin{\upharpoonright}i|\leq\chi_{i}$
and $1\Vdash\mu_{i}\leq\check{\chi}_{i}$. The additional bounds cover
the endpoint stage and the name for the final product.

Use bounded families of names and their mixtures throughout. If $S$
supplies representatives for the elements of a named set, subnames
on $S\times(\mathbb{R}_{\theta}\mathbin{\upharpoonright}i)$ represent
its subsets, and mixtures along antichains represent its elements.
Each preparation stage uses cardinal succession, a power set, partial-function
sets and finitely many such name operations, all absorbed by the next
$\beth_{\omega}$ bound. At a limit $i$, if $\mu$ bounds the earlier
representative families and $|i|$, the full-support conditions are
bounded by $\mu^{|i|}\leq2^{\mu}$. In the prepared model, the final
product is a subset of a fixed function space. Another finite number
of powerset and name operations therefore bounds its name and the
two-step conditions. Retain all representatives in these families,
so the definition requires no choice of preferred names or well-orders.

Include the auxiliary powersets, function sets, representative families
and atomic forcing tables in the construction histories. Powerset
correctness is $\Pi_{1}$, the tables are $\Delta_{1}$, and cardinalization
and the bound sequence are $\Sigma_{2}$. With these sets supplied,
the tests for collapse and product conditions are bounded. An initial
existential block for the auxiliary sets therefore gives the same
$\Sigma_{3}$ bound on valid histories. The available $\Sigma_{3}$-Collection
and Separation assemble these histories, and Full Set Induction proves
their existence and compatibility. This constructs the preparation,
the final product name and the lottery sum as sets, uniformly from
$r,\theta$.

The preservation proof uses the same bounded resources. For fixed
input posets, ordinals and names, its $\Delta$-systems, antichains,
nice names and descending sequences lie in fixed powersets and function
sets. In Easton's lemma, the $\lambda^{+}$-chain condition bounds
each antichain by $\lambda$. Deciding a sequence of length at most
$\lambda$ uses descending constructions of length at most $\lambda$,
with lower bounds supplied by $<\lambda^{+}$-closure. Nice names
are subsets of $\lambda\times\mathbb{P}$. For a proposed cofinality
decrease, fix the ordinal domain and target, so the relevant name
families and value sets are again sets. Choice supplies well-orders
for these existence proofs, and the recursions then have bounded tests
in the supplied sets. Collecting the auxiliary witnesses uses the
same $\Sigma_{3}$ bounds. Thus the local product proof is available
over $ZFC^{\Pi_{3}}$ without unrestricted Replacement. The forcing
preliminaries of Section \ref{sec:GeoSmall} supply preservation of
this background and transfer to sufficiently generic definable extensions.

If $\theta<\zeta$ and $\zeta$ is inaccessible, strong limitness
and regularity keep every $\chi_{i}$ below $\zeta$. There are fewer
than $\zeta$ stages, each taking only a countable powerset iteration
and a successor cardinal. The same induction bounds the hereditary
sizes and ranks of the representative names and conditions below a
common bound below $\zeta$. This covers the preparation, the final
product name and their two-step iteration. There are at most $2^{\theta}<\zeta$
lottery summands, with uniform bounds. Conditions are tagged by $X\subseteq\theta$,
so the rank of $r$ is irrelevant to their rank bound. The whole forcing
therefore belongs to $V_{\zeta}$ and has size below $\zeta$.
\end{proof}
We now have the forcing and its preservation properties. We use
them to show that its generic extensions can recognize their grounds.
\begin{lem}
\label{M_P-is-def}($ZFC^{\Pi_{3}}$) For a fixed ordinal $\kappa$,
$V$ has at most one ground $M\models ZFC^{\Pi_{3}}$ such that $M$
regards $\kappa$ as an infinite regular cardinal satisfying $\kappa^{<\kappa}=\kappa$
and $V$ is a $\mathbb{G}_{\kappa}^{M}$-generic extension of $M$.
If that ground exists, it is uniformly definable from $\kappa$. In
particular, for $\kappa=\omega$ or $\omega_{2}$, it is uniformly
definable without parameters.
\end{lem}

Let us denote this ground as $M_{\mathbb{G}_{\kappa}}$, and write
$M_{\mathbb{G}}$ when $\kappa=\omega$. This should be understood
as a defined relation symbol, in the same way that we refer to $L$.
More generally, the index may be given by a parameter-free definition,
provided its value is a regular cardinal satisfying $\kappa^{<\kappa}=\kappa$
in the candidate grounds and is unchanged in their designated extensions.
Definability of the index alone does not suffice. The preserved indices
$\omega$ and $\omega_{2}$ will suffice for the applications.

To prove this, we need the following theorem of Hamkins. Its $ZFC$
form is \citet{Hamkins2003ApproxCover}, Lemma 13, with closure points
defined immediately before it in Definition 12. Our additional task
is to verify that the proof works over $ZFC^{\Pi_{3}}$.
\begin{fact}
\label{fact:Geol}($ZFC^{\Pi_{3}}$; closure-point forcing.) Suppose
$\mathbb{R}*\dot{\mathbb{Q}}$ has an atomless first factor of size
at most an infinite cardinal $\kappa$, and $1_{\mathbb{R}}$ forces
that $\dot{\mathbb{Q}}$ is $\leq\kappa$-strategically closed. Its
generic extensions have $\delta$-approximation and $\delta$-cover
over the ground, where $\delta=(\kappa^{+})^{V}$. The same conclusions
transfer to the sufficiently generic definable extensions used here.
\end{fact}

\begin{proof}
We follow \citet{Hamkins2003ApproxCover}, pp. 265--266, and assume
that the reader has that proof to hand. We recall the argument, then
check the set constructions that matter in the weaker theory. Write
the extension as $V[g][H]$.

Cover is immediate from the two factors. The tail adds no ordinal
sequences of length at most $\kappa$. For a first-stage name for
such a sequence, there are at most $|\mathbb{R}|$ possible values
at each coordinate, so their union is a ground set of size at most
$\kappa$. Ground well-orders reduce sets of ground elements to this
case. The same small-forcing calculation preserves $\delta$.

For approximation, argue by induction on the length $\tau$ of a binary
sequence $s$ whose restrictions to ground sets of size at most $\kappa$
belong to the ground. Every proper initial segment of $s$ is then
old. If $\mathsf{cf}^{V}(\tau)\leq\kappa$, closure puts $s$ in $V[g]$,
and the small-forcing argument applies. At larger cofinality, Hamkins
builds a ground tree of tail names which split the possible initial
segments of $s$. A new binary sequence of length at most $\kappa$
added by the first factor selects a branch. Strategic closure gives
a lower bound, forcing the union of the corresponding initial segments
to be old. The splitting recovers the new branch from this union,
a contradiction.

The first point to check is the use of the ground in forcing assertions.
For a fixed name of length $\tau$, let
\[
B=(^{\leq\tau}2)^{V},\qquad C=\{u\subseteq\tau:|u|^{V}\leq\kappa\}.
\]
These are sets: the size test can quantify over $\mathcal{P}(\kappa\times\tau)$
once that powerset has been supplied. In the proof, being an old binary
sequence means membership in $\check{B}$, and the approximation tests
range over $\check{C}$. No predicate for the ground, or truth predicate
for it, is needed.

Next, bound the names used in the tree. Tail conditions have representatives
among mixtures of the member-names of $\dot{\mathbb{Q}}$ along first-factor
antichains. Such mixtures are indexed by a powerset of a fixed product.
A name for a play of length at most $\kappa+1$ is a sequence of these
representatives, so the plays and the values of the named strategy
also range over sets. Include their transitive closures and the needed
check names, and form the atomic and bounded forcing tables on them.
These tables have $\Delta_{1}$ definitions.

With these sets and tables supplied, deciding an old initial segment,
comparing conditions, finding a split and following the strategy are
bounded tests. The same is true of the uniformity step that fixes
the first coordinate: its quantifiers range over the first forcing,
the chosen tail names and $\tau$. Fix well-orders of the sets of
choices. The resulting plays and tree construction have histories
in fixed function sets; bounded Separation selects valid histories,
and uniqueness makes them compatible. Taking their union gives the
required recursion. Thus no unrestricted Replacement is hidden in
the tree construction.

Finally, the auxiliary powersets have $\Pi_{1}$ correctness conditions,
and the forcing tables are $\Delta_{1}$. Collecting these witnesses
and the name families uses at most the $\Sigma_{3}$ bounds described
in Lemma \ref{lem:coding-presentation}. The available $\Pi_{3}$
schemes supply $\Sigma_{3}$-Separation and Collection, while Full
Set Induction supplies the induction on $\tau$ and the existence
of the recursion histories. This verifies the proof over $ZFC^{\Pi_{3}}$.
For the definable generics of the paper, take the finite genericity
level high enough for the fixed assertions just used and apply Lemma
\ref{lem:tightFullSetInduction} separately.
\end{proof}
The role of the first factor is now visible: it supplies the new branch
of length $<\kappa$ used in the approximation argument. The tail
is closed enough to preserve the low information used to recognize
the ground, while writing its code farther up.
\begin{proof}
(of Lemma \ref{M_P-is-def}.) Work in $ZFC^{\Pi_{3}}$. Let $W=V[G]$
be a $\mathbb{G}_{\kappa}$-extension and write the generic according
to the two-step iteration. The first factor is atomless and has size
$\kappa$, while the second is $<\theta^{+}$-directed closed, with
$\theta=(2^{\kappa^{++}})^{+}$ computed in the intermediate model.
In particular it is $\leq\kappa$-strategically closed. The full iteration
preserves $\kappa^{+}$ and $\kappa^{++}$. Fact \ref{fact:Geol}
gives approximation and cover at $\delta=\kappa^{+}$, and preservation
of $\kappa^{++}$ supplies the successor-cardinal agreement required
by Theorem \ref{thm:PI_3GMDT}. Hence $V$ is defined in $W$ from
$r_{\kappa}$ by $\vartheta(x,r_{\kappa})$.

We have defined the ground from its parameter. We now make that parameter
definable from $\kappa$. By Lemma \ref{lem:localCoding}, the tail
adds no subsets of $\kappa^{++}$ and preserves the intermediate cardinal
$\theta=(2^{\kappa^{++}})^{+}$. Thus $W$ computes the same $\theta$.
It reads the code at the cardinals $\theta^{+(\xi+1)}$, for $\xi<\theta$,
and takes its rooted collapse, recovering the original $r_{\kappa}$.

Consequently there is a fixed formula $D(r,\kappa)$ saying that $r$
is the set recovered from this continuum pattern, and in the extension
\[
x\in V\quad\longleftrightarrow\quad\exists r\bigl(D(r,\kappa)\land\vartheta(x,r)\bigr).
\]
This is the purpose of the construction. Theorem \ref{thm:PI_3GMDT}
defines $V$ from $r_{\kappa}$; the continuum pattern defines $r_{\kappa}$
from $\kappa$. For $\kappa=\omega$, the starting cardinal is the
extension’s own $(2^{\omega_{2}})^{+}$, so both the coding range
and the decoded parameter are recognizable without parameters. The
displayed formula therefore defines $V$ without parameters.

If $V_{0}$ and $V_{1}$ are two such grounds at the same index $\kappa$,
their common extension computes the same $\kappa^{++}$ and $\theta$,
reads the same code, and therefore recovers the same $r$. The formula
$\vartheta(x,r)$ defines both grounds, so $V_{0}=V_{1}$. For $\kappa=\omega$
or $\omega_{2}$, preservation of the index makes this definition
parameter-free.

The same argument works for every $\mathbb{G}_{\kappa}^{V}$-generic
extension of the chosen ground: the definition of the forcing retains
the same $r_{\kappa}$, the index is preserved, and decoding gives
that parameter back. This is the stability needed in the stable definable
ground property.
\end{proof}
We have therefore supplied the recovery part of the intended applications:
the designated forcing extension recognizes its ground, and recognizes
the same ground when the construction is repeated. Whether an additional
axiom survives the designated forcing is a separate question, to be
checked in each application.

\section{Applications}

\label{sec:Applications}

\subsection{Large cardinals beyond $L$ and other extensions}

\label{subsec:MG-single}

Finally, we can offer a representative sample of strong inductive
set theories with weak Separation and Collection that are not tight.
\begin{prop}
\label{prop:MG-single}Let $n\geq3$ and let $\Phi$ be a tautology
or any of the statements below. If $ZFC^{\Pi_{n}}+\Phi$ is consistent,
then it is not tight.
\begin{enumerate}
\item $\neg\exists\alpha\ V_{\alpha}\models ZFC^{\Pi_{n}}$;
\item There is a measurable cardinal;
\item There is a supercompact cardinal;
\item $I3$;
\item There is a proper class of extendible cardinals;
\item $PFA$;
\item $MM$.
\end{enumerate}
\end{prop}

\begin{proof}
Let $S=ZFC^{\Pi_{n}}+\Phi$. In each case we first obtain a consistent
strengthening $R\supseteq S$ and choose $\kappa\in\{\omega,\omega_{2}\}$
so that $R$ proves $\kappa^{<\kappa}=\kappa$ and forcing with $\mathbb{G}_{\kappa}$
preserves $R$. Section \ref{sec:GeoSmall} already supplies the uniform
recovery of the ground. Our task is to choose an index at which the
additional axiom survives the coding forcing, and then verify the
forcing assertions required by Definition \ref{def:StableDefGprop}.
We carry out these two checks in that order.

For the tautology case, take $R=S$ and $\kappa=\omega$. The assertion
$\omega^{<\omega}=\omega$ is automatic, and the set-forcing preservation
recalled at the beginning of Section \ref{sec:GeoSmall} gives preservation
of $R$ by $\mathbb{G}_{\omega}$.

For the anti-rank assertion, let $S_{0}=ZFC^{\Pi_{n}}$. Consistency
of $S$ implies consistency of $S_{0}$, so the second incompleteness
theorem gives consistency of $R=S_{0}+\neg\mathsf{Con}(S_{0})$. This
theory implies the displayed anti-rank assertion, since a set model
of $S_{0}$ supplies the usual soundness proof. Its additional arithmetic
sentence is preserved by set forcing, so take $\kappa=\omega$.

For the large-cardinal assertions, take $R=S$ and $\kappa=\omega$.
We first check that the whole forcing $\mathbb{G}_{\omega}$ is small
below an inaccessible $\lambda$. The first Cohen factor preserves
the inaccessibility of $\lambda$. In the intermediate extension,
the retained ground parameter $r_{\omega}$ has hereditary size below
$\lambda$ and $\theta=(2^{\omega_{2}})^{+}<\lambda$. Lemma \ref{lem:coding-presentation}
therefore bounds the coding tail below $\lambda$. A maximal antichain
in the countable first factor chooses bounds and names for enumerations
of that tail. Regularity of $\lambda$ gives a common bound below
$\lambda$, and the corresponding pairs of conditions and named enumeration
entries form a dense subset of the whole iteration of size below $\lambda$.
We may replace this dense subset by an isomorphic poset in $V_{\lambda}$
when applying the small-forcing arguments.

The preservation of measurability and supercompactness can be checked
using measures, so the weak background need not construct a proper-class
ultrapower. Let $U$ be a normal $\lambda$-complete ground measure
on a set $D$, and let $B$ be the Boolean completion of the small
forcing. Since $\lambda$ is strong limit, $|B|<\lambda$. For a name
$\dot{X}$ for a subset of $D$, the Boolean values $\|\check{d}\in\dot{X}\|_{B}$,
for $d\in D$, partition $D$ into fewer than $\lambda$ ground sets.
One cell belongs to $U$. According to whether its Boolean value belongs
to the generic, that cell lies in $X$ or in its complement. Thus
the filter generated by $U$ in the extension is an ultrafilter.

This filter is still $\lambda$-complete: for fewer than $\lambda$
named members, choose maximal antichains deciding ground measure-one
subsets of them and intersect the fewer than $\lambda$ ground sets
so obtained. Normality is preserved by deciding a regressive function
on a ground measure-one set; below any condition, one stronger condition
decides its values on a measure-one subset, where ground normality
makes it constant. These constructions use atomic forcing values and
sets of names, partitions and antichains available in $ZFC^{\Pi_{3}}$.
For measurability take $D=\lambda$. For supercompactness take $D=(P_{\lambda}(\eta))^{V}$
for each $\eta\geq\lambda$, and extend the lifted measure to the
new $P_{\lambda}(\eta)$ by concentrating on $D$. It remains normal
and fine, as required.

The rank-embedding cases use the usual small-forcing lifting argument.
For an $I3$ witness $j:V_{\rho}\to V_{\rho}$ with critical point
$\lambda$, the equivalent forcing in $V_{\lambda}$ is fixed by $j$,
and the lift is given by $j^{*}(\tau_{G})=j(\tau)_{G}$. The rank-name
calculation identifies its domain and target with $V_{\rho}^{V[G]}$.
For a proper class of extendible cardinals, fix the forcing and take
an extendible $\lambda$ above its size. Given a target height $\alpha>\lambda$,
choose an inaccessible $\rho>\alpha$ and an extendibility witness
$j:V_{\rho}\to V_{\sigma}$ with $j(\lambda)>\rho$. Lift this set
embedding and restrict it to $V_{\alpha}^{V[G]}$. Thus every sufficiently
large extendible remains extendible. The assumed proper class supplies
the inaccessible heights used here. Satisfaction for these set structures
and the bounded name and valuation constructions are available in
the weak background. Together with the set-forcing preliminaries of
Section \ref{sec:GeoSmall}, this proves preservation of $R$.

For $PFA$ and $MM$, use their usual formulations \citep{Baumgartner1984PFA,ForemanMagidorShelah1988MM}
and take $\kappa=\omega_{2}$. Both axioms are preserved by $<\omega_{2}$-directed
closed set forcing.\footnote{For $MM$, see \citep[Theorem 4.3]{Larson2000Reflection}. For $PFA$,
compare the stronger closed-forcing preservation result of \citep[Theorem 6.1]{KoenigYoshinobu2004}.
The directed-closed form for both axioms is treated in \citep[Section 4.1.1]{Cox2021ForcingAxioms}.} To recall the argument over the weak background, let $\mathbb{P}$
be such a forcing and suppose $p\in\mathbb{P}$ forces that $\dot{\mathbb{Q}}$
is proper, respectively stationary-set preserving, and that $\langle\dot{D}_{\xi}:\xi<\omega_{1}\rangle$
is a family of dense subsets. The two-step iteration $\mathbb{T}=(\mathbb{P}\mathord{\downarrow}p)*\dot{\mathbb{Q}}$
belongs to the corresponding forcing class. Apply the ground forcing
axiom to the dense sets $E_{\xi}=\{\langle q,\dot{s}\rangle\in\mathbb{T}:q\Vdash\dot{s}\in\dot{D}_{\xi}\}$.

From a filter $F$ meeting all the $E_{\xi}$, choose one condition
for each requirement and close the chosen family under common extensions
in $F$ through $\omega$ stages. This gives a directed set $A\subseteq F$
of size at most $\omega_{1}$. A lower bound $p^{*}\leq p$ of its
first coordinates forces that its second coordinates are directed
and generate a filter meeting every $\dot{D}_{\xi}$. Since the construction
works below every $p$, it proves the required forcing assertion.
The closure is taken on iteration conditions, so it controls both
coordinates.

The weak-background requirements are local. The two-step properness
proof uses countable elementary submodels and their interpreted extensions;
the stationary-preservation proof uses names for clubs in $\omega_{1}$.
Satisfaction for the relevant set structures and representative families
of names are sets. Once they are fixed, the choice and closure constructions
above take place in fixed powersets. The necessary Collection, Separation
and set recursions are available in $ZFC^{\Pi_{3}}$.

We may arrange $\omega_{2}^{<\omega_{2}}=\omega_{2}$ without strengthening
the consistency assumption. In a model of $S$, let $\tau=|{}^{<\omega_{2}}2|$
and force a surjection from $\omega_{2}$ onto $\tau$ using partial
functions of size below $\omega_{2}$. This forcing is $<\omega_{2}$-directed
closed, so it preserves the forcing axiom and $\omega_{2}$, adds
no sequences of length below $\omega_{2}$, and makes $|^{<\omega_{2}}2|=\omega_{2}$.
Equivalently, $(\omega_{2})^{<\omega_{2}}=\omega_{2}$. Thus $R=S+\{\omega_{2}^{<\omega_{2}}=\omega_{2}\}$
is consistent. The forcing $\mathbb{G}_{\omega_{2}}$ preserves $R$,
by its directed closure and the preceding preservation argument. Here
preservation of $\omega_{2}^{<\omega_{2}}=\omega_{2}$ also uses preservation
of $\omega_{2}$: the forcing adds no sequences of length below that
cardinal, so the old bounds on the short function spaces remain valid.

In each case let $\mathsf{N}=M_{\mathbb{G}_{\kappa}}$ and let $S^{*}$
extend $R$ by asserting that $\mathsf{N}$ is a $\mathbb{G}_{\kappa}^{\mathsf{N}}$-ground
of the universe and satisfies $R$. The ground clause uses the $MOST$-ground
family of Section \ref{sec:GeoSmall}; the stronger axioms of $R$
are imposed separately on $\mathsf{N}$ by relativization. This is
still a theory of the form $KPC+U^{*}$ with $U^{*}$ quantifier bounded.
Indeed, $n$ and the definitions of the index and the ground are fixed.
For the non-induction axioms, relativization therefore retains a common
complexity bound. Full Set Induction in $\mathsf{N}$ follows from
ambient Full Set Induction, applied to each definable subclass of
$\mathsf{N}$, and need not be included among the extra axioms.

To witness consistency, use Lemma \ref{count-enum-thm} to choose
a model of $R$ with a definable countable enumeration, and take a
definable $\mathbb{G}_{\kappa}$-generic at a finite level covering
the non-induction axioms and the recovery assertions. Lemma \ref{lem:O2GeneralTruth}
transfers those forced assertions, Lemma \ref{lem:tightFullSetInduction}
supplies Full Set Induction, and Lemma \ref{M_P-is-def} recovers
the ground. The resulting extension satisfies $S^{*}$. The same method
justifies the preliminary consistency step for $R=S+\{\omega_{2}^{<\omega_{2}}=\omega_{2}\}$
in the forcing-axiom cases.

We now verify the four clauses of Definition \ref{def:StableDefGprop}.
The induction instances require one additional observation. Set forcing
forces each instance of Full Set Induction over a ground satisfying
Full Set Induction. Fix a formula $\varphi$, with names for its parameters,
and work below a condition $p$ forcing $\exists x\,\varphi(x)$.
Use ground Full Set Induction to choose a name $\tau$ of least rank
for which some $q\leq p$ forces $\varphi(\tau)$. If a stronger condition
forced $\varphi$ of a member of $\tau$, the atomic membership clause
would give a further extension equating that member with a member-name
of $\tau$. Substitution would then contradict the minimality of its
rank. Hence $q$ forces that $\tau$ is a minimal member of the class
defined by $\varphi$. The density of such conditions proves the induction
instance.

In a model of $S^{*}$, let $\mathbb{Q}=\mathbb{G}_{\kappa}^{\mathsf{N}}$.
Then $\mathbb{Q}\in\mathsf{N}$ is definable there and nontrivial,
and the defining ground assertion in $S^{*}$ supplies $V=\mathsf{N}[g]$.
These give the first and third clauses. The forced form of Lemma \ref{M_P-is-def}
and preservation of the index identify the recovered ground with the
original universe. In particular, in a designated extension we have
$(M_{\mathbb{G}_{\kappa}})^{\mathsf{N}[h]}=\mathsf{N}$. This gives
the fourth clause.

For the second clause, argue internally in $\mathsf{N}$. The preceding
preservation arguments force the non-induction axioms of $R$, and
the minimal-name argument forces each induction instance. Since the
recovered ground is the original universe, each axiom of $R$ also
holds when relativized to that ground. The canonical generic name
supplies the assertion that the extension is obtained from it by $\mathbb{Q}$.
Thus $\mathsf{N}\models1_{\mathbb{Q}}\Vdash S^{*}$, instance by instance.

Finally, fix a common standard complexity bound $j$ for the non-induction
axioms and the fixed recovery and stability assertions, after expanding
the definitions. Choose $k\geq\mathsf{b}(j)$, enlarged for the product-forcing
argument in Theorem \ref{thm:MainTemplate}. Lemma \ref{lem:O2GeneralTruth}
transfers these assertions to the designated definable generics. Their
Full Set Induction comes from Lemma \ref{lem:tightFullSetInduction};
it is not being transferred at a single finite complexity bound. We
have therefore verified all four clauses of the stable definable ground
property. Theorem \ref{thm:MainTemplate} applies to $S^{*}$, and
its distinct bi-interpretable extensions, being extensions of $S$,
also witness that $S$ is not tight.
\end{proof}
\begin{rem*}
Note that the tautology case already follows from Theorem \ref{thm:generic-invariance-not-tight}
and the set-forcing preservation recalled at the beginning of Section
\ref{sec:GeoSmall}.
\end{rem*}
The preceding proposition concerned restrictions on the Separation
and Collection schemes. We now consider the consequences of a theory
at a fixed quantifier complexity. This is a different restriction
since proofs of sentences of low complexity may use axioms of arbitrarily
high complexity.\footnote{For example, $\{Con^{ZFC^{\Pi_{n}}}\}_{n\geq1}$ is a set of $\Pi_{1}^{0}$
sentences, which is, thus, of bounded quantifier complexity. However,
assuming $ZFC$ is consistent, there is no $n\in\omega$ such that
$ZFC^{\Pi_{n}}$ proves every element of that set.} The following argument extends Enayat's observation about the bounded
consequences of $ZF$ and $ZFC$ to stronger theories, including theories
incompatible with $V=L$.

We use the forcings $\mathbb{G}_{\kappa}$ and the assertions $\kappa^{<\kappa}=\kappa$
from Section \ref{sec:GeoSmall}. In the proposition, preservation
of $T$ means that $T\vdash1_{\mathbb{G}_{\kappa}}\Vdash\psi$ for
every axiom $\psi$ of $T$, and hence for every theorem of $T$.
\begin{prop}
Let $\kappa$ be $\omega$ or $\omega_{2}$, and let $T$ be a consistent
extension of $ZFC+\{\kappa^{<\kappa}=\kappa\}$ which is preserved
by $\mathbb{G}_{\kappa}$. If $U$ is a quantifier bounded collection
of theorems of $T$, then $KP+U$ is not tight. Indeed, its distinct
bi-interpretable extensions may be chosen to extend $ZFC^{\Pi_{3}}+U$.\label{prop:bounded-consequences}
\end{prop}

Only the chosen consequences and their forcing translations need a
common complexity bound; the theory $T$ itself need not have one.
We shall use $T$ to witness the consistency of a suitable theory
of grounds, apply the Definable Countable Enumeration Lemma there,
and then invoke our tightness failure template.
\begin{proof}
Let
\[
R=ZFC^{\Pi_{3}}+\{\kappa^{<\kappa}=\kappa\}+\{1_{\mathbb{G}_{\kappa}}\Vdash\varphi:\varphi\in U\}.
\]
Every axiom of $R$ is a theorem of $T$, so $R$ is consistent. By
Lemma \ref{lem:O2ForcingBound}, the forcing translations of the sentences
in $U$ have a common finite complexity bound. After expanding the
fixed definitions, $R$ is therefore $KPC$ plus a quantifier bounded
set of sentences. Full Set Induction is already included in $KPC$.

Write $\mathsf{N}=M_{\mathbb{G}_{\kappa}}$ for the ground defined
in Lemma \ref{M_P-is-def}. Let $S$ consist of $ZFC^{\Pi_{3}}+U$,
the assertion that $\mathsf{N}$ is a $\mathbb{G}_{\kappa}^{\mathsf{N}}$-ground
of the universe, and the relativizations to $\mathsf{N}$ of the non-induction
axioms of $R$. The ground inherits Full Set Induction from the ambient
universe, so $S$ proves that $\mathsf{N}$ satisfies all of $R$.
Relativization to this fixed definition retains a common complexity
bound; thus $S$ is also $KPC$ plus a quantifier bounded set of sentences.

To see that $S$ is consistent, apply Lemma \ref{count-enum-thm}
to obtain a model $W$ of $R$ with a definable countable enumeration.
Use that enumeration to define a $\mathbb{G}_{\kappa}^{W}$-generic
filter at a finite level sufficient for $U$, the non-induction axioms
of $ZFC^{\Pi_{3}}$, and the fixed ground-recovery assertions. The
forcing assertions in $R$ and Lemma \ref{lem:O2GeneralTruth} give
$U$ in the extension. The preservation results of Section \ref{sec:GeoSmall}
give the non-induction axioms of $ZFC^{\Pi_{3}}$, and Lemma \ref{lem:tightFullSetInduction}
gives Full Set Induction. Lemma \ref{M_P-is-def} recovers $W$ as
the designated ground. The extension therefore satisfies $S$.

Finally, $S$ has the stable definable ground property. Its ground
$\mathsf{N}$ satisfies $R$, and hence forces every member of $U$.
The forcing preserves $ZFC^{\Pi_{3}}$ and recovers the original universe
as its designated ground, which still satisfies $R$. The canonical
generic name supplies the assertion that the extension is obtained
by the designated forcing. Thus $\mathsf{N}$ forces every axiom of
$S$, including each induction instance, as in the proof of Proposition
\ref{prop:MG-single}. Nontriviality and stability are supplied by
the same forcing and recovery results. Theorem \ref{thm:MainTemplate}
now makes $S$ non-tight. Since $S$ extends $ZFC^{\Pi_{3}}+U$, its
two distinct bi-interpretable extensions also witness the required
conclusion.
\end{proof}
For a syntactic class $\Gamma$, write
\[
T_{\Gamma}=\{\varphi\in\Gamma:T\vdash\varphi\}.
\]
The subscript thus denotes consequences of $T$, while the superscript
in $ZFC^{\Pi_{n}}$ restricts axiom schemes.
\begin{cor}
Let $\Phi$ be a tautology or one of the following assertions: there
is a measurable cardinal; there is a supercompact cardinal; $I3$;
there is a proper class of extendible cardinals; $PFA$; or $MM$.
Suppose $ZFC+\Phi$ is consistent. Then, for every standard $n\in\omega$,
the theories
\[
KP+(ZFC+\Phi)_{\Sigma_{n}}\qquad\text{and}\qquad KP+(ZFC+\Phi)_{\Pi_{n}}
\]
are neither tight nor solid. The same holds with $B\Sigma_{n}$ in
place of $\Sigma_{n}$ or $\Pi_{n}$, and with Full Set Induction
in place of $KP$.\label{cor:consequence-theories}
\end{cor}

\begin{proof}
For the tautology and the large-cardinal cases, take $\kappa=\omega$
and $T=ZFC+\Phi$. The preservation arguments of Proposition \ref{prop:MG-single}
apply. For $PFA$ and $MM$, take $\kappa=\omega_{2}$ and $T=ZFC+\Phi+\{\omega_{2}^{<\omega_{2}}=\omega_{2}\}$.
The preliminary closed forcing used there shows that $T$ is consistent,
and $\mathbb{G}_{\omega_{2}}$ preserves it. In each case the chosen
consequences of $ZFC+\Phi$ are a quantifier bounded collection of
theorems of $T$. Apply Proposition \ref{prop:bounded-consequences}.
Non-tightness passes to subtheories, giving the Full Set Induction
versions, and solidity implies tightness.
\end{proof}
In particular, if $ZF$ is consistent, the result applies to the $\Sigma_{n}$-,
$\Pi_{n}$- and $B\Sigma_{n}$-consequences of either $ZF$ or $ZFC$,
together with Full Set Induction. Here we use the equiconsistency
of $ZF$ and $ZFC$ and the inclusion of the $ZF$-theorems among
the $ZFC$-theorems. This case can also be obtained using the constructible-enumeration
argument of Section \ref{sec:KP}. The general enumeration lemma supplies
the stronger applications above.

\subsection{The inexpressibility of $V=HOD$ in quantifier bounded set theories}

In the face of the Definable Countable Enumeration Lemma, a natural
worry emerges. What if we added the axiom $V\neq HOD$ to one of our
suitable bounded theories? Then the obvious thing to think is that
models of such a theory cannot have a definable well-ordering, let
alone a definable countable enumeration. The problem is that the sentence
we use to express that the universe is hereditarily ordinal definable
over $ZF$ doesn't do the same job over theories with less Separation
and Collection.

To avoid confusion, let us write $\chi_{\mathrm{HOD}}$ to denote
the usual first-order sentence of $\mathcal{L}_{\in}$ which expresses
that every set is definable in some $V_{\alpha}$ by a formula of
set theory only using ordinal parameters. Thus, $\chi_{HOD}$ is the
sentence we use to express $V=HOD$ in $ZF$ and its extensions. In
the context of $ZF$, $\chi_{HOD}$ is equivalent to saying that
there is a definable well-ordering of the universe. But over the weak
set theories, we've been considering, $\chi_{HOD}$ merely implies
this: the converse does not hold. In particular, the Definable Countable
Enumeration Lemma implies the following. 
\begin{prop}
Let $T=KPC+S$, where $S$ is quantifier bounded. If $T+\neg\chi_{\mathrm{HOD}}$
is consistent, then it has a model $\mathcal{M}$ admitting a definable
well-ordering.
\end{prop}

\begin{proof}
The set of additional axioms $S\cup\{\neg\chi_{\mathrm{HOD}}\}$ is
quantifier bounded, since $\chi_{\mathrm{HOD}}$ is a single fixed
first-order sentence. We then apply the Definable Countable Enumeration
Lemma (Lemma \ref{count-enum-thm}) to $KPC+S+\neg\chi_{\mathrm{HOD}}$
to obtain a model $\mathcal{M}$ of $T+\neg\chi_{\mathrm{HOD}}$ along
with a parameter-free formula defining a class surjection $f:\omega^{\mathcal{M}}\twoheadrightarrow M$.
We can then well-order $\mathcal{M}$ by letting $x\prec y$ iff,
in $\mathcal{M}$, the least $n\in\omega$ such that $f(n)=x$ is
such that $n<m$ where $m\in\omega$ is the least such that $f(m)=y$.
\end{proof}
Thus, we see there are models of such theories where $\chi_{HOD}$
is false and yet we still have a definable well-ordering of the universe.
In a nutshell, without Full Reflection, the falsity of $\chi_{HOD}$
is not enough to block the existence of a complex definable well-ordering.

\subsection{Solidity and the ground axiom}

\label{subsec:solid-GA}

For the final section of the paper, we return to solidity and observe
that the sensible thing is true: solid theories cannot describe generic
extensions.

Over a weak theory, however, we must say how generic the extension
is. An ordinary generic need not give the truth lemma for every sentence
in $U$. Since $U$ is quantifier bounded, only a fixed finite level
of genericity is needed. We incorporate that requirement into the
ground axiom.

\label{def:GA-n}Let $MOST_{\mathrm{Ind}}=KPC+\text{Powerset}+\Sigma_{1}$-Separation,
as in Section \ref{sec:GeoSmall}, and let $W_{t}$ be Farmer's uniformly
defined family of $MOST$-grounds. For each fixed standard $n\geq1$,
define $\neg\mathrm{GA}_{n}$ to say that there are $t$, $\mathbb{P}\in W_{t}$
and a filter $G\subseteq\mathbb{P}$ such that 
\[
W_{t}\subsetneq V,\qquad V=W_{t}[G],\qquad G\text{ is }B\utilde{\Sigma}_{\mathsf{b}(n)}\text{-generic over }W_{t}.
\]
 Thus we retain the existing ground clause and strengthen its genericity
requirement. The same $\mathbb{P}$ and $G$ must witness both the
extension and the additional genericity.

This is a first-order sentence. For fixed $n$, use the partial satisfaction
predicate from Section \ref{sec:KP}, relativized to $W_{t}$, to
quantify over formula codes and finite tuples of ground parameters.
The additional clause says that $G$ meets each dense subclass of
$\mathbb{P}$ so defined. We do not add Full Set Induction to the
ground clause: in a model of $MOST_{\mathrm{Ind}}$, every definable
transitive class inherits Full Set Induction by relativization. In
particular, each $MOST$-ground occurring here satisfies $MOST_{\mathrm{Ind}}$.
\begin{thm}
\label{thm:solid-GA}Suppose $T=MOST_{\mathrm{Ind}}+U$, where $U\subseteq B\Sigma_{n}$
for some standard $n\geq1$. If $T$ is solid, then 
\[
T\vdash\mathrm{GA}_{n}.
\]
\end{thm}

We prove the theorem contrapositively. A model of $T+\neg\mathrm{GA}_{n}$
gives a ground in which each finite part of $U$ can be forced. Compactness
and definable countable enumeration then give a ground over which
two canonical generic extensions can be defined. Both extensions satisfy
$T$ and can interpret one another. The first defines an isomorphism
with the round-trip interpretation of itself, but cannot define an
isomorphism to the intervening extension. This is the failure of solidity
we need.

We shall use the following three results: the set-forcing preliminaries
in Section \ref{sec:KP}, together with Farmer's preservation lemma
\citep[Lemma 2.5]{Schlutzenberg2022LocalMantles}, give the required
forcing and preservation results over $MOST_{\mathrm{Ind}}$; a proper
ground is uniformly definable from Farmer's parameter, as in Theorem
\ref{thm:SchlutzenbergGMD}; and the Definable Countable Enumeration
Lemma (Lemma \ref{count-enum-thm}) applies to any consistent theory
consisting of $KPC$ and a quantifier bounded set of additional sentences,
including a fixed finite expansion by constants. Solidity allows parameters,
so the small-parameter coding of Section \ref{sec:GeoSmall} is not
needed here.

Fix $T$, $U$ and $n$ as in the theorem. We first obtain a consistent
theory of grounds. Call a forcing atomless if below every condition
there are two incompatible conditions. This will ensure that the generic
extensions used later are proper.
\begin{lem}
\label{lem:nontrivial-forcing}Suppose $T+\neg\mathrm{GA}_{n}$ is
consistent. In the language obtained by adding a constant $\dot{\mathbb{P}}$,
let $\Gamma_{U}$ consist of $MOST_{\mathrm{Ind}}$, the assertion
that $\dot{\mathbb{P}}$ is an atomless forcing notion with a greatest
condition, and, for each $\varphi\in U$, the sentence $1_{\dot{\mathbb{P}}}\Vdash\varphi$.
Then $\Gamma_{U}$ is consistent.
\end{lem}

\begin{proof}
Fix a model $V\models T+\neg\mathrm{GA}_{n}$ and write $V=W[G]$,
where $W=W_{t}$ is a proper $MOST$-ground and $G\subseteq\mathbb{P}\in W$
is $B\utilde{\Sigma}_{\mathsf{b}(n)}$-generic over $W$. As observed
above, $W\models MOST_{\mathrm{Ind}}$. It is enough to satisfy a
finite fragment of $\Gamma_{U}$. Let $F\subseteq U$ be finite. Since
$\bigwedge F$ is $B\Sigma_{n}$ and $W[G]\models\bigwedge F$, the
truth lemma for arbitrary models (Lemma \ref{lem:O2GeneralTruth})
gives some $p\in G$ such that $p\Vdash_{W}\bigwedge F$.

Call $a\in\mathbb{P}$ an atom if every two conditions below $a$
are compatible. In $W$ let
\[
D=\{a\in\mathbb{P}:a\text{ is an atom}\}\cup\{q\in\mathbb{P}:\text{no atom lies below }q\}.
\]
This set is dense in $\mathbb{P}$. Since the extension $W[G]$ is
proper, $G$ contains no atom: a generic containing an atom is determined
by that atom and hence belongs to $W$. Choose $d\in G\cap D$. Then
no atom lies below $d$. Since $p,d\in G$, choose $q\in G$ with
$q\leq p,d$. No atom lies below $q$, so the cone $\mathbb{Q}=\mathbb{P}\restriction q$
is atomless, with greatest condition $q$. Moreover, $q\leq p$, and
hence
\[
1_{\mathbb{Q}}\Vdash_{W}\bigwedge F.
\]
Thus $W$ expanded by $\mathbb{Q}$ satisfies the chosen finite fragment.
Compactness gives the result.
\end{proof}
Since $U\subseteq B\Sigma_{n}$, the uniform forcing-bound lemma (Lemma
\ref{lem:O2ForcingBound}) gives a common $\Delta_{\mathsf{b}(n)}$
bound for the sentences $1_{\dot{\mathbb{P}}}\Vdash\varphi$, where
$\varphi\in U$. The additional axioms in $MOST_{\mathrm{Ind}}$ and
the assertion that $\dot{\mathbb{P}}$ is atomless also have fixed
finite complexity. Hence $\Gamma_{U}$ is $KPC$ plus a quantifier
bounded set of sentences in the expanded language, and the Definable
Countable Enumeration Lemma (Lemma \ref{count-enum-thm}) applies
to it.

For the next lemma, fix a model $N$ of $MOST_{\mathrm{Ind}}$ carrying
a definable class surjection $e:\omega\twoheadrightarrow N$, and
let $\mathbb{P}\in N$ be atomless. Let $m=\mathsf{b}(n)=2n+3$, and
let $\langle G_{0},G_{1}\rangle$ be the canonical $B\utilde{\Sigma}_{m}$-generic
pair for $\mathbb{P}\times\mathbb{P}$ over $N$ obtained from $e$
by the least-extension construction. Write 
\[
W=N[G_{0}\times G_{1}].
\]
 We want each coordinate to recover the same ground and enumeration,
and hence the same canonical pair. The common product extension will
then let us compare the two coordinates by Set Induction.

The construction is the one from Section \ref{sec:KP}, with least
$e$-index in place of the constructible well-order. Enumerate formula-parameter
pairs using $e$ and, at each stage, take the extension with least
$e$-index meeting the current requirement, if there is one. Full
Set Induction gives every internally finite initial history. Thus
the resulting filters are definable classes of $N$; they need not
be sets of $N$. The fixed formula defining $e$, its parameters,
$\mathbb{P}$ and $m$ determine the construction.

Farmer's preservation lemma gives $MOST$ in both coordinate extensions
and in $W$. For its application, use Choice to take fixed isomorphic
copies of the forcing notions on ordinals. Since the filters are definable
over $N$ and $m\geq1$, Lemma \ref{lem:tightFullSetInduction} gives
Full Set Induction in all three extensions. They therefore satisfy
$MOST_{\mathrm{Ind}}$.

We also need each coordinate to be generic over the other for dense
sets. Suppose $\dot{D}\in N$ names a dense subset of $\mathbb{P}$
in $N[G_{0}]$. By the bounded case of Lemma \ref{lem:O2GeneralTruth},
some $p\in G_{0}$ forces that $\dot{D}$ is dense. Below $\langle p,1\rangle$,
the class 
\[
E=\{\langle r,s\rangle\in\mathbb{P}\times\mathbb{P}:r\leq p\text{ and }r\Vdash\check{s}\in\dot{D}\}
\]
 is dense: given a second-coordinate condition, density of $\dot{D}$
lets us strengthen the first coordinate to force some ground condition
below it into $\dot{D}$. Atomic forcing and $\Sigma_{1}$-Separation
make $E$ a set of $N$. Adjoin the conditions incompatible with $\langle p,1\rangle$
to obtain a globally dense set. The product generic meets it below
$\langle p,1\rangle$, so $G_{1}$ meets $D$. The other direction
is symmetric.

Choose a cardinal $\eta$ in $N$ large enough that the fixed ordinal
copies of $\mathbb{P}$ and $\mathbb{P}\times\mathbb{P}$ belong to
\[
t=(H_{\eta^{+}})^{N}.
\]
 Farmer's ground definition recovers $N$ from this same $t$ in both
coordinates and in $W$. Relativizing the enumeration formula to the
recovered ground recovers the same $e$. We allow $t$, the forcing
notions, the fixed isomorphisms and any parameters of $e$ in the
interpretations below. All are ground objects. By the canonical interpretation
of the other coordinate, we mean the quotient of $N$-names by its
recovered generic, with equality and membership defined by forcing.
\begin{lem}
\label{lem:noiso}No isomorphism from $N[G_{0}]$ to its canonical
interpretation of $N[G_{1}]$ is parametrically definable over $N[G_{0}]$.
\end{lem}

\begin{proof}
Work in $W=N[G_{0}\times G_{1}]$. The common ground $N$ is definable
in $W$ from the parameter $t$. Valuation of $N$-names using $G_{0}$
and $G_{1}$ therefore defines the two transitive coordinate classes
in $W$. If the first coordinate defined an isomorphism to its canonical
interpretation of the second, relativizing that definition to the
first coordinate and applying valuation in $W$ would give a parametrically
definable isomorphism $h:N[G_{0}]\longrightarrow N[G_{1}]$ in $W$.

Full Set Induction in $W$ shows that $h(x)=x$ for every $x\in N[G_{0}]$.
Indeed, suppose all members of $x$ are fixed. Membership preservation
gives $x\subseteq h(x)$; conversely, transitivity of the range and
surjectivity of $h$ show that every member of $h(x)$ is $h(y)$
for some $y\in x$, and hence belongs to $x$. Extensionality gives
$h(x)=x$. The two coordinate classes would consequently be equal.

But $G_{1}\notin N[G_{0}]$: otherwise $\mathbb{P}\setminus G_{1}$
would, by atomlessness and bounded Separation, be a dense set in $N[G_{0}]$
disjoint from $G_{1}$, contrary to the mutual genericity just proved.
This is a contradiction.
\end{proof}
We can now finish the argument. The common ground and canonical pair
give the interpretations in both directions; valuation of names identifies
the round trip with the original model.
\begin{proof}
We now prove Theorem \ref{thm:solid-GA} contrapositively. Assume
$T+\neg\mathrm{GA}_{n}$ is consistent. By Lemma \ref{lem:nontrivial-forcing}
and the Definable Countable Enumeration Lemma (Lemma \ref{count-enum-thm}),
fix a model $N\models\Gamma_{U}$ carrying a definable class surjection
$e:\omega\twoheadrightarrow N$. Let $\mathbb{P}$ denote the distinguished
atomless forcing notion of $N$.

Apply the preceding construction to $\mathbb{P}\times\mathbb{P}$
at the $B\utilde{\Sigma}_{\mathsf{b}(n)}$ level, and let 
\[
\mathcal{M}_{0}=N[G_{0}]\qquad\text{and}\qquad\mathcal{N}=N[G_{1}].
\]
 Both coordinates satisfy $MOST_{\mathrm{Ind}}$, as verified above.
Moreover, $1_{\mathbb{P}}\Vdash\varphi$ for every $\varphi\in U$,
and each projection is $B\utilde{\Sigma}_{\mathsf{b}(n)}$-generic
over $N$. The truth lemma for arbitrary models (Lemma \ref{lem:O2GeneralTruth})
therefore transfers every sentence of $U$ to both coordinate extensions.
Thus $\mathcal{M}_{0},\mathcal{N}\models T$. We do not require the
product extension $W$ to satisfy $U$.

Choose the common ground parameter $t$ as above. Thus, using $t$
and $\mathbb{P}$, together with the other fixed ground parameters,
$\mathcal{M}_{0}$ recovers $N$, its definable enumeration $e$ and
the canonical pair $\langle G_{0},G_{1}\rangle$. Equality and membership
in the interpreted extension are defined by forcing on ground names.
Hence it interprets $\mathcal{N}$. Symmetrically, $\mathcal{N}$
interprets a copy $\mathcal{M}_{1}$ of the first coordinate extension.

The round trip reconstructs the same canonical first coordinate, so
the usual valuation map on $N$-names gives a parametrically definable
isomorphism
\[
\mathcal{M}_{0}\cong\mathcal{M}_{1}.
\]
Inside $\mathcal{M}_{0}$, the filter $G_{0}$ is a set and valuation
of ground names is definable. Every object is the value of such a
name, and equality modulo $G_{0}$ makes the valuation independent
of the representative. We therefore have $\mathcal{M}_{0}\triangleright_{par}\mathcal{N}\triangleright_{par}\mathcal{M}_{1}$
and a parametrically definable isomorphism $\mathcal{M}_{0}\cong\mathcal{M}_{1}$.
By Lemma \ref{lem:noiso}, there is no parametrically definable isomorphism
from $\mathcal{M}_{0}$ to this interpretation of $\mathcal{N}$.
Thus $T$ is not solid.

Hence $\mathsf{Con}(T+\neg\mathrm{GA}_{n})$ implies that $T$ is
not solid. If $T$ is solid, $T+\neg\mathrm{GA}_{n}$ is therefore
inconsistent, and so $T\vdash\mathrm{GA}_{n}$.
\end{proof}
Every quantifier bounded $U$ is contained in $B\Sigma_{n}$ for some
standard $n\geq1$, so the theorem applies to every theory of the
stated form. If $r\geq n$, then $\mathrm{GA}_{n}$ implies $\mathrm{GA}_{r}$:
stronger genericity gives fewer possible witnesses to failure of the
ground axiom. The same argument also yields the ordinary Ground Axiom
over full $ZFC$.
\begin{cor}
\label{cor:solid-ordinary-GA}Suppose $U$ is quantifier bounded.
If $MOST_{\mathrm{Ind}}+U$ is solid, then 
\[
ZFC+U\vdash\mathrm{GA},
\]
 where $\mathrm{GA}$ is the usual Ground Axiom for $ZFC$-grounds.
\end{cor}

\begin{proof}
Choose a standard $n\geq1$ with $U\subseteq B\Sigma_{n}$. By Theorem
\ref{thm:solid-GA}, $MOST_{\mathrm{Ind}}+U\vdash\mathrm{GA}_{n}$.
But an ordinary generic over a $ZFC$-ground meets every definable
dense subclass of its forcing, since Separation makes each such class
a set. Thus a proper $ZFC$-ground witnessing $\neg\mathrm{GA}$ would
also witness $\neg\mathrm{GA}_{n}$.
\end{proof}
\textbf{AI assistance.} The main ideas of this paper were developed
by the authors before any use of AI. We subsequently used OpenAI's
Sol and Astra to assist with the mathematical development and preparation
of this paper, including drafting proofs, identifying errors and suggesting
repairs, locating references, and revising the exposition. Their contributions
were most significant in the ground model definability, coding forcing,
and Ground Axiom arguments.

\end{document}